\documentclass[12pt,twoside,reqno]{amsart}

\usepackage[OT1]{fontenc}
\usepackage{type1cm}

\usepackage{comment}
\usepackage{enumerate}

\usepackage{amsthm}
\RequirePackage{amsmath,amsfonts,amssymb,amsthm}

\RequirePackage[dvips]{graphicx}

\usepackage{psfrag}
\usepackage[usenames]{color}
  \numberwithin{equation}{section}

\usepackage[active]{srcltx}  

  \newcommand{\N}{\mathbb{N}}         
  \newcommand{\R}{\mathbb{R}}         
  \newcommand{\Z}{\mathbb{Z}}         

  \newcommand{\EE}{\mathbb{E}}        
  \newcommand{\PP}{\mathbb{P}}        
  \newcommand{\leb}{\mathcal{L}}      
  \newcommand{\BB}{\mathcal{B}}         

  \newcommand{\ldimloc}{\underline{\dim}}
  \newcommand{\udimloc}{\overline{\dim}}

\newcommand{\iii}{\mathtt{i}}
\newcommand{\jjj}{\mathtt{j}}

  \newcommand{\e}{\varepsilon}

  \renewcommand{\AA}{\mathbb{A}}          
  \newcommand{\Q}{\mathcal{Q}}          
  \newcommand{\cP}{\mathcal{P}}           
  \newcommand{\allbasicconds}{\eqref{SI:bounded}--\eqref{SI:spatial-independence}}    
  \newcommand{\lam}{\lambda}
  
  \newcommand{\poissonian}{\mathcal{PCM}}
  \newcommand{\subdivision}{\mathcal{SM}}

  \newcommand{\I}{\text{I}}
  \newcommand{\II}{\text{II}}
  \newcommand{\III}{\text{III}}

 \DeclareMathOperator{\iso}{ISO}
 \DeclareMathOperator{\aff}{AFF}
  \DeclareMathOperator{\simi}{SIM}

  \DeclareMathOperator{\dimloc}{\dim}
  \DeclareMathOperator{\GL}{GL}
  \DeclareMathOperator{\dist}{dist}
  \DeclareMathOperator{\diam}{diam}
  \DeclareMathOperator{\supp}{supp}   

  \newtheorem{thm}{Theorem}[section]
  \newtheorem{lemma}[thm]{Lemma}
  \newtheorem{prop}[thm]{Proposition}
  \newtheorem{cor}[thm]{Corollary}

  \theoremstyle{remark}
  \newtheorem{rem}[thm]{Remark}
  \newtheorem{rems}[thm]{Remarks}
  \newtheorem{ex}[thm]{Example}
  \newtheorem{defn}[thm]{Definition}

\begin{document}

\title[SI-martingales]{Spatially independent martingales, intersections, and applications}

\author{Pablo Shmerkin}
\address{Department of Mathematics and Statistics, Torcuato Di Tella University, and CONICET, Buenos Aires, Argentina}
\email{pshmerkin@utdt.edu}
\thanks{P.S. was partially supported by Projects PICT 2011-0436 and PICT 2013-1393 (ANPCyT). Part of this research was completed while P.S. was visiting the University of Oulu.}
\urladdr{http://www.utdt.edu/profesores/pshmerkin}

\author{Ville Suomala}
\address{Department of Mathematical Sciences, University of Oulu, Finland}
\email{ville.suomala@oulu.fi}
\urladdr{http://cc.oulu.fi/~vsuomala/}



\begin{abstract}
We define a class of random measures, spatially independent martingales, which we view as a natural generalization of the canonical random discrete set, and which includes as special cases many variants of fractal percolation and Poissonian cut-outs. We pair the random measures with deterministic families of parametrized measures $\{\eta_t\}_t$, and show that under some natural checkable conditions, a.s. the mass of the intersections is H\"older continuous as a function of $t$. This continuity phenomenon turns out to underpin a large amount of geometric information about these measures, allowing us to unify and substantially generalize a large number of existing results on the geometry of random Cantor sets and measures, as well as obtaining many new ones. Among other things, for large classes of random fractals we establish (a) very strong versions of the Marstrand-Mattila projection and slicing results, as well as dimension conservation, (b) slicing results with respect to algebraic curves and self-similar sets, (c) smoothness of convolutions of measures, including self-convolutions, and nonempty interior for sumsets, (d) rapid Fourier decay. Among other applications, we obtain an answer to a question of I. {\L}aba in connection to the restriction problem for fractal measures.
\end{abstract}

\maketitle

\tableofcontents

\section{Introduction}
\label{sec:introduction}

\subsection{Motivation and overview}

One of the most basic examples in the probabilistic method in combinatorics, going back to Erd\H{o}s' classical lower bound on the Ramsey numbers $R(k,k)$, is the random subset $E=E_N$ of $\{1,\ldots,N\}$ obtained by picking each element independently with the same probability $p=p(N)$. On one hand, this random set (or suitable modifications) serves as an example in many problems for which deterministic constructions are not yet known, or hard to construct, see \cite{AlonSpencer08}. On the other hand, the random sets $E$ are often studied for their intrinsic interest. For example, recently there has been much interest in extending results in additive combinatorics such as Szemer\'{e}di's Theorem or S\'{a}rk\H{o}zy's Theorem from $\{1,\ldots, N\}$ to the random sets $E$ (needless to say, the appropriate choice of function $p(N)$ depends on the problem at hand). See e.g. \cite{ConlonGowers10, BMS14} and references therein.

Probabilistic constructions of sets and measures in Euclidean spaces also arise in many problems in analysis and geometry. Again, such constructions are often employed to provide examples of phenomena that are hard to achieve deterministically, and are also often studied for their intrinsic interest. Although there is no canonical construction as in the discrete setting, for many problems (though by no means all) of both kinds  one seeks constructions which, to some extent, share the following two key properties of the discrete canonical random set $E$: (i) all elements have the same probability of being chosen, and (ii) for disjoint sets $(A_i)$, the random sets $E\cap A_i$ are independent. See e.g. \cite{PeresSolomyak05, Korner08, LabaPramanik09, LabaPramanik11, ShmerkinSuomala12, Chen13} for some recent examples of ad-hoc constructions of this kind, meant to solve specific problems in analysis and geometric measure theory. Some classes of random sets and measures that have been thoroughly studied for their own intrinsic interest, and which also enjoy some form of properties (i), (ii) above are fractal percolation, random cascade measures and Poissonian cut-out sets (these will all be defined later). See e.g. \cite{KahanePeyriere76, DekkingSimon08, BromanCamia10, AJJRS12, NacuWerner11, RamsSimon13, PeresRams14}.

While of course many details of these papers differ, there are a number of ideas that arise repeatedly in many of them.  The main goal of this article is to introduce and systematically study a class of random measures on Euclidean spaces which, in our opinion, provides a useful analogue of the canonical discrete random set $E$, and captures the fundamental properties that are common to many of the previously cited works. In particular, this class includes the natural measures on fractal percolation, random cascades, Poissonian random cut-outs, and Poissonian products of cylindrical pulses as concrete examples (see Section \ref{subsec:general-setup} for the description of two key examples, and Section \ref{sec:examples} for the general models). Our main focus is on intersections properties of these random measures (and the random sets obtained as their supports), both for their intrinsic interest, and because they are at the heart of other geometric problems concerning projections, convolutions, and arithmetic and geometric patterns. As a concrete application, we are able to answer a question of {\L}aba from \cite{Laba14} related to the restriction problem for fractal measures. We hope that this general approach will find other applications to problems in random structures and geometric measure theory in the future.

To further motivate this work, we recall some classical results. An affine $k$-plane $A\subset\R^d$ intersects a typical affine $\ell$-plane $B\subset\R^d$ if and only if $k+\ell>d$, in which case the intersection has dimension $k+\ell-d$. Here ``typical'' means an open dense set of full measure (in fact, Zariski dense in the appropriate variety). Similar results, going back to Kakutani's work on polar sets for Brownian motion, hold when $A$ is replaced by a random set sampled from a ``sufficiently rich'' distribution, $B$ by a given deterministic set, and linear dimension by Hausdorff dimension $\dim_H$; here ``typical'' means either almost surely or with positive probability. For example, if $A$ is one of the following random subsets of $\R^d$:
\begin{enumerate}
\item \label{it:E-random-similarities} a random similar image of a fixed set $A^0$ (chosen according to Haar measure on the group of linear similitudes),
\item \label{it:E-Brownian-paths}  a Brownian path,
\item \label{it:E-fractal-perc} fractal percolation,
\end{enumerate}
and $E$ is a deterministic set, then $A$ and $E$ intersect with positive probability if and only if $\dim_H A+\dim_H E>d$; in the latter case, $A\cap E$ has dimension at most $\dim_H A+\dim_H E-d$ almost surely, and dimension at least $\dim_H A+\dim_H E-d$ with positive probability (assuming $A^0$ has positive Hausdorff measure in its dimension in \eqref{it:E-random-similarities}). See e.g. \cite{Mattila84, Kahane86, Peres96} for the proofs.

Clearly, one cannot hope to invert the order of quantifiers in any result of this kind, since a set never intersects its complement. Nevertheless, it seems natural to ask whether these results can be strengthened by replacing the fixed set $E$ by some parametrized family $\{ E_t:t\in\Gamma\}$, and asking if the random set intersects $E_t$ in the expected dimension simultaneously for \emph{all} $t\in\Gamma$ (or at least for $t$ in some open set) with positive probability. For example, this could be a family of $k$-planes, of spheres, of self-similar sets, and so on. For random similar images of an arbitrary set, it is easy to construct counterexamples, for example for the family of $k$-planes; this is due to the limited (finite dimensional) amount of randomness. Regarding Brownian paths, it is known that if $d\ge 3$ almost surely there are cut-planes, that is, there are planes $V$ and times $t$ such that $B[0,t)$ and $B(t,1]$ lie on different sides of $V$ (where $B$ is Brownian motion); in particular, the intersection of these planes with the Brownian path is a singleton. See \cite[Theorem 0.6]{BassBurdzy99}.  On the other hand, some results of this kind for fractal percolation, regarding intersections with lines, have been obtained implicitly in \cite{FalconerGrimmett92,RamsSimon13,RamsSimon14}; some of these results will be recalled later. Compared even to Brownian paths, fractal percolation has a stronger degree of spatial independence (as it is not constrained to be a curve), and as we will see this is key in obtaining uniform intersection results with even more general families.

Yet another motivation comes from geometric measure theory. Classical results going back to Marstrand \cite{Marstrand54} in the plane and Mattila \cite{Mattila75} in higher dimensions say that for a fixed set $A\subset\R^d$, ``typical'' linear projections and intersections with affine planes behave in the ``expected'' way. For example, if $\dim_H A>k$, then for almost all $k$-planes $V$, the orthogonal projection of $A$ onto $V$ has positive Lebesgue measure. See \cite{Mattila95, Mattila04} for a good exposition of the general theory. Recently there has been substantial interest in improving these geometric results for \emph{specific} classes of sets and measures, both deterministic (see e.g. \cite{HochmanShmerkin12, Hochman14} and references therein) and, more relevant to us, random (see e.g. \cite{DekkingSimon08, RamsSimon13, FalconerJin14, RamsSimon14}).

In this article we introduce a large, and in our view natural, class of  random sequences $(\mu_n)$ with a limit $\mu_\infty$, which include as special cases the natural measure on fractal percolation, as well as other random cascade measures, and the natural measure on a large class of Poissonian random cutout fractals (see Section \ref{sec:examples}). The key properties of these measures are inductive versions of the properties (i), (ii) of the canonical random discrete set. We pair these random measures with parametrized families of (deterministic) measures $\{ \eta_t:t\in\Gamma\}$, where $\Gamma$ is a totally bounded metric space with controlled growth. Examples of families that we investigate include, among others, Hausdorff measures on $k$-planes or algebraic curves, and self-similar measures on a wide class of self-similar sets.

Our main abstract result, Theorem \ref{thm:Holder-continuity}, says that under some fairly natural conditions on both the random measures and the deterministic family, the ``intersection measures'' $\mu_\infty\cap \eta_t$ are well defined, and behave in a H\"{o}lder-continuous way as a function of $t$; in particular, with positive probability they are non-trivial for a nonempty open set of $t$. (Some extensions of this theorem are presented in Section \ref{sec:products}.) The hypotheses of Theorem \ref{thm:Holder-continuity} can be checked for many natural examples of random measures $\mu_\infty$ and parametrized families of measures $\{ \eta_t\}_{t\in\Gamma}$; more concrete sufficient geometric conditions are provided in Section \ref{sec:geometric-Holder}. From Section \ref{sec:affine} on, we start applying Theorem \ref{thm:Holder-continuity} in different settings, and deducing a large variety of applications.

Before we proceed with the details, we summarize some of our results, and how they relate to our motivation (as described above), and to existing work in the literature.
\begin{enumerate}
\item The projections of planar fractal percolation to linear subspaces (as well as some classes of non-linear projections) were investigated in \cite{RamsSimon13,RamsSimon14, SimonVago14}. Among other things, in those papers it is proved that, when the dimension of the percolation set $A$ is $>1$, then all orthogonal projections onto lines have nonempty interior, and when the dimension is $\le 1$, then all projections have the same dimension as $A$. We prove that this behavior holds for a large class of natural examples, in arbitrary dimension, including much more general subdivision random fractals and many random cut-outs (see Theorems \ref{thm:linear-projections} and \ref{thm:cor_dim_small}). As indicated in our motivation, these are considerable strengthenings, for this class of sets, of the Marstrand-Mattila Projection Theorem.
\item
Moreover, in the regime when the dimension is $s>1$, we show that for the same class of measures, almost surely the intersection with all lines has box dimension at most $s-1$ (with uniform estimates); see Section \ref{sec:dim_of_intersections}. Moreover, we show that with positive probability (and full probability on survival of a natural measure), in each direction there is an open set of lines which intersect it in Hausdorff dimension at least (and therefore exactly) $s-1$; see Section \ref{sec:lower-bound-dim-intersections}.
Note that this is much stronger than asserting that the projections have nonempty interior. Moreover, our results hold for any ambient dimension and intersections with $k$-planes, for any $k$. Furthermore, we prove similar results for intersections with large classes of self-similar sets, and (in the plane)  also with algebraic curves. In particular, these results apply to fractal percolation and random cut-outs, thereby extending Hawkes' classical intersection result \cite{Hawkes81} from a single set to natural families of sets as well as the results of Z\"ahle \cite{Zahle84}, who considered intersections of random cut-outs with a fixed $k$-plane.
\item Peres and Rams \cite{PeresRams14} proved that for the natural measure $\mu$ on a fractal percolation of dimension $>1$ in the plane, all image measures under an orthogonal projection are absolutely continuous and, {\em other than the principal directions}, have a H\"{o}lder continuous density. For the principal directions, the density is clearly discontinuous, and a similar phenomenon occurs for more general models defined in terms of subdivision inside a polyhedral grid. This led us to investigate the following question: given $1<s<2$, does there exist a measure $\mu$ supported on a set of Hausdorff dimension $s$, such that {\em all} orthogonal projections of $\mu$ have a H\"{o}lder continuous density? We give a strong affirmative answer; there is a rich family of such measures, including many arising from Poissonian cut-out process and percolation on a self-similar tiling. The cut-out construction works in any dimension and we obtain \emph{joint} H\"{o}lder continuity in the orthogonal map as well, see Theorem \ref{thm:linear-projections}. Furthermore, we look into the larger class of polynomial projections, and establish the existence of a measure $\mu$ on $\R^2$ supported on a set of any dimension $s$, $1<s<2$, with the property that \emph{all} polynomial images are absolutely continuous with a piecewise locally H\"{o}lder density. See Theorem \ref{thm:alg_curves}.
\item Closely related to the size of slices is the concept of \emph{dimension conservation}, introduced by Furstenberg in \cite{Furstenberg08}; roughly speaking, a Lipschitz map is dimension conserving if any loss of dimension in the image is compensated by an increase in the dimension of the fibres. We prove that affine and even polynomial maps restricted to many of our random sets are dimension conserving in a very strong fashion. In particular, we partially answer a question of Falconer and Jin \cite{FalconerJin14b}. See Section \ref{subsec:dimension-conservation}.
\item Another problem that has attracted much interest concerns understanding the geometry of arithmetic sums and differences of fractal sets. This is motivated in part by Palis' celebrated conjecture that typically, if $A_1, A_2$ are Cantor sets in the real line with $\dim_H(A_1)+\dim_H(A_2)>1$, then their difference set $A_1-A_2$ contains an interval. Although the conjecture was settled by Moreira and Yoccoz \cite{MoreiraYoccoz01} in the dynamical context most relevant to Palis' motivation, much attention has been devoted to its validity (or lack thereof) for various classes of random fractals, see e.g. \cite{DekkingSimon08, DSS11, DekkingKuijvenhoven11} and the references there. We prove a very strong version of Palis' conjecture when $A_1, A_2$ are independent realizations of a large class of random fractals in $\R^d$, including again many Poissonian cut-outs, as well as subdivision-type random fractals which include fractal percolation as a particular case. Namely, we show that under a suitable non-degeneracy assumption, if $\dim_H A_1, \dim_H A_2>d/2$, then $A_1+S A_2$ has nonempty interior for all $S\in\GL_d(\R)$ simultaneously. In fact, we deduce this from an even stronger result about measures: if $\mu_1,\mu_2$ are independent realizations of a random measure (which again may come from a Poissonian cut-out or repeated subdivision type of process), and the supports have dimension $>d/2$, then the convolution $\mu_1* (S\mu_2)$ is absolutely continuous with a H\"{o}lder density (and also jointly H\"{o}lder in $S$). These results are presented in Section \ref{subsec:products-independent}. See also Theorem \ref{thm:sum-random-and-deterministic} for a result on the arithmetic sum of a random set and an arbitrary deterministic Borel set.
\item It has been known since Wiener and Wintner \cite{WienerWintner38} that there are singular measures $\mu$ such that the self-convolution $\mu*\mu$ is absolutely continuous with a bounded density. Constructing examples of increasingly singular measures $\mu$ with increasingly regular self-convolutions $\mu*\mu$ is the topic of several papers (see e.g. \cite{GuptaHare04, Saeki80, Korner08}). In particular, for any $s\in (1/2,1)$, K\"{o}rner \cite{Korner08} constructs a random measure $\mu$ on the line supported on a set of dimension $s$, such that the self-convolution $\mu*\mu$ is absolutely continuous and has a H\"{o}lder density with exponent $s-1/2$, which he shows to be optimal. While K\"{o}rner has an ad-hoc construction, we show that for our main classes of examples we obtain a similar behaviour, other than for the value of the H\"{o}lder exponent. Furthermore, we prove a stronger result: for each $s\in (d/2,d)$, we exhibit a large class of random measures $\mu$ on $\R^d$ supported on sets of dimension $s$, such that $\mu * S\mu$ is absolutely continuous with a H\"{o}lder density with exponent $\gamma=\gamma(s,d)>0$, for any $S\in\GL_d(\R)$ which does not have $-1$ in its spectrum (if $S+\I_d$ is not invertible the conclusion can fail in general, but we still get a weaker result). See Section \ref{subsec:self-convolutions}.
\item Recall that a Salem measure is a measure whose Fourier transform decays as fast as its Hausdorff dimension allows, see Section \ref{subsec:Salem}. We prove that a class of measures, which includes the natural measure on fractal percolation, are Salem measures when their dimension is $\le 2$ (and this is sharp), adding to the relatively small number of known examples of Salem measures. See Theorem \ref{thm:Salem}. By the result discussed above, these measures have continuous self-convolutions provided they have dimension $>d/2$, and some of them are also Ahlfors regular. The existence of measures with these joint properties is of importance in connection with the restriction problem for fractal measures, and enables us to answer questions of Chen from \cite{Chen14} and {\L}aba from \cite{Laba14}, see Section \ref{subsec:question-laba}. We are also able to prove that a wide class of random measures has a power Fourier decay with an explicit exponent, see Corollary \ref{cor:positive-Fourier-dim}.
\end{enumerate}

\subsection{General setup and major classes of examples}
\label{subsec:general-setup}

Our general setup is as follows. We consider a class of random measures $\mu_\infty$ on $\R^d$, obtained as weak limits of absolutely continuous measures with density $\mu_n$ (we will often identify the densities $\mu_n$ with the corresponding measures). These are Kahane's $T$-martingale measures with $T=\R^d$, together with extra growth and independence conditions to be defined later; intuitively the measure $\mu_n$ should be thought of as the approximation to $\mu$ at scale $2^{-n}$ (we use dyadic scaling for notational convenience). We pair each $\mu_n$ with a family of deterministic measures $\{\eta_t\}_{t\in\Gamma}$, where the parameter set $\Gamma$ is a metric space. We study the (mass of the) ``intersections'' of $\mu_n$ and $\mu_\infty$ with the measures $\eta_t$. A priori there is no canonical way to define this, but we employ the fact that $\mu_\infty$ is a limit of pointwise defined densities
to define the limit of the total masses
\[
Y^t = \lim_{n\to\infty} \int \mu_n(x) \,d\eta_t(x)\,,
\]
provided it exists. Our main abstract result, Theorem \ref{thm:Holder-continuity}, gives broad conditions on the sequence $(\mu_n)$ and the family $\{\eta_t\}_{t\in\Gamma}$ that guarantee that the function $t\mapsto Y^t$ is everywhere well defined and H\"{o}lder continuous. These general conditions can be checked in many concrete situations, leading to the consequences and applications outlined above.

To motivate the general results, we describe two key examples, and defer to Section \ref{sec:examples} for generalizations and further examples. The first one is fractal percolation, which is also sometimes termed Mandelbrot percolation. Fix an integer $M$ and a parameter $p\in (0,1)$. We subdivide the unit cube in $\R^d$ into $M^d$ equal closed sub-cubes. We retain each of them with probability $p$ and discard it with probability $1-p$, with all the choices independent. For each of the retained cubes, we continue inductively in the same fashion, by further subdividing them into $M^d$ equal sub-cubes, retaining them with probability $p$ and discarding them otherwise, with all the choices independent. The fractal percolation limit set $A$ is the set of points which are kept at each stage of the construction, see Figure \ref{fig:perco} for an illustration of the first few steps of the construction.
\begin{figure}
   \centering
  \includegraphics[width=\textwidth]{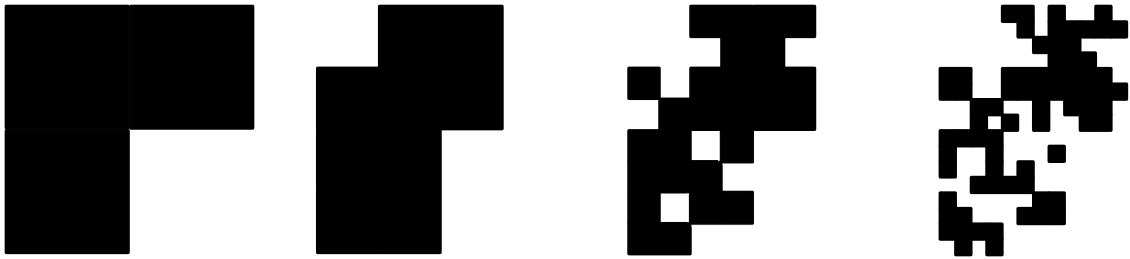}
\caption{The first 4 steps in a fractal percolation process with $p=0.7$.}
\label{fig:perco}
\end{figure}
It is well known that if $p\le M^{-d}$, then $A$ is a.s. empty, and otherwise a.s. $\dim_H A=\log (p M^d)/\log M$ conditioned on $A\neq\varnothing$. The natural measure on $A$ is the weak limit of $\mu_n:=p^{-n} \leb^d|_{A_n}$, 
where $A_n$ is the union of all retained cubes of side length $M^{-d}$, and $\leb^d$ denotes $d$-dimensional Lebesgue measure. Fractal percolation is statistically self-similar with respect to the transformations that map the unit cube to the $M$-adic sub-cubes of size $1/M$.

The geometric measure theoretical properties of fractal percolation and related models have been studied in depth, see \cite{FalconerGrimmett92, MauldinWilliams86, DekkingSimon08, AJJRS12, RamsSimon13, PeresRams14, SimonVago14, FalconerJin14b}. There is another large class of random sets and measures that has achieved growing interest in the probability literature (see e.g. \cite{BromanCamia10, NacuWerner11}) but less so in the fractal geometry literature (although the model essentially  goes back to Mandelbrot \cite{Mandelbrot72}, and Z\"ahle \cite{Zahle84} considered a closely related model).
We describe a particular example and leave further discussion to Section \ref{subsec:Poissonian}. Let $\mathbf{Q}$ be the measure $r s^{-1} dx ds$ on $\R^d\times (0,1/2)$, where $r>0$ is a real parameter. Recall that a Poisson point process with intensity $\mathbf{Q}$ is a random countable collection of points $\mathcal{Y}=\{ (x_j,r_j)\}$ such that:
\begin{itemize}
\item For any Borel set $B\subset \R^d\times (0,1/2)$, the random variable $\#(\mathcal{Y}\cap B)$  is Poisson with mean $\mathbf{Q}(B)$ ($\#X$ denotes the cardinality of $X$).
\item If $\{ B_j\}$ are pairwise disjoint subsets of $\R^d\times (0,1/2)$, then the random variables $\#(\mathcal{Y}\cap B_j)$ are independent.
\end{itemize}
One can then form the random cut-out set $A=\overline{B(0,1)\setminus\bigcup_j B(x_j,r_j)}$, see Figure \ref{fig:ball_type} for an approximation.
\begin{figure}
   \centering
  \includegraphics[width=0.7\textwidth]{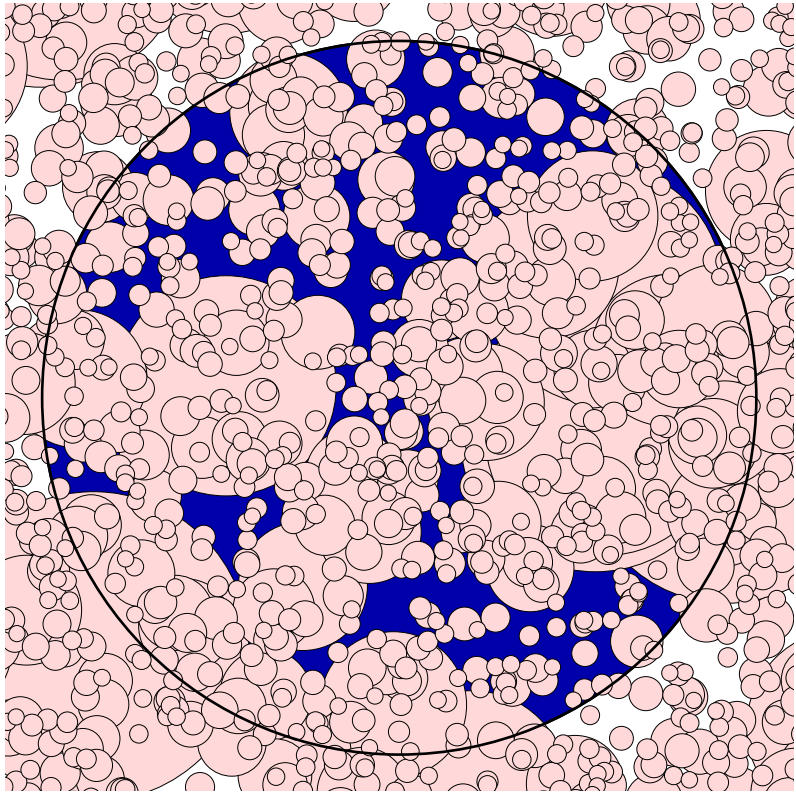}
\caption{A Poissonian cut-out process.}
\label{fig:ball_type}
\end{figure}
There is a natural measure $\mu_\infty$ supported on $A$: it is the weak limit of $d\mu_n(x):=2^{\alpha n}\mathbf{1}_{A_n}(x)$, where $A_n=B(0,1)\setminus\bigcup \{B(x_j,r_j):r_j>2^{-n}\}$, and $\alpha=r c_d$, where $c_d$ is a constant depending only on the ambient dimension $d$. It follows from standard techniques that if $\alpha\le d$ then $\dim_H A=d-\alpha$ almost surely conditioned on $\mu_\infty\neq0$; and otherwise $A$ is almost surely empty.

This model is invariant (in law) under arbitrary rotations. In particular, unlike subdivision models in a polyhedral grid, there are no ``exceptional directions''. This will be an important feature in some of our applications.

\section{Notation}
\label{sec:notation}

A measure will always refer to a locally finite Borel-regular outer measure on some metric space. The notation $B(x,r)$ stands for the closed ball of centre $x$ and radius $r$ in a metric space which will always be clear from context. The open ball will be denoted by $B^\circ(x,r)$.

We will use Landau's $O(\cdot)$ and related notation. If $n>0$ is a variable by $g(n) = O(f(n))$ we mean that there exists $C>0$ such that $g(n)\le C f(n)$ for all $n$. By $g(n)=\Omega(f(n))$ we mean $f(n)=O(g(n))$, and by $g(n)=\Theta(f(n))$ we mean that both $g(n)=O(f(n))$ and $g(n)=\Omega(f(n))$ hold. As usual, $g(n)=o(f(n))$ means that $\lim_{n\rightarrow\infty}g(n)/f(n)=0$. Occasionally we will want to emphasize the dependence of the constants implicit in the $O(\cdot)$ notation on other previously defined constants; the latter will be then added as subscripts. For example, $g(n) = O_\delta(f(n))$ means that $g(n)\le C_\delta f(n)$ for some constant $C_\delta$ which is allowed to depend on $\delta$.

We always work on some Euclidean space $\R^d$. Most of the time the ambient dimension $d$ will be clear from context so no explicit reference will be made to it. We denote by $\mathcal{Q}_n$ the family of dyadic cubes of $\R^d$ with side length $2^{-n}$. It will be convenient that these are pairwise disjoint, so we consider a suitable half-open dyadic filtration.

The indicator function of a set $E\subset\R^d$ will be denoted by either $\mathbf{1}_E$ or $\mathbf{1}[E]$, and we will write $E(\e)$ for the $\e$-neighbourhood $\{ x\in\R^d: \dist(x,E)<\e\}$. We denote the symmetric difference of two sets by $A\Delta B=(A\setminus B)\cup(B\setminus A)$.

We denote the family of finite Borel measures on $\R^d$ by $\mathcal{P}_d$. The trivial measure $\mu(B)=0$ for all sets $B\subset\R^d$ is considered as an element of $\mathcal{P}_d$. Given $\mu\in\mathcal{P}_d$, we denote $\|\mu\|=\mu(\R^d)$.

As noted earlier, $\dim_H$ denotes Hausdorff dimension. We denote upper box-counting (or Minkowski) dimension by $\overline{\dim}_B$, and box-counting dimension (when it exists) by $\dim_B$. A good introduction to fractal dimensions can be found in \cite[Chapters 2 and 3]{Falconer03}. For $\mu\in\mathcal{P}_d$ and $x\in\R^d$, we define the lower and upper dimensions of $\mu$ at $x$ by
\begin{align*}
\ldimloc(\mu,x)&=\liminf_{r\downarrow0}\frac{\log \mu(B(x,r))}{\log r}\,,\\
\udimloc(\mu,x)&=\limsup_{r\downarrow0}\frac{\log \mu(B(x,r))}{\log r}\,,
\end{align*}
and denote the common value by $\dimloc(\mu,x)$ if they are the same.

Let $\iso_d$ be the family of isometries of $\R^d$. This is a manifold diffeomorphic to $\mathbb{O}_d\times \R^d$ (via $(O,y)\mapsto f(x)=Ox+y$), where $\mathbb{O}_d$ is the $d$-dimensional orthogonal group. On $\mathbb{O}_d$ and for more general families of linear maps, we use the standard metric induced by the Euclidean operator norm $\|\cdot\|$, and this also gives a metric in $\iso_d$, and also on the space $\aff_d$ of affine maps: $d(f_1,f_2)=\|g_1-g_2\|+|z_1-z_2|$ if $g_i(x)=f_i(x)+z_i$ for $f_i$ linear and $z_i\in\R^d$. Likewise, $\simi_d$ will denote the space of non-singular similarity maps, which is identified with $(0,\infty)\times\mathbb{O}_d\times\R^d$ via $(r,O,y)\mapsto f(x)=rO(x)+y$. The space of contracting similarities, i.e. those for which $r<1$, will be denoted by $\simi_d^c$.  The identity map of $\R_d$ is denoted by $\I_d$.

The Grassmannian of $k$-dimensional linear subspaces of $\R^d$ will be denoted $\mathbb{G}_{d,k}$. It is a compact manifold of dimension $k(d-k)$, and its metric is
\[
d(V,W)=\|P_V-P_W\|\,,
\]
where $P_{(\cdot)}$ denotes orthogonal projection. The manifold of $k$-dimensional affine subspaces of $\R^d$ will be denoted $\AA_{d,k}$. It is diffeomorphic to $\mathbb{G}_{d,k} \times \R^{d-k}$, and this identification defines a natural metric.

The metrics on all these different spaces will be denoted by $d$; the ambient space will always be clear from context (also note that these metrics and the ambient dimension are denoted by the same symbol $d$).

For $m\ge 2$, let $\Sigma_m$ be the full shift on $m$ symbols. Given $F=(f_1,\ldots,f_m)\in(\simi_{d}^c)^m$, let $\Pi_F$ denote the induced projection map $\Sigma_m\to\R^d$, $(i_1,i_2,\ldots)\mapsto\lim_{N\rightarrow\infty} f_{i_1}\circ\ldots f_{i_N}(0)$, and let $E_F$ be the self-similar set $E_F=\Pi_F(\Sigma_m)$. For further background on iterated function systems, including the definitions of the open set and strong separation conditions, see e.g. \cite[Section 2.2]{Falconer97}.

Throughout the paper, $C, C', C_1$, etc, denote deterministic constants whose precise value is of no importance (and their value may change from line to line), while $K, K'$ etc. will always denote random real numbers.

We summarize our notation and notational conventions in Table \ref{notation}. Many of these concepts will be defined later.
\begin{table}
\begin{tabular}[h]{|l|l|}
\hline
$\leb$, $\leb^d$ & Lebesgue measure on $\R$, $\R^d$\\
$(\mu_n)$,\,\, $\mu_\infty$ & an SI-martingale and its limit\\
$\mathcal{PCM}$ & Poissonian cutout martingales\\
$\mathcal{SM}$ & subdivision martingales\\
$\mu_n^{\text{ball}}$,\, $\mu_n^{\text{snow}}$, \,$\mu_n^{\text{perc}}$  & specific examples of SI-martingales of various types\\
$A_n$, $A$ & random sets related to an SI-martingale $(\mu_n)$:\\
			& $A=\{x\in\R^d\,:\,\mu_n(x)\neq 0\}$, $A=\overline{\cap_n A_n}$\\
$\{\eta_t\,:\,t\in\Gamma\}$ & parametrized family of (deterministic) measures\\
$\mu_{n}^t$ & the ``intersection'' of $\mu_n$ and $\eta_t$.\\
$Y_{n}^t$,\,\, $Y^t$ & total mass of $\mu_{n}^t$, and its limit\\
$\mathcal{X}$ & the family of compact sets of $\R^d$\\
$\mathbf{Q}, \mathbf{Q}_0$ & intensity measures on $\mathcal{X}$\\
$\mathcal{Y}$ &Poisson point process (with intensity $\mathbf{Q}$)\\
$\mathcal{F}$ & family of Borel sets consisting of\\ 					&``typical'' shapes of a specific SI-martingale\\
$\Lambda$ & an element of $\mathcal{F}$\\
$\Omega$ & the seed of the SI-martingale (the support of $\mu_0$)\\
$\mathcal{K}$,\,\,$\mathcal{K}_k$ & the family of real algebraic curves in $\R^2$\\
& and the ones of degree at most $k$.\\
$\mathcal{P}_d$ & the collection of finite Borel measures on $\R^d$.\\
$\dim_F\mu$ & Fourier dimension of
$\mu\in\mathcal{P}_d$\\
$\mathcal{Q}$,\,\, $\mathcal{Q}_n$ & the family of half-open dyadic cubes of $\R^d$\\
& and the ones with side-length $2^{-n}$\\
$\I_d$ & the identity map on $\R^d$\\
$\iso_d$ & the family of isometries of $\R^d$\\
$\aff_d$ & the family of affine maps on $\R^d$\\
$\AA_{d,k}$		& the manifold of $k$-dimensional affine subspaces of $\R^d$\\
$\simi_d^c$ & the family of contracting similarities on $\R^d$\\
$E_F$ & the self-similar set corresponding to the IFS $F\in(\simi_d^c)^m$\\
$\Sigma_m$ & the code space $\{1,\ldots, m\}^\N$.\\
$E(\e)$ & open $\e$-neighbourhood of a set $E\subset\R^d$\\
$\Lambda^\kappa$ & a quantitative interior of the set $\Lambda$\\
$\Lambda_\varrho$ & regular inner approximation of $\Lambda$\\
$P_V$ &orthogonal projection onto the linear subspace $V$\\
\hline
\end{tabular}
\vspace{2mm}
\caption{Summary of notation\label{notation}}
\end{table}

\section{The setting}
\label{sec:setting}

\subsection{A class of random measures}
\label{subsec:classofrandomeasures}

In this section we introduce our general setup. Recall that our ultimate goal is to study intersection properties of random measures $\mu_\infty$  with a deterministic family of measures $\{ \eta_t\}_{t\in\Gamma}$ (and likewise for their supports). We begin by describing the main properties that will be required of the random measures.

We consider a sequence of functions $\mu_n\colon\R^d\to[0,+\infty)$, corresponding to the densities of absolutely continuous measures (also denoted $\mu_n$) satisfying the following properties:

\begin{enumerate}
\renewcommand{\labelenumi}{(SI\arabic{enumi})}
\renewcommand{\theenumi}{SI\arabic{enumi}}
\item \label{SI:bounded} $\mu_0$ is a deterministic bounded function with bounded support.
\item \label{SI:martingale} There exists an increasing filtration of $\sigma$-algebras $\mathcal{B}_n$ such that $\mu_n$ is $\mathcal{B}_n$-measurable.
Moreover, for all $x\in\R^d$ and all $n\in\N$,
\[
\EE(\mu_{n+1}(x)|\BB_n) = \mu_n(x)\,.
\]
\item \label{SI:quotients-bounded}
There is $C<\infty$ such that $\mu_{n+1}(x)\le C\mu_n(x)$ for all $x\in\R^d$ and $n\in\N$.

\item \label{SI:spatial-independence} There is $C<\infty$ such that for any $(C 2^{-n})$-separated family $\mathcal{Q}$ of dyadic cubes of length $2^{-(n+1)}$, the restrictions $\{\mu_{n+1}|_Q | \mathcal{B}_n\}$ are independent.

\end{enumerate}

\begin{defn}
We call a random sequence $( \mu_n)$ satisfying \eqref{SI:bounded}--\eqref{SI:spatial-independence}  a \textbf{spatially independent martingale}, or SI-martingale for short.
\end{defn}

In other words, $(\mu_n)$ is a $T$-martingale (with $T=\R^d$) in the sense of Kahane  \cite{Kahane1987} with the extra growth and independence conditions \eqref{SI:quotients-bounded}, \eqref{SI:spatial-independence}. Intuitively, $\mu_n$ should be thought of as an absolutely continuous approximation of $\mu$ at scale $2^{-n}$.

It is well known that a.s. the sequence $(\mu_n)$ is weakly convergent; denote the limit by $\mu_\infty$. It follows easily that the sequence $\{\supp \mu_n\}$ is a decreasing sequence of compact sets and that $\supp(\mu_\infty)\subset\bigcap_{n=1}^\infty \supp\mu_n$. Note that we do not exclude the possibility that $\mu_n$ is trivial for some (and hence all sufficiently large) $n$.

We remark that the $\mu_n$ are actual functions (defined pointwise) and not just elements of $L^1$. This is crucial because we will be integrating these functions with respect to singular measures. We also note that the fractal percolation and ball cutout examples discussed in the introduction are easily checked to be SI-martingales.

We call condition \eqref{SI:spatial-independence} \textbf{uniform spatial independence}. Although it is the central property that sets our class apart from general Kahane martingales, all of our results hold under substantially weaker independence hypotheses. Since many important examples do indeed have uniform spatial independence, in this article we always assume this condition (except for slight variations in Section \ref{sec:products}), and defer the study of the weaker conditions and their consequences to a forthcoming article \cite{ShmerkinSuomala14}.

Starting with the seminal paper of Kahane \cite{Kahane1987}, there is a rich literature on $T$-martingales, and the important special case of random multiplicative cascades: see e.g. \cite{Liu00, BarralMandelbrot02, Barral03, BarralFan05}. In these papers the main emphasis is on the multifractal properties of the limit measures. Our conditions certainly do not exclude multifractal measures (in particular, large classes of random multiplicative cascades are indeed SI-martingales to which many of our results apply), but our emphasis is different, and for simplicity most of our examples will be monofractal measures.

An important special case is that in which $\mu_n=\beta_n^{-1}\mathbf{1}_{A_n}$ for some (possibly random) sequence $\beta_n$. Denote $A=\overline{\cap_n A_n}$. In this case, $\beta_n^{-1}$ should be thought of as the approximate value of the Lebesgue volume of $A(2^{-n})$, the $(2^{-n})$-neighbourhood of $A$. Recall that the box dimension of a set $E\subset\R^d$ can be defined as
\[
\dim_B(E) = \lim_{n\to\infty} d - \frac{\log_2\leb^d(E(2^{-n}))}{n}\,.
\]
(See e.g. \cite[Proposition 3.2]{Falconer03}). Thus, intuitively, $\lim_{n\to\infty}\frac{\log_2\beta_n}{n}$ should equal $d-\dim(A)$.
This statement can be verified (for both $\dim_H$ and $\dim_B$) in many cases, but in the generality of the given hypotheses it may fail.

\subsection{Parametrized families of measures}
\label{subsec:parametrized-familis}

We now introduce the parametrized families $\{ \eta_t\}_{t\in\Gamma}$ of (deterministic) measures. We always assume the parameter space is a totally bounded metric space $(\Gamma,d)$.  We start by introducing some natural classes of examples; we will come back to them repeatedly in the later parts of the paper. In all cases, $\Upsilon$ is a fixed bounded subset of $\R^d$, such as the unit ball.

\begin{itemize}
\item For some $1\le k<d$, $\Gamma$ is the subset of $\mathbb{A}_{d,k}$ of $k$-planes which intersect $\Upsilon$, with the induced natural metric, and $\eta_V$ is $k$-dimensional Hausdorff measure on $V\in\Gamma$.
\item In this example, $d=2$. Given some $k\in\N$, $\Gamma$ is the family of all algebraic curves of degree at most $k$ which intersect $\Upsilon$, $d$ is a natural metric (see Definition \ref{def:metric-algebraic}) and $\eta_\gamma$ is length measure on $\gamma\cap \Upsilon$.
\item  Let $\nu$ be an arbitrary measure, and let $\Gamma$ be a totally bounded subset of $\iso_d$ with the induced metric. The measures are $\eta_f=f\nu$.
This example generalizes the first one (in which $\nu$ is $k$-dimensional Lebesgue measure on some fixed $k$-plane).
\item Let $m\ge 2$, and let $\Gamma$ be a totally bounded subset of $(\simi_d^c)^m$. Suppose that each iterated function system (IFS) $(F_1,\ldots,F_m)\in\Gamma$ satisfies the open set condition. The measure $\eta_{(F_1,\ldots,F_m)}$ is the natural self-similar measure for the corresponding IFS.
\end{itemize}

We will occasionally state results for all $t\in\Gamma$, where $\Gamma$ is actually unbounded (for example, $\Gamma=\mathbb{A}_{d,k}$). However in these cases it will be clear that the statement is non-trivial only for those $t\in\Gamma$ for which $\supp\eta_t$ intersects a fixed compact set, and this family will be totally bounded.

In most of the examples above, the $\eta_t$-mass of small balls is controlled by a power of the radius which is uniform both in the centre and the parameter $t$. This motivates the following definition.
\begin{defn}
We say that the family $\{ \eta_t\}_{t\in\Gamma}$ has \textbf{Frostman exponent}
$s>0$, if there exists a constant $C>0$ such that
\begin{equation}\label{eq:def_dim_deterministic_family}
\eta_t(B(x,r)) \le C r^s\quad\text{for all }x\in\R^d,\, t\in\Gamma,\,0<r<1\,.
\end{equation}
\end{defn}

We emphasize that $s$ is not unique. In practice, we try to choose $s$ as large as possible, but even in that case, $s$ is the ``worst-case'' exponent over all measures, and for some $t$ better Frostman exponents may exist.

Our main objects of interest will be the ``intersections'' of the random measures $\mu_n$ and $\mu_\infty$ with  $\eta_t$, and their behaviour as $t$ varies.
Formally, we define:
\[
\mu_n^t(A) = \int_A \mu_n(x) d\eta_t(x)\,,
\]
for each Borel set $A\subset\R^d$, $n\in\N$ and $t\in\Gamma$. Note that for each fixed $t$, $(\mu^{t}_n)$ is again a $T$-martingale, thus there is a.s a weak limit $\mu^{t}_\infty$ with $\supp\mu^{t}_\infty\subset\supp\eta_t$. We are mainly interested in the asymptotic behaviour of the total mass, and denote
\begin{align*}
Y_n^t &= \|\mu_n^t\| = \int \mu_n(x) d\eta_t(x)\,,\\
Y^t &=\lim_{n\to\infty} Y_n^t \quad\text{(if the limit exists)}\,.
\end{align*}

The reason we focus on the masses $Y^t$ rather than the actual measures  $\mu^t_\infty$ is twofold. Firstly, for some of our target applications, we only want to know that certain fibers containing the support of the $\mu_t^\infty$ are nonempty, and for this $Y^t>0$ suffices. Secondly, $Y^t$ itself captures (perhaps surprisingly) detailed information about the measures $\mu^t_\infty$ (and their supports), such as their dimension. See Sections \ref{sec:dim-of-projections}--\ref{sec:lower-bound-dim-intersections} for details.

Since our random densities $\mu_n$ are compactly supported, it follows that a.s. for each fixed $t$, $Y^t$ equals $\mu^{t}_\infty(\R^d)$. In Theorem \ref{thm:Holder-continuity}, we prove that in many cases $Y^t$ is a.s. defined for \emph{all} $t\in\Gamma$ and H\"older continuous with respect to $t$. We call the measures $\mu^{t}_n$ and $\mu^{t}_\infty$ ``intersections'', because our results have corollaries on the size of the intersection of $\supp\mu_\infty$ and $\supp\eta_t$ (see Sections \ref{sec:dim_of_intersections} and \ref{sec:lower-bound-dim-intersections}), but also due to the close connection to the more standard intersection measures defined via the slicing method, see \cite[Section 13.3]{Mattila95}. To emphasize this connection, we include the following proposition (which will not be used later in the paper). We omit the proof, which is a simple exercise combining the definition above with those found in \cite{Mattila95} for the intersections $\mu\cap\eta_t$ for almost all $t\in\R^d$.

\begin{prop}
Let $\Gamma=\R^d$ and $\eta_t$ be the translate of a fixed measure $\eta$ under $x\mapsto t+x$. Let $(\mu_n)$ be an SI-martingale. Then, for all $n\in\N$, it follows that
\[
\mu_{n}^{t}=\mu_n\cap \eta_t\text{ for almost all }t\in\R^d\,.
\]
\end{prop}
In many cases, we can use the results of this paper to show that the above proposition remains true for the limit measures and holds for all $t$, i.e. $\mu^{t}_\infty=\mu_\infty\cap \eta_t$ a.s. for all $t$ simultaneously. It is also possible to consider intersections for more general classes of transformations. We do not pursue this direction further since, for our applications, the limit of the total mass $Y^t$ is more important (and easier to handle) than the intersection measures $\mu^{t}_n$, $\mu^{t}_\infty$ themselves.

The role of uniform spatial independence is to ensure that, with overwhelming probability, the convergence of $Y_n^t$ is very fast, provided $\|\mu_n\|_\infty$ does not grow too quickly. This is made precise in the next key technical lemma, which, apart from slight modifications of the same argument, is the only place in the article where spatial independence gets used. Special cases of this appear in \cite{FalconerGrimmett92}, \cite{PeresRams14} and \cite[Theorem 3.1]{ShmerkinSuomala12}, and our proof is similar.

\begin{lemma} \label{lem:large-deviation}
Let $(\mu_n)$ be an SI-martingale. Fix $n$, and let $\eta\in\mathcal{P}_d$ such that $\eta(Q)\le C_1\, 2^{sn}$ for all $Q\in\mathcal{Q}_n$. Write $M=\sup_{x\in\R^d}\mu_n(x)$. Then, for any $\kappa>0$ with
\begin{equation}\label{eq:kappaeq}
\kappa^2 2^{sn}M^{-1}\ge\delta>0\,,
\end{equation}
it holds that
\[
\PP\left(\left.\left|\int \mu_{n+1} \,d\eta - \int \mu_n \,d\eta \right|\ge \kappa\sqrt{\left|\int \mu_n\,d\eta\right|}\,\right| \,\mathcal{B}_n\right) = O\left(\exp\left(-\Omega_{C,C_1,\delta}\left(\kappa^2 2^{s n}M^{-1}\right)\right)\right),
\]
where $C$ is the constant from the definition of SI-martingale.
\end{lemma}

In particular, this holds uniformly for all measures in a family $\{\eta_t\}$ with Frostman exponent $s$. In the proof we will make use of Hoeffding's inequality \cite{Hoeffding63}:
\begin{lemma} \label{lem:HoeffdingJanson}
Let $\{ X_i: i\in I\}$ be zero mean independent random variables satisfying $|X_i|\le R$. Then for all $\kappa>0$,
\begin{equation} \label{eq:Hoeffding-Janson}
\mathbb{P}\left(\left|\sum_{i\in I} X_i\right|> \kappa\right) \le 2\exp\left(\frac{-\kappa^2}{2R^2\# I}\right).
\end{equation}
\end{lemma}

\begin{proof}[Proof of Lemma \ref{lem:large-deviation}]
By replacing $\kappa$ by $\kappa/\sqrt{\delta}$ we may assume that $\delta=1$. We condition on $\BB_n$, and write $d\nu_n= \mu_n \, d\eta$ and $Y_n=\nu_n(\R^d)=\int\mu_n\,d\eta$ for simplicity. The constants implicit in the $O$ notation may depend on $C, C_1$. We decompose $\mathcal{Q}_{n+1}$ into the families
\[
\mathcal{Q}_{n+1}^\ell = \{ Q\in \mathcal{Q}_{n+1}: C M 2^{-s\ell}< \nu_n(\widehat{Q}) \le C M 2^{s(1-\ell)} \}\,,
\]
where $\widehat{Q}\in\mathcal{Q}_n$ is the dyadic cube containing $Q$. Then $\mathcal{Q}_{n+1}^\ell$ is empty for all $\ell\le n$ (of course it is also empty for all but finitely many other $\ell$).
For each $Q\in\Q$, let $X_Q=\nu_{n+1}(Q)-\nu_n(Q)$. Then $\EE(X_Q)=0$ for all $Q\in\Q$, thanks to the martingale assumption \eqref{SI:martingale} (recall that we are conditioning on $\mathcal{B}_n$). Also, by \eqref{SI:quotients-bounded},
\[
|X_Q| = O(\nu_n(Q)) \le O(1) 2^{-s\ell} M \,\,\text{ for all } Q\in\Q_{n+1}^\ell\,.
\]
Moreover, since $Y_n = \sum_{Q\in\Q_n} \nu_n(Q)$,
\[
\#\Q_{n+1}^\ell = O(1) 2^{s\ell}M^{-1} Y_n\,.
\]

Thanks to \eqref{SI:spatial-independence}, we can split the random variables $\{X_Q\}_{Q\in\Q_{n+1}^\ell}$ into $O(1)$ disjoint families, such that the random variables inside each family are independent. By Hoeffding's inequality \eqref{eq:Hoeffding-Janson},
\[
\PP\left(\left|\sum_{Q\in\Q_{n+1}^\ell} X_Q\right|>\frac{\kappa\sqrt{Y_n}}{2(\ell-n)^2} \right) =O\left(\exp\left(-\Omega\left((\ell-n)^{-4}\kappa^2 2^{s\ell} M^{-1}\right)\right)\right)\,,
\]
for any $\kappa>0$. It follows that
\begin{align*}
\PP\left(\left|Y_{n+1}-Y_n\right|>\kappa\sqrt{Y_n} \right) &\le \sum_{\ell> n}\PP\left(\left|\sum_{Q\in\Q_{n+1}^\ell} X_Q\right|>\frac{\kappa\sqrt{Y_n}}{2(\ell-n)^2} \right) \\
&=O(1) \exp\left(-\Omega\left(\kappa^2 2^{s n}M^{-1}\right)\right),
\end{align*}
for any $\kappa>0$, where we use $\kappa^2 2^{sn} M^{-1}\ge 1$ for the last estimate. This is what we wanted to show.
\end{proof}

As a first consequence of Lemma \ref{lem:large-deviation}, we deduce that under a natural assumption $\PP(Y^t)>0$ and, in particular, the limit $\mu_\infty$ is non-trivial with positive probability.

\begin{lemma} \label{lem:Yt-survives}
Let $\eta\in\mathcal{P}_d$ satisfy $\eta(B(x,r)) \le C_1 r^s$ for all $x\in\R^d$ and some $C_1,s>0$. Let $( \mu_n)$ be an SI-martingale such that a.s. $\mu_n(x)\le 2^{\alpha n}$ for all $n,x$, where $\alpha<s$, and suppose that $\int \mu_0\,d\eta>0$.

Then the sequence $\int \mu_n\,d\eta$ converges a.s. to a non-zero random variable $Y$. Moreover,
\[
\PP\left(Y>M\right)=O(\exp(-\Omega(M)))\,,
\]
where the implicit constants are independent of $M$ but may depend on the remaining data.
\end{lemma}
\begin{proof}
Pick $0 < \lambda < (s-\alpha)/2$. Again write $Y_n=\int \mu_n\,d\eta$. Then, by Lemma \ref{lem:large-deviation},
\begin{equation}\label{eq:cor_large_deviation}
\PP(|Y_{n+1}- Y_n| > 2^{-\lambda n}\sqrt{Y_n}) = O\left(\exp(-2^{\Omega(n)})\right)\,.
\end{equation}
By the Borel-Cantelli lemma, $|Y_{n+1}-Y_n|< 2^{-\lambda n}\sqrt{Y_n}$ for all but finitely many $n$, so $Y_n$ converges a.s. to a random variable $Y$. Let $c=\EE(Y_0)=\int \mu_0\, d\eta>0$. Since $Y_n$ is a martingale, $\EE(Y_n)=c$ for all $n$ and therefore, using that $\mu_n(x)\le 2^{\alpha n}$, we get that $\PP(Y_n>\frac{c}{2}) \ge \tfrac{c}{2} 2^{-\alpha n}$. In particular, if $n_0$ is large enough, then recalling \eqref{eq:cor_large_deviation},
\[
\PP(Y_{n_0} > \tfrac{c}{2}) > \sum_{n=n_0}^\infty \PP(|Y_{n+1}- Y_n| > 2^{-\lambda n}\sqrt{Y_n})\,,
\]
which implies (taking $n_0$ suitably large) that $\PP(Y_n>c/4 \text{ for all } n\ge n_0)>0$. This gives the first claim.

For the tail bound, we use Lemma \ref{lem:large-deviation} to conclude that conditional on $Y_n\le M(1-2^{-n\lambda})$, we have
\[
\PP\left(Y_{n+1}-Y_n>2^{(1-n)\lambda}M\right)=O\left(\exp(-M 2^{\Omega(n)})\right)\,.
\]
Summing over all $n\ge 1$ then gives the claim.
\end{proof}

\begin{rem} \label{rem:survives}
In particular, applying the above to $\eta=\mu_0 dx$, we obtain that if $\alpha<d$, then the SI-martingale itself survives with positive probability. The meaning of the hypothesis $\alpha<s$ will be discussed in the next section, after Theorem \ref{thm:Holder-continuity}.
\end{rem}

\section{H\"{o}lder continuity of intersections}
\label{sec:continuity}

In this section we prove the main abstract result of the paper:

\begin{thm} \label{thm:Holder-continuity}
Let $(\mu_n)_{n\in\N}$ be an SI-martingale, and let $\{\eta_t\}_{t\in\Gamma}$ be a family of measures indexed by a metric space $(\Gamma,d)$.
We assume that there are positive constants $\alpha,s,\theta,\gamma_0,C$ such that the following holds:
\begin{enumerate}
\renewcommand{\labelenumi}{(H\arabic{enumi})}
\renewcommand{\theenumi}{H\arabic{enumi}}
\item \label{H:size-parameter-space} For any $\xi>0$, $\Gamma$ can be covered by $\exp(O_\xi(r^{-\xi}))$ balls of radius $r$ for all $r>0$.
\item \label{H:dim-deterministic-measures} The family $\{\eta_t\}$ has Frostman exponent $s$.
\item \label{H:codim-random-measure} Almost surely, $\mu_n(x)\le C\, 2^{\alpha n}$ for all $n\in \N$ and $x\in\R^d$.
\item \label{H:Holder-a-priori} Almost surely, there is a random integer $N_0$, such that
    \begin{equation} \label{eq:Holder-a-priori}
        \sup_{t,u\in\Gamma,t\neq u;n\ge N_0} \frac{\left|Y_n^t-Y_n^u\right|}{2^{\theta n}\,d(t,u)^{\gamma_0}} \le C\,.
    \end{equation}
\end{enumerate}

Further, suppose that $s>\alpha$. Then almost surely $Y_n^t$ converges uniformly in $t$, exponentially fast, to a limit $Y^t$. Moreover, the function $t\mapsto Y^t$ is H\"{o}lder continuous with exponent $\gamma$, for any
\begin{equation}\label{eq:quantitative_gamma}
\gamma < \frac{(s-\alpha)\gamma_0}{s-\alpha+2\theta}\,.
\end{equation}
\end{thm}

We remark that the special case of this theorem in which $(\mu_n)$ is the natural measure on planar fractal percolation, and $\{\eta_t\}$ is the family of length measures on lines making an angle at least $\e>0$ with the axes, was essentially proved by  Peres and Rams \cite{PeresRams14}, and we use some of their ideas.

Before presenting the proof, we make some comments on the hypotheses. Condition \eqref{H:size-parameter-space} says that the parameter space is ``almost'' finite dimensional. In most cases of interest, $\Gamma$ can in fact be covered by $O(r^{-N})$ balls of radius $r$ for some fixed $N>0$ (in other words, $\Gamma$ has finite upper box counting dimension, which clearly implies  \eqref{H:size-parameter-space}).

Hypotheses \eqref{H:dim-deterministic-measures} and \eqref{H:codim-random-measure} say that, in some appropriate sense, the deterministic measures $\eta_t$ have dimension (at least) $s$, and the random measures $\mu$ have dimension (at least) $d-\alpha$. The hypothesis $s>\alpha$ then says that the sum of these dimensions exceeds the dimension $d$ of the ambient space, which is a reasonable assumption if we want these measures to have nontrivial intersection. We will later see that in many examples $s\ge \alpha$ is a necessary condition even for the existence of $Y^t$, and often even $s>\alpha$ is necessary.

The a priori H\"{o}lder condition \eqref{H:Holder-a-priori} may appear rather mysterious: one needs to assume that the functions $Y_n$ are H\"{o}lder, with a constant that is allowed to increase exponentially in $n$, in order to conclude that the $Y_n$ are indeed uniformly H\"{o}lder.  As we will see, geometric arguments can often be used to establish \eqref{H:Holder-a-priori}, making the theorem effective in many situations of interest.

\begin{proof}[Proof of Theorem \ref{thm:Holder-continuity}]
Pick constants $D, B,\lambda,\xi>0$ such that
\begin{align} \label{eq:bound-A}
&\theta/\gamma_0 < D < B\,,\\
\label{eq:bound-lambda} &\lambda<\frac12(s-\alpha-B\xi)\,.
\end{align}
We observe that such choices are possible because $s-\alpha>0$. Also, let
\begin{equation}\label{eq:choice_gamma}
0<\gamma<\min\left(\gamma_0-\frac{\theta}{D}\,, \frac{\lambda}{D}\right)\,.
\end{equation}
Note that the right-hand side is positive thanks to \eqref{eq:bound-A}. By taking $\gamma_0-\frac{\theta}{D}=\frac{\lambda}{D}$ and $\lambda$ arbitrarily close to $(s-\alpha)/2$, we can make $\gamma$ arbitrarily close to the right-hand side of \eqref{eq:quantitative_gamma}. Thus, our task is to show that $Y_n^t$ converges uniformly and exponentially fast to a limit which is H\"{o}lder in $t$ with exponent $\gamma$.

For each $n$, let $\Gamma_n$ be a $(2^{-nB})$-dense family with $\exp(O(2^{B\xi n}))$ elements, whose existence is guaranteed by \eqref{H:size-parameter-space}.

We first sketch the argument. We want to estimate $X_{n+1}$ in terms of $X_n$, where $X_k=\sup_{t\neq u} X_k(t,u)$, and
\[
X_k(t,u) = \frac{|Y_k^t-Y_k^u|}{d(t,u)^\gamma}.
\]
If $d(t,u)\le 2^{-D n}$, we simply use the a priori H\"{o}lder estimate \eqref{eq:Holder-a-priori} to get a deterministic bound. Otherwise, we find $t_0,u_0$ in $\Gamma_n$ such that $d(t,t_0), d(u,u_0)<2^{-B n}$ and estimate
\[
|Y_{n+1}^t-Y_{n+1}^u| \le \I+\II+\III,
\]
where
\begin{align*}
\I &= |Y_n^t-Y_n^u|,\\
\II &= |Y_{n+1}^t-Y_{n+1}^{t_0}|+|Y_n^t-Y_n^{t_0}|+|Y_{n+1}^u-Y_{n+1}^{u_0}|+|Y_n^u-Y_n^{u_0}|,\\
\III &= |Y_{n+1}^{t_0}-Y_n^{t_0}|+|Y_{n+1}^{u_0}-Y_n^{u_0}|.
\end{align*}
The term $\I$ will be estimated inductively, for $\II$ we will use again the a priori estimate \eqref{eq:Holder-a-priori} and to deal with $\III$ we appeal to the fact that almost surely, there is $N_1\in\N$ such that
\begin{equation}\label{eq:bound_Z_n}
\max_{v\in\Gamma_n} |Y_{n+1}^v - Y_n^v| \le 2^{-\lambda n}\max(\overline{Y}_n,1) \quad\text{for all }n\ge N_1\,,
\end{equation}
where $\overline{Y}_n = \max_{t\in\Gamma} Y_n^t$.

We proceed to the details. Our first goal is to verify \eqref{eq:bound_Z_n}.
For a given $v\in\Gamma_n$, we know from Lemma \ref{lem:large-deviation} and our assumptions that
\begin{equation} \label{eq:appl-large-deviation}
\PP(|Y_{n+1}^v - Y_n^v|> 2^{-\lambda n} \sqrt{Y_n^v}) \le O(1)\exp\left(-\Omega(2^{(s-\alpha-2\lambda)n})\right).
\end{equation}
Observe that the application of Lemma \ref{lem:large-deviation} is justified, since  \eqref{eq:kappaeq} holds by \eqref{eq:bound-lambda}. Recalling that $\#\Gamma_n = \exp(O(2^{B \xi n}))$, and using \eqref{eq:bound-lambda}, we deduce from \eqref{eq:appl-large-deviation} that
\begin{align*}
\PP\left(\max_{v\in\Gamma_n} |Y_{n+1}^v - Y_n^v|> 2^{-\lambda n}\overline{Y}_n^{1/2}\right)&=O\left(\#\Gamma_n\right)\exp\left(-\Omega(2^{(s-\alpha-2\lambda)n})\right)\\
&\le\exp\left(O(2^{B\xi n})-\Omega(2^{(s-\alpha-2\lambda)n})\right)\\
&= O\left(\exp\left(-\Omega(2^{c n})\right)\right)\,,
\end{align*}
for $c=s-\alpha-2\lambda>0$. Since $x^{1/2}\le \max(x,1)$, and $\sum_n\exp(-\Omega(2^{c n}))<\infty$, \eqref{eq:bound_Z_n} follows from the Borel-Cantelli lemma.

For the rest of the proof, we fix $N=\max(N_0,N_1)$, where $N_0$ is such that \eqref{eq:Holder-a-priori} holds, and $N_1$ is such that \eqref{eq:bound_Z_n} is valid for all $n\ge N_1$.

If $n\ge N$ and $d(t,u)\le 2^{-D n}$ then, by \eqref{eq:Holder-a-priori},
\begin{equation} \label{eq:estimate-close}
|Y_{n+1}^t-Y_{n+1}^u|\le 2^{(n+1)\theta} d(t,u)^{\gamma_0} \le  O(1) d(t,u)^{\gamma_0-\theta/D} \le O(1) d(t,u)^\gamma\,.
\end{equation}

From now on we consider the case $d(t,u)>2^{-Dn}$. By definition,
\begin{equation} \label{eq:bound-I}
\I \le X_n d(t,u)^\gamma\,.
\end{equation}

Let $t_0,u_0\in\Gamma_n$ be $(2^{-Bn})$-close to $t,u$. Pick $n\ge N$. Using the H\"{o}lder bound \eqref{eq:Holder-a-priori}, we get $|Y_k^t-Y_k^{t_0}|\le 2^{k\theta} 2^{-\gamma_0 B n}$ for $k=n,n+1$, and likewise for $u,u_0$, whence
\begin{equation} \label{eq:bound-II}
\II \le O(1) 2^{-(\gamma_0 B- \theta-\gamma D)n}\, d(t,u)^\gamma\,.
\end{equation}
Note that due to \eqref{eq:bound-A} and \eqref{eq:choice_gamma}, the exponent $\gamma_0 B- \theta-\gamma D$ is positive.

We are left to estimating $\III$. We first claim that
\begin{equation}\label{eq:bound_C_N}
\sup_{n\ge N} \overline{Y}_n = O(\overline{Y}_N+1)<\infty\,.
\end{equation}
Let $n\ge N$. Using \eqref{H:Holder-a-priori} again to estimate $Y_{n+1}^t$ via $Y_{n+1}^{t_0}$, with $t_0\in\Gamma_n$, $d(t,t_0)\le 2^{-B n}$, we have
\begin{align*}
\overline{Y}_{n+1} &\le \left(\max_{v\in\Gamma_n} Y_{n+1}^v\right)+ O(1)2^{(\theta-B\gamma_0)n} \\
&\le \overline{Y}_n + \max(1,\overline{Y}_n) 2^{-\lambda n}+ O(1)2^{(\theta-B\gamma_0)n}\,.
\end{align*}
Recall that we are conditioning on \eqref{eq:bound_Z_n}. Since $\lambda>0$ and  $\theta-B\gamma_0<0$, this implies \eqref{eq:bound_C_N}.

Combining \eqref{eq:bound_Z_n} and \eqref{eq:bound_C_N}, we deduce that $\max_{v\in\Gamma_n} |Y_{n+1}^v - Y_n^v|\le O(\overline{Y}_N+1) 2^{-\lambda n}$ and, in particular,
\begin{equation} \label{eq:bound-III}
\III \le O(\overline{Y}_N+1)\, 2^{-(\lambda-\gamma D) n} \,d(t,u)^\gamma\,,
\end{equation}
where $\lambda-\gamma D>0$ by \eqref{eq:choice_gamma}.

Recapitulating \eqref{eq:bound-I}, \eqref{eq:bound-II} and \eqref{eq:bound-III}, we have shown that there is a constant $\e>0$ such that
\begin{equation} \label{eq:key-Holder-estimate}
|Y_{n+1}^t-Y_{n+1}^u| \le (X_n+ O(\overline{Y}_N+1) 2^{-\e n})d(t,u)^\gamma\quad\text{for all } n\ge N, t, u\in\Gamma\,,
\end{equation}
which immediately yields $\overline{X}:=\sup_n X_n<\infty$.

We are left to show that almost surely $Y_n^t$ converges uniformly, at exponential speed, since then we will have
\[
|Y^t-Y^u| = \lim_{n\to\infty} |Y_n^t-Y_n^u| \le \overline{X} d(t,u)^\gamma\,.
\]
Once again estimating $Y_m^t- Y_n^t$ via $Y_m^{t_0}-Y_n^{t_0}$ and using the a priori estimate \eqref{eq:Holder-a-priori}, we conclude that for all $t$, $\{ Y_n^t\}$ is a uniformly Cauchy sequence with exponentially decreasing differences, and this finishes the proof.
\end{proof}

\begin{rem}
Although condition \eqref{H:size-parameter-space} holds in all of our examples, for the conclusion of Theorem \ref{thm:Holder-continuity} (other than the actual value of the H\"{o}lder exponent) it is enough that  $\Gamma$ can be covered by $\exp(O(r^{-\xi}))$ balls of radius $r$, where $\xi$ satisfies
\begin{equation} \label{eq:relations-parameter}
\gamma_0(s-\alpha) > \xi\theta\,.
\end{equation}
Indeed, the proof works almost verbatim in this case. This allows substantially larger parameter spaces $\Gamma$.

Also, in \eqref{H:codim-random-measure}, we could allow the constant $C$ to be random with a suitably fast decaying tail. Again, as this condition holds with a deterministic $C$ in all our applications, we do not consider this modification here.
\end{rem}

In Section \ref{sec:lower-bound-dim-intersections} we will require a tail estimate for the random variable $\sup_{t\in\Gamma} Y^t$ in the setting of Theorem \ref{thm:Holder-continuity}. This can be easily gleaned from the proof, in terms of a tail estimate for the random variable $N_0$ in \eqref{H:Holder-a-priori}. Although we will  have no use for it, we also provide a tail estimate for the H\"{o}lder constant of $t\mapsto Y^t$.

\begin{cor} \label{cor:tail_estimate}
In the setting of Theorem \ref{thm:Holder-continuity}, let
\begin{align*}
\overline{Y} &= \sup_{t\in\Gamma} Y^t\,,\\
X &= \sup_{t\neq u\in \Gamma} \frac{|Y^t-Y^u|}{d(t,u)^\gamma}\,.
\end{align*}
Then there are constants $C,\delta>0$ such that
\[
\mathbb{P}(\overline{Y}>x), \mathbb{P}(X>x) \le \mathbb{P}(N_0>\log x/C)+ \exp(-x^{\delta})\,.
\]
\end{cor}
\begin{proof}
As in the proof of Theorem  \ref{thm:Holder-continuity}, let $X_n=\sup_{t\neq u} |Y_n^t-Y_n^u| d(t,u)^{-\gamma}$ and $\overline{Y}_n=\sup_{t\in\Gamma} Y_n^t$. It follows from \eqref{eq:key-Holder-estimate} that $X_{n+1}\le X_n+O(\overline{Y}_N)2^{-\e n}$ for all $n\ge N=\max(N_0,N_1)$, where $\e>0$ is a deterministic constant, $N_0$ is such that \eqref{eq:Holder-a-priori} holds, and $N_1$ is such that \eqref{eq:bound_Z_n} is valid for all $n\ge N_1$. Likewise, it follows from \eqref{eq:bound_C_N} that $\overline{Y}=O(\overline{Y}_N+1)$.

Let $Z$ be either $X$ or $\overline{Y}$. We have seen that $Z\le O(X_N+\overline{Y}_N+1) \le O(e^{O(N)})$, where the second inequality is due to the growth condition \eqref{SI:quotients-bounded} in the definition of SI-martingale, and the a priori H\"{o}lder condition \eqref{H:Holder-a-priori}. This implies that, for some constant $C>0$,
\[
\mathbb{P}(Z> e^{CN}) \le \mathbb{P}(N_0>N)+\mathbb{P}(N_1>N)\,.
\]
Recall that $N_1=\max\{ n:E_n \text{ holds}\}$ (or $N_1=1$ if $E_n$ does not hold for any $n$), where
\[
\PP(E_n) = O\left(\exp(-\Omega(2^{cn}))\right) \quad\text{for some }c>0\,.
\]
Hence, $\mathbb{P}(N_1>N)\le O(1)\exp(-\Omega(2^{c N}))$. Taking $x= e^{CN}$ (as we may), this yields the result.
\end{proof}

As noted earlier, in many situations $d-\alpha$ equals the a.s. dimension of the random measure $\mu$, and thus the condition $s>\alpha$ simply means that $\dim \mu+\dim\eta_t>d$. When $\dim\mu+\dim\eta_t<d$, we can no longer expect $Y^{t}_n$ to converge to a continuous (or even finite) limit for most $t$.
In this case, the following variant of Theorem \ref{thm:Holder-continuity} is sometimes useful.

\begin{thm}\label{thm:small_dimension_projections}
Suppose that \eqref{H:dim-deterministic-measures}--\eqref{H:codim-random-measure} hold together with the following condition:
\begin{enumerate}
\renewcommand{\labelenumi}{(H\arabic{enumi})}
\renewcommand{\theenumi}{H\arabic{enumi}}
\setcounter{enumi}{4}
\item\label{H:finite_approx_family}
There is $\theta>0$ such that the following holds. For all $\xi>0$, there exist families $\Gamma_n\subset\Gamma$ with at most $\exp(O_\xi(2^{n\xi}))$ elements, such that for some random variable $N_0\in\N$,
\begin{equation} \label{eq:finite-family-controls-all}
\sup_{t\in\Gamma}Y_n^t\le\sup_{t\in\Gamma_n}Y_n^t+o(2^{\theta n})\quad\text{for all }n\ge N_0\,.
\end{equation}
\end{enumerate}
Suppose further that
\begin{equation}\label{dim_small}
0\le\alpha-s<\theta\,.
\end{equation}
Then, almost surely,
\[
\sup_{n\in\N,\,t\in\Gamma}2^{-\theta n}Y_n^t<\infty\,.
\]
\end{thm}

\begin{proof}
Pick $\xi>0$ such that $\xi<\theta+s-\alpha$ and let $\Gamma_n\subset\Gamma$ be the collections given by \eqref{H:finite_approx_family}. Denote $\overline{Y}_n = \sup_{t\in\Gamma_n} Y_n^t$.
We claim that it is enough to show that
\begin{equation} \label{eq:borel-cantelli}
\sum_{n=1}^\infty \PP\left(\overline{Y}_{n+1}-\overline{Y}_n>\sqrt{2^{\theta n} \overline{Y}_n}\right)
<\infty\,.
\end{equation}
Indeed, if \eqref{eq:borel-cantelli} holds, then almost surely there is $N_1\ge N_0$ such that for $n\ge N_1$ we have
\[
\overline{Y}_{n+1}\le\sup_{t\in\Gamma}{Y}^t_{n+1} \le \overline{Y}_n + \sqrt{\overline{Y}_n 2^{\theta n}}+ o(2^{\theta n})\,,
\]
and this implies $\sup_{t\in\Gamma}Y^{t}_n=O\left(\max(1, \overline{Y}_{N_1}) 2^{\theta n}\right)$
for all $n\ge N_1$ and, in particular,
\[
\sup_{n\in\N, t\in\Gamma}2^{-\theta n}Y_{n}^t<\infty\,.
\]

Therefore fix $n$ and condition on $\mathcal{B}_n$. Pick $t\in\Gamma$.
It follows from
Lemma \ref{lem:large-deviation} that
\[
\PP\left(Y_{n+1}^t - Y_n^t >  \sqrt{ 2^{\theta n} Y_n^t}\right) =O\left(\exp(-\Omega(2^{n(\theta+s-\alpha)}))\right)\,.
\]
Recall that the use of Lemma \ref{lem:HoeffdingJanson} is justified
by \eqref{dim_small}, \eqref{H:dim-deterministic-measures}, \eqref{H:codim-random-measure}.

Applying the above estimate for each $t\in \Gamma_n$, we observe that
\begin{align*}
\PP\left(\overline{Y}_{n+1} - \overline{Y}_n \ge\sqrt{ 2^{\theta n} \overline{Y}_n} \right) &\le \exp\left(O(2^{n\xi})-\Omega(2^{n(\theta+s-\alpha)})\right)\\
&=O\left(\exp\left(-\Omega(2^{n(\theta+s-\alpha)})\right)\right)\,.
\end{align*}
Therefore, \eqref{eq:borel-cantelli} holds and the claim follows.
\end{proof}

\begin{rem}
It is straightforward to check that, together, \eqref{H:size-parameter-space} and \eqref{H:Holder-a-priori} imply \eqref{H:finite_approx_family} for any value of $\theta>0$. Hence, in particular, Theorem \ref{thm:small_dimension_projections} holds under the same assumptions of Theorem \ref{thm:Holder-continuity} (other than the sign of $\alpha-s$). However, there are some important examples in which \eqref{H:finite_approx_family} can be checked but \eqref{H:Holder-a-priori} fails: this is the case where $\mu_n$ is constructed by subdivision inside a polyhedral grid, and $\eta_t$ are Hausdorff measures on $k$-planes. Also in this case, \eqref{H:finite_approx_family} holds for any $\theta>0$. See Remark \ref{rem:dim-of-projs} (ii).
\end{rem}

\section{Classes of spatially independent martingales}
\label{sec:examples}

\subsection{Cut-out measures arising from Poisson point processes}
\label{subsec:Poissonian}

In this section we describe several classes of examples of spatially independent martingales. We start with cut-out measures driven by Poissonian point processes. This class generalizes the example in the introduction, in which balls generated by a Poisson point process were removed; the generalization consists in replacing balls by more general sets. In one dimension, the model essentially goes back to Mandelbrot \cite{Mandelbrot72}. We adapt our definition from \cite{NacuWerner11}. The Hausdorff dimension of these and related cut-out sets was calculated in \cite{Zahle84, Fitzsimmonsetal85, Rivero03, Thacker06, NacuWerner11}. The connectivity properties of Poissonian cut-outs were recently investigated in \cite{BromanCamia10}.

The following discussion is adapted from \cite[Section 2]{NacuWerner11}. Let $\mathcal{X}$ denote the class of all compact subsets of $\R^d$, endowed with the Hausdorff metric and the associated Borel $\sigma$-algebra. We consider (infinite but $\sigma$-finite) Borel measures $\mathbf{Q}$ on $\mathcal{X}$ satisfying the following properties for all Borel sets $\mathcal{A}\subset\mathcal{X}$:
\begin{enumerate}
\renewcommand{\labelenumi}{(M\arabic{enumi})}
\renewcommand{\theenumi}{M\arabic{enumi}}
\item \label{M:translation-invariance} $\mathbf{Q}$ is translation invariant, i.e. $\mathbf{Q}(\mathcal{A}) = \mathbf{Q}(\{\Lambda+t:\Lambda\in\mathcal{A}\})$ for all $t\in\mathbb{R}$.
\item \label{M:scale-invariance} $\mathbf{Q}$ is scale invariant, i.e. $\mathbf{Q}(\mathcal{A}) = \mathbf{Q}(\{s \Lambda:\Lambda\in\mathcal{A}\})$ for all $s>0$, where $s A = \{ s x: x\in A\}$.
\item \label{M:locally-finite} $\mathbf{Q}$ is locally finite, meaning that the $\mathbf{Q}$-measure of the family of all sets of diameter larger than $1$ that are contained in $[-1,1]^d$ is finite.
\item \label{M:zero-boundary} $\leb^d(\partial\Lambda)=0$ for $\mathbf{Q}$-almost all $\Lambda\in\mathcal{X}$.
\end{enumerate}

The following is proved in \cite[Lemma 1]{NacuWerner11}. (In \cite{NacuWerner11} the support of these distributions are curves rather than sets, but this does not make any difference in the proof of the claim. Also, in \cite{NacuWerner11} it is assumed that $d=2$, but the proof works in any dimension.)

\begin{lemma} \label{lem:structure-scale-translation-invariant-distr}
Write $\sigma=s^{-1}ds$ on $(0,\infty)$.
\begin{enumerate}
\item[\emph{(i)}] If $\mathbf{Q}_0$ is any measure on $\mathcal{X}$ such that $\int \diam(\Lambda)^d d\mathbf{Q}_0(\Lambda)<\infty$, then the distribution $\mathbf{Q}$ obtained as the push down of $\leb^d\times \sigma\times\mathbf{Q}_0$ under $(t,s,\Lambda)\mapsto s(\Lambda+t)$, satisfies the properties \eqref{M:translation-invariance}--\eqref{M:locally-finite}.
\item[\emph{(ii)}] Conversely, any $\mathbf{Q}$ satisfying \eqref{M:translation-invariance}--\eqref{M:locally-finite} can be obtained in this way starting from a corresponding $\mathbf{Q}_0$, which can also be required to be finite and supported on sets of diameter $1$.
\end{enumerate}
\end{lemma}

As a simple example, we can take $\mathbf{Q}_0= r\delta_\Lambda$ for some $r>0$ and $\Lambda\in \mathcal{X}$. If $\Lambda$ is a ball, then the resulting distribution $\mathbf{Q}$ is, additionally, rotationally invariant. We note that the constant $r$ plays a crucial role in the geometry of the cut-out set, to be defined next.

Given a distribution $\mathbf{Q}$
as above, we construct a Poisson point process with intensity $\mathbf{Q}$. We recall that this is a random countable collection $\mathcal{Y}= \{ \Lambda_j\}$ of sets in $\mathcal{X}$ satisfying the following properties:
\begin{enumerate}
\renewcommand{\labelenumi}{(PPP\arabic{enumi})}
\renewcommand{\theenumi}{PPP\arabic{enumi}}
\item \label{PPP:Poisson} For each Borel set $\mathcal{A}\subset \mathcal{X}$, the random variable $\#(\mathcal{A}\cap\mathcal{Y})$ (i.e. the number of elements in $\mathcal{A}\cap\mathcal{Y}$) has Poisson distribution with expectation $\mathbf{Q}(\mathcal{A})$. (If $\mathbf{Q}(\mathcal{A})=\infty$, then $\mathcal{A}\cap\mathcal{Y}$ is infinite almost surely.)
\item \label{PPP:independence} If $\{\mathcal{A}_i\}$ are pairwise disjoint Borel subsets of $\mathcal{X}$, then the random variables $\#(\mathcal{A}_i\cap \mathcal{Y})$ are independent.
\end{enumerate}

We can then form the random cut-out set
\[
\R^d\setminus \bigcup_{\diam(\Lambda_j)\le 1}\Lambda_j\,.
\]
Note that without the restrictions on the diameters, the limit set would be a.s. empty. Write $\PP$ for the induced distribution of $A$.

We note that $\PP$ inherits translation invariance from $\mathbf{Q}$  (it is not fully scale-invariant since we are imposing an upper bound on the removed sets). If $\mathbf{Q}$ is rotation invariant, then so is $\PP$. We contrast this with subdivision fractals such as fractal percolation, which have only very limited scale, translation and rotational symmetries: those arising from the filtration they are defined on (and this only in the most homogeneous cases, such as fractal percolation with a uniform parameter $p$).

If $\PP(\cup \{\Lambda_j\,:\,\diam(\Lambda_j)\le 1\}\neq\R^d)>0$ then, conditional on the cut-out set being nonempty, it is almost surely unbounded (this can be seen from \eqref{M:translation-invariance}). Hence in order to make this model fit into the framework of Section \ref{sec:setting}, we need to restrict attention to a bounded domain. This is enough to obtain an SI-martingale.

Let $\Delta_{a}^b$ denote the family of sets $\Lambda\in\mathcal{X}$ with diameter $a\le\diam(\Lambda)< b$. Moreover, for $\mathcal{Y}\subset\mathcal{X}$, let $\Delta_{a}^b(\mathcal{Y})$ be the union of the sets in $\Delta_{a}^b\cap\mathcal{Y}$.

\begin{lemma} \label{lem:Poisson-satisfies-properties}
Let $\Omega\subset\R^d$ be a bounded set. Let $\mathbf{Q}$ be a distribution satisfying \eqref{M:translation-invariance}--\eqref{M:locally-finite}.

Define $0\le \alpha=\alpha(\mathbf{Q})<\infty$ as
\begin{equation} \label{eq:def-alpha}
2^{-\alpha}= \PP(0\notin\Delta_{1/2}^1(\mathcal{Y}))\,,
\end{equation}
where $\mathcal{Y}$ is a realization of a Poisson point process with intensity $\mathbf{Q}$, and let $A_n= \Omega\setminus\Delta_{2^{-n}}^1(\mathcal{Y})$ (in particular, $A_0=\Omega$).

Then the sequence $\mu_n(x):=2^{n\alpha} \mathbf{1}_{A_n}(x)$ satisfies \allbasicconds.
\end{lemma}
\begin{proof}
Properties \eqref{SI:bounded}
and \eqref{SI:quotients-bounded} are clear. Note that $\alpha$ is well defined thanks to \eqref{M:locally-finite}.
The martingale property \eqref{SI:martingale} holds with $\mathcal{B}_n$ equal to the $\sigma$-algebra generated by $\Delta_{2^{-n}}^1(\mathcal{Y})$ since, for any $x\in\Omega$,
\begin{align*}
\EE(\mu_{n+1}(x)\,|\,\mathcal{B}_n)=2^{\alpha}\mu_n(x)\PP(x\notin\Delta_{2^{-(n+1)}}^{2^{-n}}(\mathcal{Y}))=\mu_n(x)\,,
\end{align*}
using the scale and translation invariance of the distribution $\mathbf{Q}$.

Finally, to verify \eqref{SI:spatial-independence}, note that if $\{ B_i\}$ is a collection of Borel sets with
\[
\dist(B_i,B_j)>3 \times 2^{-n}\quad\text{ for all }i\neq j\,,
\]
then the random variables $\leb^d(A_{n+1}\cap B_i|\mathcal{B}_n)$ are independent by \eqref{PPP:independence}, since the collections
\[
\mathcal{X}_i = \{ \Lambda\in\mathcal{Y}:\diam(\Lambda)\le 2^{-n}, \Lambda\cap B_i\neq\varnothing \}
\]
are disjoint.
\end{proof}

\begin{defn}
If $(\mu_n)$ is a sequence as in Lemma \ref{lem:Poisson-satisfies-properties}, we will say that $(\mu_n)$ is a \textbf{Poissonian cutout martingale} and that $(A_n)$ is a \textbf{Poissonian cutout set}. Sometimes we will abuse notation and refer in this way to the limits $\mu_\infty, A:=\overline{\cap_{n\in\N}A_n}$. We take the closure in the definition of $A$ to ensure that $\supp\mu_{\infty}\subset A$. The class of Poissonian cutout martingales will be denoted $\poissonian$.
\end{defn}

The following is immediate from the definition of $\alpha$ and the translation and scale-invariance of $\mathbf{Q}$:
\begin{lemma} \label{lem:prob-x-in-An}
For any $x\in\Omega$, $\PP(x\in A_n) = 2^{-\alpha n}$ and, more generally, $\PP(x\notin \Delta_{2^{-n}}^{2^{-m}}(\mathcal{Y}))= 2^{(m-n)\alpha}$ if $n\ge m$.
\end{lemma}

Recall from Lemma \ref{lem:structure-scale-translation-invariant-distr} that any $\mathbf{Q}$ satisfying \eqref{M:translation-invariance}--\eqref{M:locally-finite} can be obtained from a  finite measure $\mathbf{Q}_0$. The following explicit expression for $\alpha$ will be useful later.

\begin{lemma}\label{lem:alpha}
\begin{equation}\label{eq:alpha}
\alpha=\int \leb^d(\Lambda) \,d\mathbf{Q}_0(\Lambda)\,.
\end{equation}
\end{lemma}
\begin{proof}
Let $\mathcal{A}\subset\mathcal{X}$ be the compact sets containing $0$ of diameter in $[1/2,1)$. Then, using Lemma \ref{lem:structure-scale-translation-invariant-distr} and Fubini,
\begin{align*}
\mathbf{Q}(\mathcal{A}) &= \bigl(\leb^d\times\sigma\times\mathbf{Q}_0\bigr)\bigl\{ (t,s,\Lambda): 0\in s(\Lambda+t), \diam(s(\Lambda+t))\in [1/2,1)\bigr\}\\
&= \bigl(\leb^d\times\sigma\times\mathbf{Q}_0\bigr)\bigl\{ (t,s,\Lambda): -t\in\Lambda, s\in [(2\diam(\Lambda))^{-1},\diam(\Lambda)^{-1}) \bigr\}\\
&= \int \sigma\left([(2\diam(\Lambda))^{-1},\diam(\Lambda)^{-1})\right) \leb^d(\Lambda) \,d\mathbf{Q}_0(\Lambda)\\
&= \log(2) \int \leb^d(\Lambda) \,d\mathbf{Q}_0(\Lambda)\,.
\end{align*}
Equation \eqref{eq:alpha} follows since, by \eqref{PPP:Poisson},
\[
2^{-\alpha} = \PP(0\notin\Delta_{1/2}^1(\mathcal{Y})) = \PP(\#(\mathcal{A}\cap\mathcal{Y})=0) = \exp\left(-\log(2) \int \leb^d(\Lambda) \,d\mathbf{Q}_0(\Lambda)\right).
\]
\end{proof}

\begin{rem}\label{rem:alpha}
The following generalization of Lemma \ref{lem:alpha} will be used in the proof of Lemma \ref{lem:bound_A_n_nbhd} and also later in Sections \ref{sec:dim_of_intersections} and \ref{sec:lower-bound-dim-intersections}.
Suppose that we are given a Borel map $\Lambda\mapsto\Lambda'$, $\mathcal{X}_1\to\mathcal{X}$, where $\mathcal{X}_1$ are the compact sets of unit diameter. Let $\mathbf{Q}_0$ be supported on $\mathcal{X}_1$. For a realization $\mathcal{Y}=\{s_i(\Lambda_i+t_i)\}$ of the Poisson process as in Lemma \ref{lem:structure-scale-translation-invariant-distr} (i), we define 
\[A'_n=\R^d\setminus
\bigcup_{2^{-n}\le s_i<1}s_i((\Lambda_i)'+t_i)\,.\]
Then for all $x\in\R^d$, $n\in\N$, we have
\begin{equation*}
\PP\left(x\in A'_n\right)=2^{-\beta n}\,,
\end{equation*}
where 
\begin{equation}\label{eq:alpharoo}
\beta=\int \leb^d(\Lambda') \,d\mathbf{Q}_0(\Lambda)\,.
\end{equation}
This follows exactly as in the proof of Lemma \ref{lem:alpha} using Lemma \ref{lem:structure-scale-translation-invariant-distr} and \eqref{PPP:Poisson}-\eqref{PPP:independence}.
\end{rem}

Under the assumptions \eqref{M:translation-invariance}--\eqref{M:zero-boundary}, it holds that if $\alpha>d$, then $A=\emptyset$ almost surely, while if $\alpha< d$, then $\mu_\infty=\lim \mu_n$ is nontrivial with positive probability and, moreover,
\[
\dim_H(A)=\dim_B(A)=d-\alpha \quad\text{almost surely on } \mu_\infty\neq 0\,.
\]
This formula was obtained in some special cases in \cite{Thacker06} and \cite{NacuWerner11}, and the general case can be proved using similar ideas. We provide the details for the reader's convenience.

\begin{thm}\label{thm:dim_cut_out_set}
Let $\Omega$ be any bounded set with $\leb^d(\Omega)>0$, and let $\mathbf{Q}$ be a distribution satisfying \eqref{M:translation-invariance}--\eqref{M:zero-boundary} with $0\le\alpha< d$.
Then almost surely $\dim_B(A)\le d-\alpha$. Moreover, $\mu_\infty\neq 0$ with positive probability and, almost surely conditioned on $\mu_\infty\neq 0$,
\begin{enumerate}
\item $\dimloc(\mu_\infty,x)= d-\alpha$ for $\mu_\infty$-almost all $x\in A$, and
\item $\dim_H(A)=\dim_B(A)=d-\alpha$.
\end{enumerate}
\end{thm}

We first prove a lemma.

\begin{lemma}\label{lem:bound_A_n_nbhd}
Under the assumptions of Theorem \ref{thm:dim_cut_out_set}, given $\e>0$ there is $C\ge 1$ (independent of $n$) such that
\begin{equation}\label{eq:A_n_bound}
\PP\left(x\in A(2^{-n})\right)\le C\, 2^{-(\alpha-\e) n}\text{ for all }x\in\Omega(2^{-n}), n\in\N\,.
\end{equation}
Moreover, there is $C\ge 1$ such that, for all $x,y\in\Omega$ and all $n$,
\begin{equation}\label{eq:2_moment_bound}
\PP\left(x,y\in A_n\right)\le C \,2^{-2\alpha n}|x-y|^{-\alpha}\,.
\end{equation}
\end{lemma}

\begin{proof}
Given $\Lambda\in\mathcal{X}$ and $\delta>0$, let
\[
\Lambda^\delta = \Lambda \cap \{ x:\dist(x,\partial \Lambda) \ge \delta \diam(\Lambda)\}\,.
\]
Set
\[
B_n^\delta = \Omega(2^{-n})\setminus \bigcup_{\Lambda\in\Delta_{2^{-n}}^1\cap \mathcal{Y}}\Lambda^{\delta}\,.
\]
Then $A((\delta/2) 2^{-n})\subset A_n(\delta 2^{-n}) \subset B_n^\delta$ by definition. On the other hand, from Remark \ref{rem:alpha} we infer that $\PP(x\in B_n^\delta) = 2^{-\beta n}$ for each $x\in\Omega(2^{-n})$, where
\[\beta=\int \leb^d(\Lambda^\delta) \,d\mathbf{Q}_0(\Lambda)\,.\]

It follows from \eqref{M:zero-boundary} that $\leb^d(\Lambda^\delta)\to\leb^d(\Lambda)$ as $\delta\to 0$, for $\mathbf{Q}$ almost all $\Lambda$. By \eqref{eq:alpha}, \eqref{eq:alpharoo} and monotone convergence, we observe that $\beta\to\alpha$ as $\delta\to 0$. This yields the first claim.

For the second claim, fix $x,y\in\Omega$ and let $m\in\N$ such that $2^{-m}< |x-y|\le 2^{-m+1}$.
If $m>n$,
the claim is a direct conclusion of Lemma \ref{lem:prob-x-in-An}. For $m\le n$, we have
\[
\{x,y\in A_n\}\subset\{x\in A_m\}\cap\{x\notin\Delta_{2^{-n}}^{2^{-m}}(\mathcal{Y})\}\cap\{y\notin\Delta_{2^{-n}}^{2^{-m}}(\mathcal{Y})\}
\]
and these three events are independent. Together with Lemma \ref{lem:prob-x-in-An}, this gives the claim.
\end{proof}

\begin{proof}[Proof of Theorem \ref{thm:dim_cut_out_set}]
The fact that $\mu_\infty\neq 0$ with positive probability follows from Remark \ref{rem:survives}. Fix $\e>0$. We prove the remaining claims by verifying that there is $C>0$, such that
almost surely
\begin{equation}\label{eq:udimb}
\leb^d(A(2^{-n}))\le C\,2^{-n(\alpha-2\e)}\text{ for large enough }n
\end{equation}
and that almost surely, for $\mu_\infty$-almost all $x\in A$,
\begin{equation}\label{eq:ldimsmall}
\mu_\infty(B(x,r))\le r^{d-\alpha-\e}\text{ for sufficiently small }r>0\,.
\end{equation}
Equation \eqref{eq:udimb} implies that $\dim_H(A)\le\dim_B(A)\le d-\alpha$ and, moreover,
\[
\udimloc(\mu_\infty,x)\le d-\alpha \text{ for } \mu_\infty\text{-almost all } x\in A\,,
\]
whereas \eqref{eq:ldimsmall} gives the desired lower bounds.

Using Fubini and \eqref{eq:A_n_bound}, we have
\begin{align*}
\EE(\leb^d(A(2^{-n})))=\int_{x\in\Omega}\PP(x\in A(2^{-n}))\,dx =O\left(2^{-(\alpha-\e) n}\right)\,.
\end{align*}
Thus
\[
\EE\left(\sum_{n=1}^\infty\leb^d(A(2^{-n})) 2^{n(\alpha-2\e)}\right)=O\left( \sum_{n}2^{-n\e}\right)<\infty\,.
\]
Equation \eqref{eq:udimb} follows at once from Borel-Cantelli.

To prove \eqref{eq:ldimsmall}, we first estimate
\begin{align*}
&\EE\left(\int\int |x-y|^{\alpha-d+\e}\,d\mu_\infty(x)d\mu_\infty(y)\right)\\
&\le\EE\left(\lim_{n\rightarrow\infty}\int\int |x-y|^{\alpha-d+\e}\,d\mu_n(x)d\mu_n(y)\right)\\
&\le\liminf_{n\rightarrow\infty}\EE\left(\int\int |x-y|^{\alpha-d+\e}\,d\mu_n(x)d\mu_n(y)\right)& \text{(by Fatou)}\\
&=\liminf_{n\rightarrow\infty}2^{2\alpha n}\,\EE\left(\int_{A_n}\int_{A_n} |x-y|^{\alpha-d+\e}\,dx dy\right)\\
&=\liminf_{n\rightarrow\infty}2^{2\alpha n}\int_{\Omega}\int_{\Omega}|x-y|^{\alpha-d+\e}\PP\left(x,y\in A_n\right)\,dx dy&\text{(by Fubini)}\\
&\le\int_{\Omega}\int_{\Omega}|x-y|^{-d+\e}\,dx dy<\infty\,.&\text{(by \eqref{eq:2_moment_bound})}
\end{align*}
Thus, a.s. we have the energy estimate $\int\int |x-y|^{\alpha-d+\e}\,d\mu_\infty(x)d\mu_\infty(y)<\infty$ a.s. and, reducing $\e$ slightly, this implies \eqref{eq:ldimsmall}.
\end{proof}

\begin{rem}\label{rem:boxdimestimate}
It follows from \eqref{eq:alpha} that if instead of starting with $\mathbf{Q}_0$ we start with $r\mathbf{Q}_0$, $r>0$, then the corresponding $\alpha$ gets multiplied by $r$. Therefore by considering the family $\{r\mathbf{Q_0}, r>0\}$ we range over all possible dimensions of $A$ and $\mu_\infty$.
\end{rem}

\begin{rem}
Although it seems very plausible, at least under some assumptions on the measure $\mathbf{Q}$, we do not know if $\mu_\infty\neq 0$ almost surely conditioned on $A\neq\varnothing$.
\end{rem}

We give some concrete examples for future reference.

\begin{ex}\label{ex:balltype}\textbf{(Ball type cut-out measures.)}
Let $c_d=\leb^d(B^d(0,\tfrac12))$.
Fix $0<r<d/c_d$ and let $(\mu_n)$ be the random sequence generated by $\mathbf{Q}_0=r\delta_{B(0,1/2)}$, starting from the seed $\Omega=B(0,1)$, see Figure \ref{fig:ball_type}. Then the removed sets are balls, and we have $\dim_B(A)=\dim_H(A)=\dim(\mu_\infty)=d-\alpha$, almost surely on non-degeneracy of $\mu_\infty$, where  $\alpha=c_d r$.

We will denote this martingale by $\mu_n^{\text{ball}}$, or by $\mu_n^{\text{ball}(\alpha,d)}$ when we want to emphasize the value of $\alpha$ and the ambient dimension.
\end{ex}

\begin{ex}\label{ex:snowflake}\textbf{(Snowflake type cut-out measures.)}
Let $\Lambda_0\subset\R^2$ be the Von-Koch snowflake domain (see e.g. \cite[p. xix]{Falconer03} for the construction of the snowflake curve and Figure \ref{fig:tiling} for a picture) scaled so that $\diam(\Lambda_0)=1$. Let $c=\leb^d(\Lambda_0)$.

Let $(\mu_n)$ be the random sequence generated by $\mathbf{Q}_0=r\delta_{\Lambda_0}$, starting from the seed $\Omega=\Lambda_0$ with a parameter $0<r<2/c$. Then a.s. all the removed sets are scaled and translated copies of $\Lambda_0$ and we have $\dim_B(A)=\dim_H(A)=\dim(\mu_\infty)=d-\alpha$, almost surely on non-degeneracy of $\mu_\infty$, where  $\alpha=r c$.

We will denote this martingale by $\mu_n^{\text{snow}}$, or by $\mu_n^{\text{snow}(\alpha)}$ when we want to emphasize the value of $\alpha$.
\end{ex}

\begin{rem}
The above snowflake example is not rotationally invariant, but we can obtain a measure $\mathbf{Q}_0$ (and hence $\mathbf{Q}$) which is rotationally invariant by starting with $\Lambda_0$, and rotating it by a uniformly random angle. Then the law of the random measures $\mu_n$ is rotationally invariant as well. The Brownian loop soup is another class of planar $\poissonian$ (see \cite{NacuWerner11}) which is even conformally invariant,
\end{rem}

We finish this section by introducing a useful class of sequences $(\mu_n)$ which contains $\poissonian$ as a special case.

\begin{defn} \label{def:cutout-type}
We say that a sequence $(\mu_n)$ of random densities on $\R^d$ is of \textbf{$(\mathcal{F},\alpha,\zeta)$-cutout type}, where $\alpha,\zeta>0$ and $\mathcal{F}$ is a family of Borel sets $\Lambda\subset B(0,R)\subset \R^d$ (for some large $R>0$), if there are $C>0$ and a set $\Omega\in\mathcal{F}$ (the \textbf{seed} of the construction), such that
\begin{equation} \label{eq:cutout-type}
\mu_n = 2^{\alpha n}\mathbf{1}_{A_n}\,,
\end{equation}
where $A_n=\Omega\setminus \cup_{j=1}^{M_n} \Lambda_j^{(n)}$
for some random subset $\{ \Lambda_j^{(n)} \}_{j=1}^{M_n}$ of $\mathcal{F}$, and, moreover, there is a finite random variable $N_0$ such that $M_n\le C\, 2^{\zeta n}$ for all $n\ge N_0$.
\end{defn}

Clearly, if $\mu_n$ is an SI-martingale of $(\mathcal{F},\alpha,\zeta$)-cutout type, then $\mu_n(x)\le 2^{\alpha n}$ for all $n,x$. In other words, \eqref{H:codim-random-measure} holds. Let us see that, indeed, $\poissonian$ martingales are of cutout type.

\begin{lemma}\label{lem:Poisson_is_cut_out_type}
Let $(\mu_n)\in\poissonian$ with associated measure $\mathbf{Q}$, initial domain $\Omega$, and parameter $\alpha$. Then $(\mu_n)$ is of $(\mathcal{F},\alpha,d)$-cutout type for any $\mathcal{F}\subset\mathcal{X}$ that contains $\Omega$ and $\mathbf{Q}$-almost all $\Lambda\in\mathcal{X}$ such that $\diam(\Lambda)\le 1$ and $\Lambda\cap\Omega\neq\emptyset$.

Moreover, the random variable $N_0$ in the definition of cutout type satisfies
\[
\PP(N_0>N) \le 2\exp\left(-2^{Nd}\right)\,.
\]
for some $C,\delta>0$.
\end{lemma}
\begin{proof}

Write $\mathcal{X}_\Omega$ for the sets in $\mathcal{X}$ with $\diam(\Lambda)\le 1$ that hit $\Omega$, and pick $\mathcal{F}\subset\mathcal{X}_\Omega$ such that $\Omega\in\mathcal{F}$ and $\mathbf{Q}(\mathcal{X}_\Omega\setminus\mathcal{F})=0$. Then the decomposition \eqref{eq:cutout-type} in the definition of cutout type is clear, with
\[
M_n = \#\left(\Delta_{2^{-n}}^{1}\cap \mathcal{Y}\cap \mathcal{X}_\Omega\right)\,.
\]
By \eqref{PPP:Poisson}, $M_n$ is a Poisson variable with mean $\lam_n:=\mathbf{Q}(\Delta_{2^{-n}}^{1}\cap\mathcal{F})$. Recalling the decomposition given by Lemma \ref{lem:structure-scale-translation-invariant-distr}(ii), it is easy to check that
\begin{align*}
\lam_n &\le \mathbf{Q}_0(\mathcal{X})(\sigma\times\leb^d)\{(s,t): 2^{-n}\le s<1, B(st, s)\cap\Omega\neq\varnothing\}\\
&= \int_{2^{-n}}^1 O_{\mathbf{Q},\Omega}(s^{-d}) d\sigma(s)\\
&= O_{\mathbf{Q},\Omega}(2^{dn})\,.
\end{align*}
Since Poisson variables have subexponential tails, there is a constant $C=C(\mathbf{Q},\Omega)>0$ such that
\begin{equation}\label{eq:M_n_bound}
\PP(M_n > C 2^{dn}) \le \exp(-2^{dn})\text{ for all }n.
\end{equation}
Let $E_n$ be the event that $M_n\le  C \,2^{nd}$. By the Borel-Cantelli lemma, almost surely there exists $N_0$ such that $E_n$ holds for all $n\ge N_0$. Moreover,
\[
\mathbb{P}(N_0\ge N) \le 2 \exp(-2^{Nd})\,,
\]
and this completes the proof.
\end{proof}

In practice, the family $\mathcal{F}$ will be clear from context. For example, for $\mu_n^{\text{ball}}$, $\mathcal{F}$ is the collection of all balls of radius $O(1)$, and for $\mu_n^{\text{snow}}$ it consists of homothetic copies of the snowflake domain.

\subsection{Subdivision fractals}
\label{sec:subdivision}

Our next class of examples includes many generalizations of fractal percolation and more general Mandelbrot's multiplicative cascades (see e.g. \cite{Mandelbrot74}, or \cite{Barral00} and references therein). Unlike standard percolation, we will allow certain dependencies and less homogeneity. We start by defining nested families which generalize the dyadic family. Let $F_0$ be a bounded set, which we call the {\em seed} of the construction, and let $\mathcal{F}_n$ be an increasing sequence of finite, atomic $\sigma$-algebras on $F_0$, with $\mathcal{F}_0=\{ F_0,\emptyset\}$. In the sequel we identify $\mathcal{F}_n$ with the collection of its atoms, which we assume to be Borel sets. Write $\mathcal{F}=\cup_{n=0}^\infty \mathcal{F}_n$.

The next lemma describes our general construction.

\begin{lemma}\label{lem:subdivision_satisfies_properties}
Let $\{ \mathcal{F}_n\}_{n=0}^\infty$ be an increasing filtration of finite $\sigma$-algebras as above. Let $\{ W_F\}_{F\in\mathcal{F}}$ be random variables such that the following holds for some (deterministic) $C\ge 1$.
\begin{enumerate}
\renewcommand{\labelenumi}{(SD\arabic{enumi})}
\renewcommand{\theenumi}{SD\arabic{enumi}}
\item \label{SD:measurability} The law of $W_F, F\in\mathcal{F}_{n+1}$ is measurable with respect to the $\sigma$-algebra $\mathcal{B}_{n}$ generated by $\{ W_F\}_{F\in\mathcal{F}_k,k\le n}$.
\item\label{SD:controlled-growth} Almost surely $W_F\in [0,C]$ for all $F\in\mathcal{F}$.
\item\label{SD:martingale} $\EE(W_F|\mathcal{B}_n)=1$ for all $F\in\mathcal{F}_{n+1}$.
\item\label{SD:independent} If $\{ F_j\}\subset \mathcal{F}_{n+1}$ and the $F_j$ are subsets of pairwise disjoint cells in $\mathcal{F}_n$, then $W_{F_j}|\mathcal{B}_n$ are independent.
\item\label{SD:finite_elements} For all $n$ and all $F\in\mathcal{F}_n$, there is a ball $B\subset F$ such that
$C^{-1} 2^{-n}\le \diam(B)\le \diam(F) \le C 2^{-n}$.
\end{enumerate}

We define a sequence $(\mu_n)$ as follows. Let $\mu_0=\mathbf{1}[F_0]$. For $n\ge 0$, set
\[
\mu_{n+1}(x)=W_F \mu_n(x)\,,
\]
where $F\in\mathcal{F}_{n+1}$ is the partition element containing $x$.

Then $(\mu_n)$ is an SI-martingale.
\end{lemma}
\begin{proof}
The routine verification is left to the reader.
\end{proof}

\begin{defn} \label{def:subdivision}
The martingales $(\mu_n)$ satisfying \eqref{SD:measurability}--\eqref{SD:finite_elements} above will be called \textbf{subdivision martingales}, and the class of all of them will be denoted $\subdivision$.
\end{defn}

Multiplicative cascades correspond to the case in which the $W_F$ are IID, and fractal percolation to the special case of this in which $W_F$ takes the values $0,p^{-1}$ with probabilities $1-p,p$. Subdivision martingales also include measures supported on inhomogeneous fractal percolation: here the $W_F$ are independent but not identically distributed, and $W_F$ takes values in $\{ 0, p_F^{-1}\}$ for some $p_F\in (0,1)$. We note that in all these cases we allow more general filtrations than the usual definitions (which are tied to the dyadic or $M$-adic grid).

\begin{defn}
Fractal percolation on the dyadic grid with survival probability $p$ will be denoted by $\mu_n^{\text{perc}}$ or, if we want to emphasize the parameter and the ambient dimension, by $\mu_n^{\text{perc}(\alpha,d)}$, where $\alpha=-\log_2 p$. See Figure \ref{fig:perco}.
\end{defn}

Fractal percolation in an arbitrary base $M\ge 2$ also fits into our context; we only have to ``stop'' the construction at certain steps so that after $n$ steps the cubes have size comparable to $2^{-n}$. We could also allow bases other than $2$ in the definition of SI-martingale, but as a similar stopping time construction allows to pass between different bases, we see no point in doing so.

In general, the boundaries of the elements in $\mathcal{F}_n$ may be very irregular. The following example in the plane has very similar geometric properties to $\mu_n^{\text{snow}}$.

\begin{ex} \label{ex:fractal-tiling}
Let $\Lambda_0\subset\R^2$ be the closure of the Von Koch snowflake domain scaled to have diameter one.
It is well known that $\Lambda_0$ is self-similar. More precisely, $\Lambda_0=\cup_{i=1}^7\Lambda_i$, where $\Lambda_1,\ldots,\Lambda_7$ are scaled and translated copies of $\Lambda_0$ (with contraction ratio $1/3$ for $i=1,\ldots,6$ and $1/\sqrt{3}$ for $\Lambda_7$), and the interiors of $\Lambda_1,\ldots,\Lambda_7$ are disjoint. See Figure \ref{fig:tiling}. Write $\Lambda_i = S_i \Lambda_0$ for suitable homotheties $S_i$, $i\in\{1,\ldots,7\}$. Inductively, the sets $S_{i_1}\cdots S_{i_n}(\Lambda_0)$ also tile $\Lambda_0$ for each $n$, but these families do not satisfy \eqref{SD:finite_elements}. Instead, we set
\[
\mathcal{F}_n = \left\{ S_{i_1}\cdots S_{i_k}(\Lambda_0): \diam(S_{i_1}\cdots S_{i_k}(\Lambda_0))\le 2^{-n}<\diam(S_{i_1}\cdots S_{i_{k-1}}(\Lambda_0))\right\}\,.
\]

Given $0<p<1$, we can define fractal percolation for this filtration, as explained above. Note that in order for \eqref{SI:martingale} to hold also on the boundaries of the sets $S_{i_1}\cdots S_{i_k}(\Lambda_0)$, we have to delete (parts of the) boundaries in a suitable way, so that each boundary point of some $F\in\mathcal{F}_n$ belongs to exactly one of the sets in $\mathcal{F}_n$. The way in which we do this is of no consequence in any of our later applications. We denote this SI-martingale by $\mu_n^{\text{snowtile}}$ or $\mu_n^{\text{snowtile}(\alpha)}$ for $\alpha=-\log_2 p$.
\end{ex}

\begin{figure}
   \centering
\resizebox{0.9\textwidth}{!}{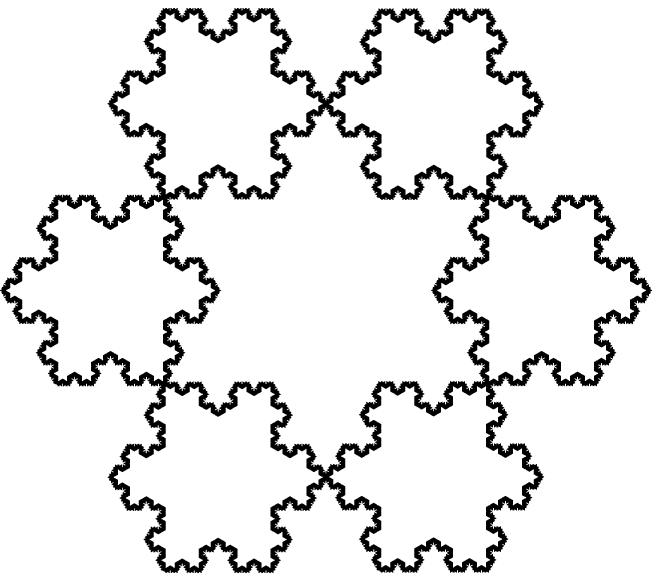
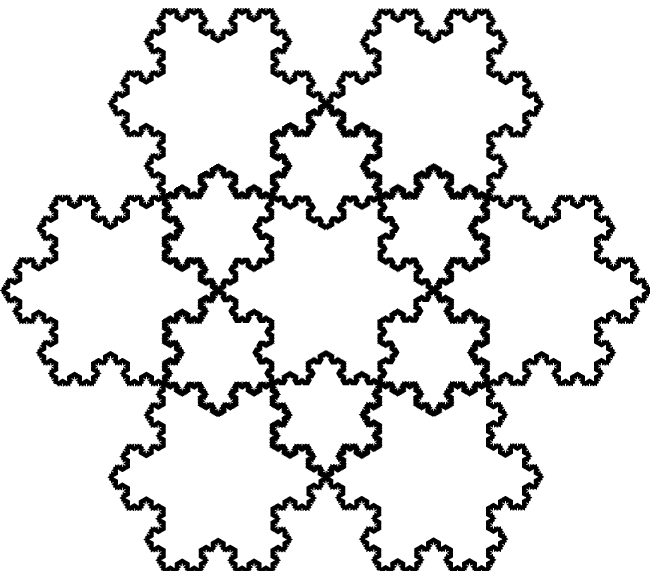}
\caption{The sets $\Lambda_1,\ldots,\Lambda_7$ in the tiling of the snowflake domain. On the right, the sets forming the family $\mathcal{F}_1$ are shown. Note that eg. $\Lambda_7$ does not belong to any $\mathcal{F}_n$ while the set $\Lambda$ as in the picture belongs to both $\mathcal{F}_1$ and $\mathcal{F}_2$.}
\label{fig:tiling}
\end{figure}

\begin{rem}\label{rem:tiling}
The above construction can be defined on any self-similar set $\Lambda_0\subset\R^d$ which has nonempty interior and satisfies the open set condition. There are many such $\Lambda_0$ with fractal boundaries, see e.g. \cite{Bandt91}. We denote this class of SI-martingales by $\mu_n^{\text{tile}}$.
\end{rem}

Definition \ref{def:subdivision} could be further relaxed in various directions. For example, the filtration itself could be allowed to be random. For simplicity we do not consider this here.

We finish by defining a class of sequences of random measures, which includes $\subdivision$ as a special case, but contains also other examples of interest.
\begin{defn} \label{def:cell-type}
Let $\mathcal{F}$ be a family of Borel sets $\Lambda\subset B(0,R)\subset \R^d$ (for some large $R>0$) and $\tau,\zeta>0$. We say that a sequence $(\mu_n)$ of random measures is of \textbf{$(\mathcal{F},\tau,\zeta)$-cell type}, if there is $C>0$ such that
\begin{equation} \label{eq:cell-type}
\mu_n=\sum_{j=1}^{M_n} c_j^{(n)} \mathbf{1}[F_j^{(n)}]\,,
\end{equation}
where:
\begin{enumerate}
\item $\left\{ F_j^{(n)} \right\}_{j=1}^{M_n}$ is a random subset of $\mathcal{F}$,
\item The random variables $c_j^{(n)}$ satisfy $0\le c_j^{(n)}\le C\,2^{\tau n}$ for all $j,n$ almost surely,
\item Almost surely there is $N_0$ such that $M_n \le C\, 2^{\zeta n}$ for all $n\ge N_0$.
\end{enumerate}
\end{defn}

It follows from \eqref{SD:finite_elements} that $\#F_n = O(2^{dn})$. Hence, if $(\mu_n)\in\subdivision$, then $(\mu_n)$ is of $(\mathcal{F},\tau,d)$-cell type, with $\mathcal{F}=\bigcup_n\mathcal{F}_n$, and a value of $\tau$ that depends on the particular sequence, but can be calculated in many specific instances. Moreover, in the case of subdivision martingales (or, more generally, when the cells $F_j^{(n)}$ have bounded overlapping), $\mu_n(x) \le C' 2^{\tau n}$ for all $n,x$ and some $C'>0$.  In particular, $\mu_n^{\text{perc}(\alpha,d)}$ is of $(\mathcal{F},\alpha,d)$-cell type, where $\mathcal{F}=\mathcal{Q}$ is the collection of dyadic cubes.

\subsection{Products of SI-martingales}
\label{subsec:products}

Let $\mu_n^{(1)},\ldots, \mu_n^{(\ell)}$ be independent SI-martingales on the same ambient space $\R^d$. Then
\[
\mu_n = \mu_n^{(1)}\cdots \mu_n^{(\ell)}
\]
is easily checked to also be an SI-martingale. Moreover, if $\mu_n^{(i)}$ is of $(\mathcal{F}_i,\alpha_i,\zeta_i)$-cutout type for each $i$, then $\mu_n$ is of $(\cup_i \mathcal{F}_i,\sum_i \alpha_i,\max_i \zeta_i)$-cutout type. Likewise, if $\mu_n^{(i)}$ is of $(\mathcal{F}_i,\tau_i,\zeta_i)$-cell type for each $i$, then $\mu$ is of $(\mathcal{F},\sum_i \tau_i,\prod_i\zeta_i)$-cell type, where
\[
\mathcal{F} = \{ F_1\cap\cdots\cap F_\ell: F_i \in\mathcal{F}_i \}\,.
\]
Since $\supp\mu_n=\cap_i \supp \mu_n^{(i)}$, products of SI-martingales are useful to study intersections of random Cantor sets. For example, intersections of fractal percolations constructed with different grid sizes, or with different (possibly rotated) coordinate systems, fall into our framework, and everything we prove for fractal percolation will in fact hold for these intersections as well.

Let us point here that taking \emph{Cartesian} products of SI-martingales does not yield an SI-martingale. This will be discussed in detail in Section \ref{sec:products} below.

\subsection{Further examples}
\label{subsec:further-examples}

Although in the rest of the paper we will focus on the classes of examples described in the previous sections, we conclude this section with a brief description of some further classes of spatially independent martingales, both to illustrate the generality of the definition and because they may be useful in some situations.

\begin{enumerate}
\item Let $Q=[0,1]^d$ and fix an integer $m\ge 2$ and a number $r\in (0,1)$ such that $m r^d<1$. Let $z_1,\ldots,z_m$ be points sampled independently and uniformly from $Q$, and let $Q_i$, $i\in\{1,\ldots,m\}$, be the closed cubes with centre $z_i$ and side-length $r$, where we think of $Q$ as a torus, so that some of the cubes may consist of several (Euclidean) disconnected parallelepipeds. Also note that the cubes $Q_i$ are allowed to overlap. We continue this construction inductively - even if $Q_i$ are disconnected, we think of them as a single cube for the purposes of the construction, and choose $m$ independent points $z_{ij}$ uniformly from $Q_i$, and form the sets $Q_{ij}$ which are cubes with centre $z_{ij}$ and side length $r^2$, now thinking of $Q_i$ as a torus. Continuing inductively, we obtain $m^n$ sets $Q_{i_1\ldots i_n}$ for any $n\ge 1$. Now set $\mu_0=\mathbf{1}_{[0,1]^d}$, and
    \[
    \mu_n = \frac{1}{(m r^d)^n} \sum_{i_1\ldots i_n} \mathbf{1}[Q_{i_1\ldots i_n}]\,.
    \]
    In order to adapt to the dyadic scaling, let $(\widetilde{\mu}_n)_n=(\mu_{k_n})_n$ for a nondecreasing sequence $(k_n)$ with $r^{k_n}=\Theta(2^{-n})$; this is easily checked to be an SI-martingale. Moreover, it is of $(\mathcal{F},\tau,\zeta)$-cell type, where $\mathcal{F}$ is the family of parallelepids with all sides of length $\le 1$,  $\tau=\log_2(m r^d)/\log_2 r$ and $\zeta=-\log_2 m/\log_2 r$.

    More generally, $m$ and $r$ could be random variables with suitable tail decay, and could depend on the generation $n$. This construction is related to the random sets in \cite{Korner08}.

\item Our next class of examples are (essentially) the Poissonian cylindrical pulses introduced by Barral and Mandelbrot in \cite{BarralMandelbrot02} and later used as models for many natural and economic phenomena. Let $\mathbf{Q}$ be a measure on $\mathcal{X}$ satisfying \eqref{M:translation-invariance}--\eqref{M:zero-boundary}, and let $W$ be a nonnegative, bounded random variable of unit expectation. Fix a domain $\Omega\subset\R^d$. Let $(\Lambda_j)$ be a realization of the Poisson point process with intensity $\mathbf{Q}$, and let $W_j$ be independent copies of $W$, and also independent of $(\Lambda_j)$. Finally, set
    \[
    \mu_n = \prod_{j: \diam(\Lambda_j) \in [2^{-n},1)}  P_j,\quad\text{where } P_j = W_j \mathbf{1}[\Lambda_j] + \mathbf{1}[\Omega\setminus\Lambda_j]\,.
    \]
    (The functions $P_j$ are the ``pulses''). Then $(\mu_n)$ is an SI-martingale; this can be seen as in Lemma \ref{lem:Poisson-satisfies-properties}. Unlike \cite{BarralMandelbrot02}, we have required $W$ to be bounded so that \eqref{SI:bounded} holds. On the other hand, in \cite{BarralMandelbrot02} the removed shapes are intervals in $\R$, while here we allow much greater generality. These measures do not fall neatly into the cutout-type or cell-type categories (they are in fact of cell-type but the decomposition is somewhat awkward). We believe our methods can still be applied to this class, but we do not pursue this here.

\item Various classes of random covering sets (see e.g. \cite{Kahane85, FanWu04, BarralFan05, Chenetal14}) may be studied by means of random measures that are closely related to SI-martingales. A priori, there is no natural SI-martingale supported on the random covering set $A$, but it is still possible to construct martingale measures supported on random Cantor sets $A'$ which capture a ``large'' proportion of $A$. In \cite{Chenetal14}, largely inspired by the present work and partly adapting our method, uniform Marstrand type projection theorems for ball-type random covering sets in $\R^d$ are obtained.
\end{enumerate}

\section{A geometric criterion for H\"{o}lder continuity}
\label{sec:geometric-Holder}

In Theorem \ref{thm:Holder-continuity}, the most difficult hypothesis to check is often the a priori H\"{o}lder condition \eqref{H:Holder-a-priori}. From this section on, we specialize to SI-martingales of cutout or cell type (recall Definitions \ref{def:cutout-type} and \ref{def:cell-type}), and to families $\{\eta_t\}_{t\in\Gamma}$ that have a mild additional structure. This setting is still general enough to allow us to deduce all of our geometric applications, yet reduces the verification of \eqref{H:Holder-a-priori} to a concrete, and in many cases simple or even trivial, geometric problem relating the shapes in $\mathcal{F}$ with a single measure $\nu$ (typically, we are left to proving a power bound for the $\nu$-measure of small neighbourhoods of the boundaries of shapes in $\mathcal{F}$).

We now describe the structure we will assume on the family $\{\eta_t\}_{t\in\Gamma}$ (this will be slightly generalized below). Consider a metric space $(\Gamma,d)$ (the ``parameter space'') and another space $\mathcal{M}$ (the ``reference space'') together with a measure $\nu$ on $\mathcal{M}$. Suppose for each $t\in \Gamma$ there is a ``projection map''  $\Pi_t:\mathcal{M}\to \R^d$, such that $\eta_t=\Pi_t\nu$.

We give some examples that illustrate the naturality of this setting; we will later come back to each of these (and suitable variants and generalizations) in more detail.
\begin{enumerate}
\item Let $0<k<d$, and let $\mathcal{M}\in\AA_{d,k}$ be a fixed $k$-plane endowed with the $k$-dimensional Hausdorff measure $\nu$, let $\Omega$ be any bounded domain, and let $\Gamma$ be the subset of $\iso_d$ of all $f$ such that $f\mathcal{M}\cap \Omega\neq\varnothing$. We set $\Pi_f=f$. Then $\{ \nu_f\}_{f\in\Gamma}$ consists of $k$-dimensional Hausdorff measures on $V$ for all $V\in\AA_{d,k}$ which intersect $\Omega$.
\item More generally, we can start with an arbitrary measure $\nu$ on $\R^d$, set $\mathcal{M}=\R^d$, and take $\Gamma$ as any totally bounded subset of $\aff_d$ (with $\Pi_f=f$).
\item\label{ex:Bernoulli} Fix an integer $m\ge 2$, and let $\mathcal{M}=\Sigma_m$ be the full-shift on $m$ symbols equipped with some Bernoulli measure $\nu$. Let $\Gamma$ be a totally bounded subset of $(\simi_d^c)^m$ with the induced metric, and given $g=(g_1,\ldots,g_m)\in \Gamma$ let $\Pi_g:\mathcal{M}\to\R^d$ be the projection map for the iterated function system $(g_1,\ldots,g_m)$. Then $\{\eta_g:g\in\Gamma\}$ is a family of self-similar measures.
\end{enumerate}

\begin{prop} \label{prop:cutout-measures-satisfy-Holder-a-priori}
Suppose $(\mu_n)$ is an SI-martingale, which is of $(\mathcal{F},\tau,\zeta)$-cutout type, or of $(\mathcal{F},\tau,\zeta)$-cell type.

As above, let $(\mathcal{M},\nu)$ be a metric measure space, let $\{\Pi_t:\mathcal{M}\to\R^d\}_{t\in\Gamma}$ be a family of maps parametrized by a metric space $\Gamma$, and set $\eta_t=\Pi_t\nu$.

Suppose that there are constants $0<\gamma_0,C<\infty$ such that the following holds:
\begin{equation}
\label{eq:Holder-shape}\nu(\Pi_t^{-1}(\Lambda)\Delta\Pi_u^{-1}(\Lambda))\le C\, d(t,u)^{\gamma_0} \quad\text{ for all } \Lambda\in\mathcal{F}\,,t,u\in\Gamma\,.
\end{equation}

Then
\[
\sup_{t\neq u\in\Gamma,n\ge N_0} \frac{\left|\int \mu_n d\eta_t -\int \mu_n d\eta_u\right|}{2^{(\tau+\zeta)n}d(t,u)^{\gamma_0}} \le C' < \infty\,,
\]
where $C'$ is a quantitative constant, and $N_0$ is the random variable in Definitions \ref{def:cutout-type}, \ref{def:cell-type}.
\end{prop}
In other words, \eqref{H:Holder-a-priori} holds for the family $\{\eta_t\}_{t\in\Gamma}$ with $\theta=\tau+\zeta$.
\begin{proof}
Firstly we assume $(\mu_n)$ is of $(\mathcal{F},\tau,\zeta)$-cutout type. Let $E_n$ be the event that $M_n\le C\, 2^{\zeta n}$; then by assumption there is $N_0$ such that $E_n$ holds for all $n\ge N_0$. Keeping \eqref{eq:cutout-type} in mind, we only need to show that
\begin{equation} \label{eq:Holder-reference-measure}
\left|\Pi_t\nu(\Omega\setminus\cup_{i=1}^{M_n} \Lambda_i) -\Pi_u\nu(\Omega\setminus\cup_{i=1}^{M_n} \Lambda_i)\right|\le C (M_n+1) d(t,u)^{\gamma_0}\,,
\end{equation}
whenever $\{\Lambda_i\}_{i=1}^{M_n}\subset\mathcal{F}$.

Write
\begin{align*}
A_i&=\Pi_t^{-1}(\Lambda_i\cap\Omega)\setminus  \Pi_u^{-1}(\Lambda_i\cap\Omega)\,,\\
B_i&=\Pi_u^{-1}(\Lambda_i\cap\Omega)\setminus  \Pi_t^{-1}(\Lambda_i\cap\Omega)\,,\\
D&=\bigcup_{i=1}^{M_n}\Pi_t^{-1}(\Lambda_i\cap\Omega)\cap \Pi_u^{-1}(\Lambda_i\cap\Omega)\,.
\end{align*}
We can then decompose
\begin{align*}
\Pi_t\nu(\Omega\setminus\cup_{i=1}^{M_n} \Lambda_i)&=\nu(\Pi_t^{-1}\Omega)- \nu\left(D\right)-\nu\left(\cup_{i=1}^{M_n} A_i\setminus D\right),\\
\Pi_u\nu(\Omega\setminus\cup_{i=1}^{M_n} \Lambda_i)&=\nu(\Pi_u^{-1}\Omega)- \nu\left(D\right)-\nu\left(\cup_{i=1}^{M_n} B_i\setminus D\right),\\
\nu(\Pi_t^{-1}\Omega) - \nu(\Pi_u^{-1}\Omega) &= \nu\left(\Pi_t^{-1}(\Omega)\setminus\Pi_u^{-1}(\Omega)\right) -  \nu\left(\Pi_u^{-1}(\Omega)\setminus\Pi_t^{-1}(\Omega)\right),
\end{align*}
so that the left-hand side of \eqref{eq:Holder-reference-measure} gets bounded by
\begin{equation}\label{eq:noD}
\nu(\Pi_t^{-1}\Omega\Delta\Pi_u^{-1}\Omega)+ \sum_{i=1}^{M_n} \nu(A_i)+\nu(B_i),
\end{equation}
and \eqref{eq:Holder-reference-measure} follows from the hypothesis \eqref{eq:Holder-shape}. Note that e.g.
\begin{align*}
A_i \subset\left(\Pi_t^{-1}(\Lambda_i)\setminus\Pi_{u}^{-1}(\Lambda_i)\right)\cup\left(\Pi_t^{-1}(\Omega)\setminus\Pi_{u}^{-1}(\Omega)\right)\,.
\end{align*}

If instead $(\mu_n)$ is of cell type, the proof is even easier. Letting $E_n, N_0$ be as before, and taking $n\ge N_0$, we estimate
\begin{align*}
\left|\int \mu_n d\Pi_t\nu -\int \mu_n d\Pi_u\nu\right| &\le \left| \sum_{j=1}^{M_n} c_j\,\left(  \nu(\Pi_t^{-1}(F_j^{(n)})) - \nu(\Pi_u^{-1}(F_j^{(n)}))\right)   \right|\\
&\le M_n \max_j c_j \max_{F\in\mathcal{F}} | \nu(\Pi_t^{-1}(F_j^{(n)})) - \nu(\Pi_u^{-1}(F_j^{(n)})) | \\
&\le O(1) 2^{\zeta n} 2^{\tau n} d(t,u)^{\gamma_0}\,,
\end{align*}
yielding the result.
\end{proof}

\begin{rem}
The exponent $\tau+\zeta$ from the previous proposition is in general very far from optimal. We discuss possible avenues to improve the exponent in the case of Poissonian cutouts, but similar remarks are valid for other SI martingales. To begin with, most of the removed shapes are completely covered by other removed shapes so they can be ignored. In general, the measures $\Pi\nu_t$ have very small support, so that in fact much fewer than $M_n$ sets $\Lambda_j$ need to be considered for each of them. Next, condition \eqref{eq:Holder-shape} is not scale invariant, which roughly means that it does not take into account the size of $\Lambda$. Due to the scale invariance of $\mathbf{Q}$, for sets $\Lambda\in\mathcal{F}$ of small diameter, an estimate much better than \eqref{eq:Holder-shape} holds. Last, but not least, our estimate ignores all cancelation in $\nu(\cup_i A_i)-\nu(\cup_i B_i)$;  because of the independence in the construction, there is in fact a large amount of cancelation. In order to keep the exposition as general and simple as possible, and because finding the optimal exponent appears in any case to be quite difficult, we ignore these possible improvements.
\end{rem}

In some applications, it is convenient to allow the measure on $\mathcal{M}$ being projected to also depend on the parameter $t$. For example, instead of projecting a fixed Bernoulli measure to a family of self-similar sets as in \eqref{ex:Bernoulli} above, it is more natural to consider, for each IFS of similarities, the natural Bernoulli measure. Or one might want to prove a statement for \emph{all} Bernoulli measures simultaneously. We give a variant of Proposition \ref{prop:cutout-measures-satisfy-Holder-a-priori} for this purpose.

\begin{prop} \label{prop:cutout-measures-satisfy-Holder-a-priori-2}
Suppose $(\mu_n)$ is an SI-martingale, which is of $(\mathcal{F},\tau,\zeta)$-cutout type, or of $(\mathcal{F},\tau,\zeta)$-cell type.

Let $\{\eta_t\}_{t\in\Gamma}$ be a family of measures in $\R^d$ parametrized by $(\Gamma,d)$. Let $\mathcal{M}$ be a metric space, and assume that
for each pair $t,u\in\Gamma$, there are maps $\Pi_t,\Pi_u\,\colon \mathcal{M}\to\R^d$ and measures $\nu_t$, $\nu_u$ on $\mathcal{M}$ such that $\eta_t=\Pi_t\nu_t$, $\eta_u=\Pi_u\nu_u$.

Suppose that there are constants $0<\gamma_0,C<\infty$ such that the following holds:
\begin{equation}
\label{eq:Holder-shape-2} \nu_t(\Pi_t^{-1}\Lambda\Delta\Pi_u^{-1}\Lambda)\le C\, d(t,u)^{\gamma_0}\quad\text{for all $\Lambda\in\mathcal{F},t,u\in\Gamma$}\,,
\end{equation}
and (in the cutout type case) also
\begin{equation}\label{eq:Holder-measure}
\left|\nu_t\left(\cup_{i=1}^M\left(\Pi_{u}^{-1}(\Lambda_i\cap\Omega)\right)\right)-\nu_u\left(\cup_{i=1}^M\left(\Pi_{u}^{-1}(\Lambda_i\cap\Omega)\right)\right)\right|\le C M d(t,u)^{\gamma_0}\,,
\end{equation}
whenever $\Lambda_1,\ldots,\Lambda_M\in\mathcal{F}$ and $t,u\in\Gamma$.

Then
\[
\sup_{t\neq u\in\Gamma,n\ge N_0} \frac{\left|\int \mu_n d\eta_t -\int \mu_n d\eta_u\right|}{2^{(\tau+\zeta)n}d(t,u)^{\gamma_0}} \le C' < \infty\,,
\]
where $C'$ is a quantitative constant, and $N_0$ is the random variable in Definitions \ref{def:cutout-type}, \ref{def:cell-type}.
\end{prop}
\begin{proof}
The proof is a small modification of the proof of Proposition \ref{prop:cutout-measures-satisfy-Holder-a-priori}. We consider only the cutout type case, since the cell type case is similar but easier. Now \eqref{eq:Holder-reference-measure} becomes
\begin{equation*}
\left|\Pi_t\nu_t(\Omega\setminus\cup_{i=1}^{M_n} \Lambda_i) -\Pi_u\nu_u(\Omega\setminus\cup_{i=1}^{M_n} \Lambda_i)\right|\le C (M_n+1) d(t,u)^{\gamma_0}\,.
\end{equation*}
Writing $A=\Omega\setminus\cup_{i=1}^{M_n}\Lambda_i$, we bound the left-hand side by
\[
|\nu_t(\Pi_t^{-1}A)-\nu_t(\Pi_u^{-1}A)| + |\nu_t(\Pi_u^{-1}A)-\nu_u(\Pi_u^{-1}A)|\,,
\]
and, in turn, bound the first term as in Proposition \ref{prop:cutout-measures-satisfy-Holder-a-priori} (now using \eqref{eq:Holder-shape-2}), and the second term using \eqref{eq:Holder-measure}.
\end{proof}

\section{Affine intersections and projections}
\label{sec:affine}

\subsection{Main result on affine intersections}

We are now ready to start applying the machinery developed in the previous sections to the specific geometric problems that motivated this work. The next theorem provides a rich class of random fractal measures with the property that \emph{all} projections onto $k$-dimensional planes $V$ are absolutely continuous with a H\"{o}lder density, which is moreover also H\"{o}lder in $V$. To the best of our knowledge, no examples of singular measures with this property were known.

\begin{thm} \label{thm:linear-projections}
Let $(\mu_n)$ be an SI-martingale which either is of $(\mathcal{F},\tau,\zeta)$-cell type and satisfies $\mu_n(x)\le 2^{\alpha n}$ for all $n$ and $x$, or is of $(\mathcal{F},\alpha,\zeta)$-cutout type. Let $k\in\{1,\ldots,d-1\}$ with $\alpha<k$.

Suppose that there are constants $0<\gamma_0,C<\infty$ such that for all $V\in\mathbb{A}_{d,k}$, all $\varepsilon>0$ and any isometry $f$ which is $\varepsilon$-close to the identity, we have
\begin{equation} \label{eq:non-flat-shapes}
\mathcal{H}^k\left(V\cap \Lambda\setminus f(\Lambda)\right)\le C\, \varepsilon^{\gamma_0} \quad\text{for all } \Lambda\in\mathcal{F}.
\end{equation}
Fix $\gamma$ such that
\[
0<\gamma< \left\{
        \begin{array}{ll}
           \frac{(k-\alpha)\gamma_0}{(k-\alpha)+2(\alpha+\zeta)} & \text{ in the cutout-type case} \\
            \frac{(k-\alpha)\gamma_0}{(k-\alpha)+2(\tau+\zeta)} & \text{ in the cell-type case}
        \end{array}
          \right..
\]
Then there is a finite random variable $K$ such that:
\begin{enumerate}
\item[\emph{(i)}] The sequence $Y_n^V := \int_V \mu_n\, d\mathcal{H}^k$ converges uniformly over all $V\in\mathbb{A}_{d,k}$. Denote the limit by $Y^V$.
\item[\emph{(ii)}] $|Y^V-Y^W| \le K\, d(V,W)^\gamma$.
\item[\emph{(iii)}] For each $V\in\mathbb{G}_{d,d-k}$, the projection $P_V\mu_\infty$ is absolutely continuous. Moreover, if we denote its density by $f_V$, then the map $(x,V)\mapsto f_V(x)$ is H\"{o}lder continuous with exponent $\gamma$.
\end{enumerate}

\end{thm}

Before giving the short proof of the theorem, we make some remarks on the assumptions. Note that if $\mu\in\mathcal{P}_d$ is such that $\ldimloc(\mu,x)<d-k$ for a single point $x$, then no projection of $\mu$ onto a $(d-k)$-plane can have a bounded density (since the same would hold for the projection of $x$). Hence in order for the last part of the theorem to hold for a given measure, the spectrum of local dimensions must lie to the right of $d-k$. This is essentially what the assumption $\mu_n(x) = O(2^{\alpha n})$ is saying.

The assumption \eqref{eq:non-flat-shapes} asserts that the (boundaries of) the shapes from $\mathcal{F}$ are non-flat in a certain quantitative sense - for example, in the case of shapes with smooth boundary, it is enough that all principal curvatures have the same sign and are uniformly bounded away from zero. This non-flatness condition is natural - the fact that cubes have flat boundaries is precisely the reason why the theorem fails for fractal percolation, and it also fails for random measures in $\mathcal{PCM}$ when $\mathbf{Q}$ gives positive mass to shapes with a $k$-flat piece in the boundary.

\begin{proof}[Proof of Theorem \ref{thm:linear-projections}]
Let $\Upsilon$ be a bounded open set containing $\supp\mu_0$ (this is possible by \eqref{SI:bounded}).

Fix an affine $k$-plane $V_0$, and set $\nu=\mathcal{H}^k|_{V_0}$. Let $\Gamma$ be the set of isometries $f$ such that $fV_0\cap \Upsilon\neq\varnothing$; this is a bounded family, and the images $\{f V_0:f\in\Gamma\}$ yield all $k$-planes that intersect $\Omega$. Consider the family of measures $\{ f\nu\}_{f\in\Gamma}$.

Hypotheses \eqref{H:size-parameter-space}, \eqref{H:dim-deterministic-measures} and \eqref{H:codim-random-measure} of Theorem \ref{thm:Holder-continuity} hold either trivially or by assumption (with $s=k$), while \eqref{H:Holder-a-priori}, with $\theta=\alpha+\zeta$ in the cutout-type case and $\theta=\tau+\zeta$ in the cell-type case, follows from Proposition \ref{prop:cutout-measures-satisfy-Holder-a-priori} and the non-flatness assumption \eqref{eq:non-flat-shapes}. Since we also assume $\alpha<k$, we can apply Theorem \ref{thm:Holder-continuity} to deduce the first two assertions. To pass from the metric on isometries to the metric on $\mathbb{A}_{d,k}$, we use that
\[
d(V,W) = \Theta(\min\{d(f,g):f,g\in\text{ISO}_d, fV_0=V, gV_0=W\})\,.
\]

 It remains to show (iii). If $V$ is a linear $(d-k)$-plane and $t\in V$, let $V^\perp_t$ denote the affine $k$-plane orthogonal to $V$ passing through $t$. Let $B^\circ(x,r)\subset V$ denote the open ball of centre $x$ and radius $r$. Using Fubini's Theorem and (ii), we find that for all $x\in V$, $0<r<1$,
\begin{align*}
P_V\mu_\infty(B^\circ(x,r)) &\le \liminf_{n\to\infty} P_V\mu_n(B^\circ(x,r))\\
&=\liminf_{n\to\infty} \int_{B^\circ(x,r)} Y_n^{V^\perp_t} dt \le K'\, r^{d-k}\,,
\end{align*}
for some finite random variable $K'$; thus $P_V\mu_\infty$ is absolutely continuous. Moreover, thanks to the uniform convergence of $Y_n^{V^\perp_t}$, the density of $P_V\mu_\infty$ is $f_V(t)=Y^{V^\perp_t}$ for all $V\in\mathbf{G}(d,d-k)$ and $t\in V$. The claimed H\"{o}lder continuity is then immediate from (ii).
\end{proof}

\begin{cor} \label{cor:Holder-projections}
For any $1\le k<\alpha<d$, there is a measure $\mu$ of exact dimension $d-\alpha$ (i.e. $\dimloc(\mu,x)=d-\alpha$ for $\mu$ almost all $x$), whose support is a set of Hausdorff and box dimension $d-\alpha$, such that $P_V\mu$ is absolutely continuous for all $V\in\mathbb{G}_{d,d-k}$ and, denoting the density by $f_V$, the map $(V,x)\mapsto f_V(x)$ is H\"{o}lder with some quantitative exponent.
\end{cor}
\begin{proof}
Take $\mu=\mu_\infty^{\text{ball}(\alpha,d)}$. The fact that $\mu$ has exact dimension $d-\alpha$ follows from Theorem \ref{thm:dim_cut_out_set}. Since $\mathcal{F}$ consists of balls of diameter $\le 1$ that intersect the compact set $\supp\mu_0$, the non-flatness assumption \eqref{eq:non-flat-shapes} is seen to hold with $\gamma_0=k/2$. The remaining claims are then immediate from Theorem \ref{thm:linear-projections}.
\end{proof}

\begin{rem} \label{rem:affine-perc}
If the assumption \eqref{eq:non-flat-shapes} in Theorem \ref{thm:linear-projections} holds only for some subset $\mathcal{A}$ of affine planes, then so do the first two claims restricted to $\mathcal{A}$. If $\mathcal{A}=\{ V+t:V\in \mathcal{G},t\in\R^d\}$ for some $\mathcal{G}\subset\mathbb{G}_{d,k}$, then the last claim holds for $V^\perp\in\mathcal{G}$. In particular, this applies to $\mu_n^{\text{perc}(\alpha,d)}$ and $\mathcal{A}$ equal to the $k$-planes that make an angle at least $\e>0$ with any coordinate hyperplane. Since $\e$ is arbitrary, we recover the result of Peres and Rams \cite{PeresRams14} on projections of the fractal percolation measure in non-principal directions (and generalize it to arbitrary dimensions). More generally, if all shapes in $\mathcal{F}$ are polyhedra with faces parallel to a hyperplane in $\{ W_1,\ldots, W_\ell\}\subset\mathbb{G}_{d,d-1}$, then the last part of Theorem \ref{thm:linear-projections} holds for planes $V$ such that $V^\perp$ makes an angle at least $\e$ with all $W_i$. This applies, for example, to percolation constructed in a triangular or ``simplex'' grid, or to intersections of finitely many percolations constructed from grids of different sizes and orientations (recall Section \ref{subsec:products}).
\end{rem}

\subsection{Applications: non-tube-null sets and Fourier decay}

Let $B$ be the unit ball of $\R^d$. Recall that a set $E\subset B$ is \textbf{tube-null} if, for each $\e>0$, there are countably many lines $L_j$ and numbers $\delta_j$ such that $E\subset \cup_j L_j(\delta_j)$ and $\sum_j \delta_j^{d-1} <\e$. Motivated by the connections between tube-null sets and the localization problem for the Fourier transform in dimension $d\ge 2$, Carbery, Soria and Vargas \cite{CSV07} asked what is the infimum of the dimensions of sets in $\R^d$ which are \emph{not} tube null. In \cite{ShmerkinSuomala12}, we employed a subdivision type random construction to show that the answer is $d-1$. We did this by showing that the natural measure $\mu$ on such constructions has the property that for some random $K>0$, all its orthogonal projections onto hyperplanes are absolutely continuous with a density bounded by $K$. On the other hand, a simple argument shows that given a measure $\mu$ with this property, no set of positive $\mu$-measure can be tube-null, see \cite[Proof of Theorem 1.2]{ShmerkinSuomala12}. Hence, it follows from Theorem \ref{thm:linear-projections} that any set carrying positive mass for any limit martingale $\mu_\infty$ satisfying the assumptions of the theorem for $\alpha\in (0,1)$, also  cannot be tube null. This is the case, in particular, for $\mu_\infty^{\text{ball}(\alpha,d)}$.  Thus, we again find that the existence of non-tube null sets of small dimension is a rather general phenomenon due to a small number of underlying probabilistic and geometric properties.

We conclude this section with a short discussion of the sharpness of Corollary \ref{cor:Holder-projections}, and its relationship with Fourier dimension. Recall that given a finite measure $\mu$ on $\R^d$, its \textbf{Fourier dimension} is
\[
\dim_F\mu= \sup\{ 0\le\sigma\le d: \widehat{\mu}(x)=O_\sigma(|x|^{-\sigma/2}) \}\,,
\]
where $\widehat{\mu}(x)=\int_{\R^d}\exp(-2\pi i x\cdot y)\,d\mu(y)$.

It is well known that $\dim_F\mu\le \dim_H\mu$, where the Hausdorff dimension $\dim_H\mu$ is defined as the infimum of the Hausdorff dimensions of sets of positive $\mu$-measure. A large value of the Fourier dimension (especially in relation to the Hausdorff dimension) is a sort of ``pseudo-randomness'' indicator, see e.g. the discussion in \cite{LabaPramanik09}. The next lemma follows from a classical result relating H\"{o}lder smoothness with the decay of the Fourier transform.

\begin{lemma} \label{lem:Fourier-Holder-projections}
Let $\mu$ be a compactly supported measure on $\mathbb{R}^d$ with the following property: there are $\gamma, C>0$ such that for $\theta\in S^{d-1}$, the projection $\mu_\theta$ of $\mu$ onto a line in direction $\theta$ is absolutely continuous, and its density $f_\ell$ satisfies $|f_\ell(x)-f_\ell(y)|\le C\, |x-y|^\gamma$. Then $\dim_F\mu\ge 2\gamma$.
\end{lemma}
\begin{proof}
It is a classical result of Zygmund that the Fourier transform of a compactly supported function $f$ on $\R^d$ which is H\"{o}lder continuous with exponent $\gamma$ and constant $C$, satisfies $|\widehat{f}(x)| \le O_d(C)|x|^{-\gamma}$, see e.g. \cite[Theorem 3.2.9]{Grafakos08}. (The result is stated there for Fourier series, but the conclusion for the Fourier transform follows easily by considering appropriate dilates of the function.)  Since
\[
\widehat{\mu}(x\theta) = \widehat{\mu}_\theta(x\theta)\quad\text{for all }\theta\in S^{d-1}, x\in\R\,,
\]
Zygmunds's result and the hypotheses of the lemma imply that $|\widehat{\mu}(x)| = O_{d}(C) |x|^{-\gamma}$, as claimed.
\end{proof}

This lemma shows that, given $\mu\in\mathcal{P}_d$ with $\dim_H\mu=s<2$, the best possible uniform regularity for its projections is H\"{o}lder with exponent $s/2$. In particular, it cannot happen that all projections are Lipschitz with a uniform Lipschitz constant. In this sense, Corollary \ref{cor:Holder-projections} is sharp when $d=2,k=1$, up to the value of the H\"{o}lder exponent. However, we do not know what happens in higher dimensions, nor do we know whether, even in the plane, it can happen that all the projections of a measure of dimension less than $2$ are Lipschitz. On the positive side, we get the following immediate consequence of Theorem \ref{thm:linear-projections} and Lemma \ref{lem:Fourier-Holder-projections}.

\begin{cor} \label{cor:positive-Fourier-dim}
Under the assumptions of Theorem \ref{thm:linear-projections} with $k=1$,
\[
\dim_F(\mu_\infty) \ge \frac{2(1-\alpha)\gamma_0}{(1-\alpha)+2(\alpha+\zeta)}
\]
in the cutout-type case, and likewise with $2(\tau+\zeta)$ in place of $2(\alpha+\zeta)$ in the cell-type case.
\end{cor}

 In general this is  far from sharp. In Section \ref{subsec:Salem} we show that some SI-martingales $\mu_\infty$ are in fact Salem measures, that is, $\dim_F(\mu_\infty)=\dim_H(\mu_\infty)$. However, the class for which we can prove this is rather restricted, and is different than the class to which the corollary applies. We remark that having positive Fourier dimension is already an important feature of a measure. If $\dim_F \mu>0$ it follows, for example, that sufficiently high convolution powers of $\mu$ become arbitrarily smooth.

\section{Fractal boundaries and intersections with algebraic curves}
\label{sec:algebraic}

\subsection{Intersections with algebraic curves}
Theorem \ref{thm:linear-projections} deals with intersections with (Hausdorff measure on) planes and linear projections. A natural next step is to replace planes with algebraic varieties, and linear projections by polynomial projections. However, if the shapes in $\mathcal{F}$ are balls, then in general there will be no continuity with respect to intersections with (surface measure on) spheres: consider $\mu_n^{\text{ball}}$. In general, if the boundaries of the removed shapes are algebraic, then one cannot hope for continuity with respect to algebraic objects. In the plane, we overcome this issue by removing shapes with a fractal boundary instead (we employ Von Koch snowflakes to provide concrete examples, although it will be clear that many other similar shapes, including many domains with a finite union of unrectifiable self-similar arc as boundaries, could be chosen instead). We believe that a suitable analogue of the Von Koch snowflake should exist in any dimension $d\ge 2$, but have not been able to prove it.

We start by introducing families of curves that include algebraic curves as special cases. We denote by $[x,y]$ the line segment joining $x$ and $y$, and by $[x,y]_\e$ the $(\e |x-y|)$-neighbourhood of $[x,y]$. In this section, we denote the length of a segment $J$ by $|J|$.

\begin{defn}
Let $h:(0,1)\to (0,1)$ be a function. A Jordan arc $V\subset\R^2$ is \textbf{$h$-linearly approximable} if for every $x,y\in V$ and every $\e>0$, the part of the curve $V_{x,y}$ joining $x,y$ has a subcurve $V_{u,v}$ such that
\begin{enumerate}
\renewcommand{\labelenumi}{(LA\arabic{enumi})}
\renewcommand{\theenumi}{LA\arabic{enumi}}
\item\label{LA1} $V_{u,v}\subset [u,v]_\e$,
\item\label{LA2} $h(\e)\mathcal{H}^1(V_{x,y})\le |u-v|$.
\end{enumerate}
\end{defn}

The following geometric lemma will help us verify condition \eqref{eq:Holder-shape-2} when $\eta_t$ are length measures on linearly approximable curves, and the boundaries of shapes in $\Lambda$ satisfy the conditions \eqref{K1} and \eqref{K2} in the lemma.

\begin{lemma}\label{lem:porous}
Let $\Lambda\subset\R^2$ be a Jordan domain, bounded by a closed Jordan curve $\pi$. Given $x,y\in\pi$, we denote by $\pi_{x,y}$ the component of $\pi\setminus\{x,y\}$ with smaller diameter. Suppose that $\pi$ has the following properties for some $1<C<\infty$:
\begin{enumerate}
\renewcommand{\labelenumi}{(K\arabic{enumi})}
\renewcommand{\theenumi}{K\arabic{enumi}}
\item\label{K1} $\diam(\pi_{x,y})\le C|x-y|$ for all $x,y\in\pi$.
\item\label{K2} $\pi_{x,y}\not\subset [x,y]_{1/C}$ for all $x,y\in\pi$.
\end{enumerate}
Then there is $0<\gamma_1=\gamma_1(C,h)<1$ such that for each $h$-linearly approximable curve $V$ of length at most $1$, and each $0<\varepsilon<1$, $V\cap\{x\in\Lambda\,:d(x,\partial\Lambda)<\varepsilon\}$ can be covered by $O_{C,h}(\varepsilon^{-\gamma_1}$) balls of radius $\varepsilon$.
\end{lemma}

\begin{proof}
Conditions \eqref{K1} and \eqref{K2} imply that $\pi$ is uniformly porous in the following sense: There is $\varrho=\varrho(C)>0$ such that whenever $J\subset\R^2$ is a line segment, we can find $x\in J$ with $B(x,\varrho|J|)\cap \pi=\emptyset$.

Before verifying this porosity condition, let us show how it can be used to prove the claim of the Lemma.
Without loss of generality, we may assume that $V$ is a curve of length $1$, so (abusing notation slightly)
let $V\colon[0,1]\rightarrow \R^2$ be an arclength parametrization of $V$.
Let $\varepsilon=\varrho/2$ and $\delta=h(\varepsilon)$, and choose integers $n_0>1/\delta$ and $n>n_0/\e$.

Now that we have fixed the parameters $n$ and $n_0$, we first observe that \eqref{LA1} and \eqref{LA2} imply that $V_{u_0,v_0}\subset [u_0,v_0]_{\e}$ for some subcurve $V_{u_0,v_0}\subset V$ of length $\mathcal{H}^1(V_{u_0,v_0})\ge 1/n_0$.

Decomposing $V$ into $n$ subcurves of length $1/n$, and applying the porosity condition to $J=[u_0,v_0]$, we observe that $V_{u,v}\cap \pi(1/n)=\emptyset$, where $V_{u,v}$ is one of these $n$ subcurves, see Figure \ref{fig:porosity}. To summarize, we have shown that we may cover $V\cap \pi(1/n)$ with $n-1$ subcurves of length $1/n$.

\begin{figure}
   \centering
\resizebox{0.9\textwidth}{!}{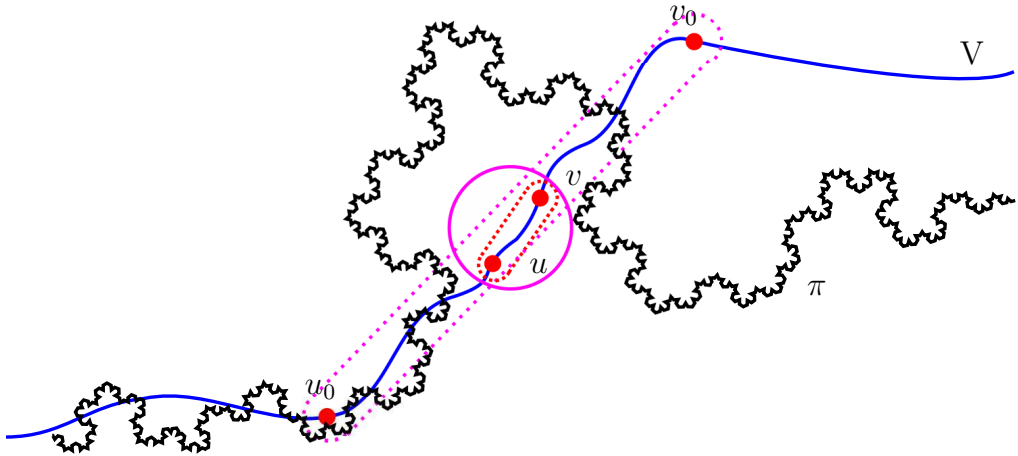}
\caption{Illustration for the proof of Lemma \ref{lem:porous}. The ball $B=B(x,\varrho|J_{u_0,v_0}|)$ shown in the picture satisfies $V_{u,v}(1/n)\subset B\subset\R^2\setminus\pi$.}
\label{fig:porosity}
\end{figure}

We then continue inductively. Replacing $V$ by these $n-1$ subcurves, we observe that $V\cap \pi(n^{-2})$ can be covered by $(n-1)^2$ curves of length $n^{-2}$, and more generally, for each $k\in\N$, $V\cap \pi(n^{-k})$ may be covered by $(n-1)^k$ curves of length $n^{-k}$. This gives the claim for $\gamma_1=\log(n-1)/\log n<1$.

It remains to verify the uniform porosity condition. Let $m$ be a sufficiently large integer to be determined later. Suppose the porosity condition does not hold with the parameter $\varrho=1/m$. Then we can find a line segment $J$ and equally spaced points $x_0,x_1,\ldots,x_m$ on $J$ and $y_j\in B(x_j,|J|/m)\cap\pi$ for each $j$. By condition \eqref{K1}, each $\pi_{y_j,y_{j+1}}$ is within distance $O_C(|J|/m)$ of $[x_j,x_{j+1}]$ and, since $\pi_{y_0,y_m}\subset\bigcup_{j=0}^{m-1}\pi_{y_j,y_{j+1}}$, and  $|y_0-x_0|, |y_m-x_m|\le|J|/m$, also
\[
\pi_{y_0,y_m}\subset [y_0,y_m]_{O_C(|J|/m)}\,.
\]
But $|y_0-y_m|=\Omega_C(|J|)$, so we arrive at a contradiction with \eqref{K2}, provided $m$ is chosen large enough.
\end{proof}

\begin{rem}
We discuss briefly the geometric assumptions in Lemma \ref{lem:porous}. The condition \eqref{K1} says that $\pi$ is a quasicircle (see e.g. \cite{Gehring82}). It is well known that the Von Koch snowflake curve is a quasicircle. In addition to this standard example, there are many other domains $\Lambda$ with piecewise self-similar boundary satisfying \eqref{K1} (see \cite{Aseevetal03}). Further examples (not necessarily self-similar) can be found as generalized snowflakes following the construction in \cite{Rohde01}.

Condition \eqref{K2} is a quantitative unrectifiability property of the boundary $\pi$ which is easily seen to hold for the snowflake curve. More generally, \eqref{K2} is valid for any self-similar curve which is not a line segment. Indeed, if $\pi$ is such a curve with endpoints $x_0$ and $y_0$, there is $C'<\infty$ and $z\in\pi_{x_0,y_0}$ with
\begin{equation*}
z\notin[x_0,y_0]_{1/C'}\,.
\end{equation*}
Now for any distinct points $x,y\in\pi$, there is $f\in\simi^{c}_2$ with contraction ratio $\Omega(\diam(\pi_{x,y}))$ such that $f(\pi)\subset\pi_{x,y}$. Then \eqref{K2} holds for $f(\pi)$ with the  constant $C'$ and
 consequently for $\pi_{x,y}$ with a slightly larger constant $C=O(C')$.
\end{rem}

Let $k\in\N$ and denote by $\mathcal{K}_k$ the family of (real) algebraic curves in $\R^2$ of degree at most $k$. For future use, we also set $\mathcal{K}=\cup_{k\in\N}\mathcal{K}_k$. Given $V\in\mathcal{K}_k$, let $\eta_V$ be the length measure on $V$.

Since our goal is to obtain an analogue of Theorem \ref{thm:linear-projections} for algebraic curves, we need a suitable metric on the space of curves in $\mathcal{K}_k$ which meet a fixed large ball $\Upsilon$.

\begin{defn}\label{def:metric-algebraic}
Let $\Upsilon$ be a fixed large ball. For $V,W\in\mathcal{K}_k$, 
let
\[
d(V,W)=\sup\left\{\left|\int L\,d\eta_V-\int L\,d\eta_W\right|\right\}\,,
\]
where the supremum is over all Lipschitz functions $L$ supported on $\Upsilon$ and with Lipschitz constant at most one.
\end{defn}

It is easy to see that $(\eta_V,\eta_W)\mapsto d(V,W)$ does define a metric among the measures $\eta_V|_\Upsilon$, $V\in\mathcal{K}_k$, and this metric induces the weak topology, see \cite[\S 14.12]{Mattila95}. It is not quite a metric on $\mathcal{K}_k$ as $d$ does not distinguish $V,W\in\mathcal{K}_k$ if $V$ and $W$ are contained in $\R^2\setminus\Upsilon$, or if the set $(V\Delta W)\cap\Upsilon$ is finite. For simplicity, we will slightly abuse notation and stick to the notation $d(V,W)$; the ball $\Upsilon$ and the corresponding identifications will always be clear from context. We also note that the metric $d(V,W)$ is closely related to the $1$-Wasserstein metric, although the latter is usually defined only on probability measures.

When $k=1$, the metric $d(V,W)$ is easily seen to be comparable to the standard metric between affine lines which hit $\Upsilon$. We remark that for other natural choices of metric with this property, one cannot hope to even get continuity of $V\mapsto \int_V \mu_n\,d\eta_V$, already for $k=2$. For example, Hausdorff metric (on the intersections $V\cap \Upsilon)$ does not work, because the Hausdorff distance between two bounded algebraic curves can be very small even if their lengths are far away from each other; consider, for example, a line segment and a thin ellipsoid. Using the normalized coefficients of the defining polynomials does not work either, since two real algebraic curves may be geometrically far away even if their coefficients are very close; consider for example $(x^2+\e)y=0$ and $x^2 y=0$. Nevertheless, if one restricts
attention to a small neighbourhood of a polynomial for which $0$ is a regular value, then it can be shown (see \eqref{eq:d-dominated-by-coeff-diff}) that $d$ is dominated by the difference between the coefficients. This
observation is crucial in our applications to polynomial projections in Section \ref{sec:poly_projections}.

\begin{thm}\label{thm:alg_curves}
Let $(\mu_n)$ be an SI-martingale in $\R^2$ of $(\mathcal{F},\alpha,\zeta)$-cutout type, or of $(\mathcal{F},\tau,\zeta)$-cell type and suppose $\mu_n(x) \le C\, 2^{\alpha n}$, for some $\alpha\in (0,1)$. Furthermore, assume that each $\Lambda\in\mathcal{F}$ satisfies \eqref{K1} and \eqref{K2} with some uniform constant $C$.

Then there are a quantitative constant $\gamma$ and a.s. finite random variable $K>0$ such that a.s.:
\begin{enumerate}
\item[\emph{(i)}]\label{claim:ac1} The sequence $Y_n^V:=\int \mu_n \,d\eta_V$ converges uniformly over all $V\in\mathcal{K}_k$; denote the limit by $Y^{V}$.
\item[\emph{(ii)}] $|Y^V-Y^W|\le K\, d(V,W)^\gamma$ for all $V,W\in\mathcal{K}_k$,
\end{enumerate}
where $d$ is as defined above with respect to any open ball $\Upsilon=B^\circ(0,R)$ containing $\supp\mu_0$.
\end{thm}

\begin{proof}
Note that the result is an immediate consequence of Theorem \ref{thm:Holder-continuity}; thus it suffices to check that all its assumptions are met. We only have to check \eqref{H:size-parameter-space} and \eqref{H:Holder-a-priori}, as the remaining hypotheses hold trivially.

Condition \eqref{H:size-parameter-space} is essentially proved in \cite{ShmerkinSuomala12}, except that in that paper the space consists of
the subcurves of curves in $\mathcal{K}_k$ which are also graphs of convex/concave functions with uniformly bounded derivative, and the metric is the Hausdorff metric.

Let us show how \eqref{H:size-parameter-space} for
$\Gamma=\{\eta_V\}_{V\in\mathcal{K}_k}$
can be derived from the corresponding result in \cite{ShmerkinSuomala12}. To begin with, we note that (say, by Bezout's theorem) for any $V\in\mathcal{K}_k$, $V\cap\Upsilon$ may be split into $M=O_k(1)$ (possibly degenerate) disjoint curves, which are graphs of functions $y=f(x),J\to\R$ or $x=f(y), J\to\R$, where $J\subset[-R,R]$ is an interval (depending on $f$) so that the derivative of $f$ is monotone on $J$ and satisfies $|f'|\le 1$. Let us denote by $\mathcal{G}^x_k$ the family of such graphs $y=f(x)$, and by $\mathcal{G}^y_k$ the graphs $x=f(y)$. We define an embedding $e\colon\mathcal{K}_k\to (\mathcal{G}^x_k)^M\times(\mathcal{G}^y_x)^M=: X$ as follows: Suppose $V\cap\Upsilon$ is a disjoint union of graphs of $V_1,\ldots,V_M\in\mathcal{G}^x_k$ and $V_{M+1},\ldots,V_{2M}\in\mathcal{G}^y_k$. Define $e(V)=(V_1,\ldots,V_{2M})$. This expression is not unique, but we make a choice for each $V\in\mathcal{K}^\Upsilon_{k}$. Also, most of the sets $V_i$ are actually empty and thus it is practical to consider $\varnothing$ as an element of $\mathcal{G}^x_k, \mathcal{G}^y_k$.

By \cite[Lemma 5.2]{ShmerkinSuomala12},  $\mathcal{G}^x_k$ and $\mathcal{G}^y_k$ satisfy \eqref{H:size-parameter-space} where the metric is the Hausdorff metric. If on $X$ we consider the $\ell_\infty$ metric $d_X$ (where each $\mathcal{G}^x_k,\mathcal{G}^y_k$ is endowed with the Hausdorff metric), it follows that $X$ satisfies \eqref{H:size-parameter-space} as well.

Now suppose that $V,W\in\mathcal{K}_k$ with $e(V)=(V_1,\ldots, V_{2M}), e(W)=(W_1,\ldots,W_{2M})$, such that for each $1\le i\le 2M$, the curves $V_i,W_i$ are $\e$-close in the Hausdorff metric. Suppose (for notational simplicity) that $i\le M$ and let $f\colon I_i\to\R$, $g\colon J_i\to\R$ be the functions with graphs $V_i,W_i$.

Using that $|f'|,|g'|$ are uniformly bounded, it follows that
\begin{equation}\label{eq:Hausdorff1}
|I_i\Delta J_i|+\sup_{x\in I_i\cap J_i}|f(x)-g(x)|=O(\e)\,.
\end{equation}
Denote by $\tilde{f}, \tilde{g}$ the restrictions of $f$ and $g$ to $I_i\cap J_i$, and by $\widetilde{V}_i\subset V_i,\widetilde{W}_i\subset W_i$ the graphs of $\tilde{f},\tilde{g}$ respectively. Let $\Pi_V^i\colon(0,\ell(\widetilde{V}_i))\to\widetilde{V}_i$, $\Pi_W^i\colon(0,\ell(\widetilde{W}_i))\to\widetilde{W}_i$ be continuous arclength parametrizations increasing in the $x$ variable.
A simple calculation (see \cite[Lemma 5.1]{ShmerkinSuomala12}) implies that
\begin{equation}\label{eq:C_1bound}
|\ell(\widetilde{V}_i)-\ell(\widetilde{W}_i)|=O(\e)
\end{equation}
and combining this with \eqref{eq:Hausdorff1} yields
\begin{equation}\label{eq:hausdorff2}
\sup_{0< t<\ell_i}
|\Pi_V^{i}(t)-\Pi_W^{i}(t)|=O(\e)\,,
\end{equation}
where $\ell_i=\min(\ell(\widetilde{V}_i),\ell(\widetilde{W}_i))$.

Let $L$ be a $1$-Lipschitz function supported on $\Upsilon$. Using \eqref{eq:Hausdorff1}--\eqref{eq:hausdorff2}, we have
\begin{align*}
&\left|\int_{V_i} L\,d\eta_V-\int_{W_i}L\,d\eta_W\right|=O(\e)+\left|\int_{t=0}^{\ell(V_i)} L(\Pi_{V}^i(t))\,dt-\int_{t=0}^{\ell(W_i)} L(\Pi_{W}^i(t))\,dt \right|\\
&\le O(\varepsilon)+\int_{t=0}^{\ell_i} \left|L(\Pi_{V}^i(t))-L(\Pi_{W}^i(t))\right|\,dt=O(\varepsilon)\,.\\
\end{align*}
Since
\[\left|\int_{V} L\,d\eta_V-\int_{W}L\,d\eta_W\right|\le\sum_{i=1}^{2M}\left|\int_{V_i} L\,d\eta_V-\int_{W_i}L\,d\eta_W\right|\]
and $M=O_k(1)$, we arrive at $d(V,W)=O(\varepsilon)$.

To summarize, we have shown that \eqref{H:size-parameter-space} holds in the product space $(X,d_X)$
and found an embedding $e\colon\Gamma\to X$ such that $e^{-1}\colon e(\Gamma)\to\Gamma$ is onto and Lipschitz, where $\Gamma=\{\eta_V|_\Upsilon\,:\,V\in\mathcal{K}_k\}$. This shows that \eqref{H:size-parameter-space} holds also on our original space $(\Gamma,d)$.

It remains to verify \eqref{H:Holder-a-priori}. We could apply Proposition \ref{prop:cutout-measures-satisfy-Holder-a-priori-2}, but it would be unnecessarily technical to construct the required parametrizations, so we give a direct proof instead. We assume that $(\mu_n)$ is of $(\mathcal{F},\tau,\zeta)$-cutout type, the proof in the cell type case is almost identical. Let $E_n$ be the event that $M_n\le C 2^{\zeta n}$. Then, a.s. there is $N_0$ such that $E_n$ holds for all $n\ge N_0$.

Given $n\ge N_0$, let $\Lambda_i\in\mathcal{F}$, $1\le i\le M_n$ such that $A_n=\Omega\setminus\cup_{i=1}^{M_n}\Lambda_i$. Pick $V,W\in\mathcal{K}_k$ with $d(V,W)<\varepsilon$. Denote $\Lambda_0=\Omega$. Lemma \ref{lem:porous} implies that for each $i=0,\ldots, M_n$, both $V\cap\{x\in\Lambda_i\,:d(x,\partial\Lambda_i)<\varepsilon^{1/2}\}$ and $W\cap\{x\in\Lambda_i\,:d(x,\partial\Lambda_i)<\varepsilon^{1/2}\}$ can be covered by $O_{C,k}(\varepsilon^{-\gamma_1/2}$) balls of radius $\varepsilon^{1/2}$ for some $0<\gamma_1=\gamma_1(C,k)<1$. Let $U_i$ be the union of these balls and $U=\cup_{i=0}^{M_n}U_i$.

Let $R=(V\cup W)\setminus (U\cup A_n)$ and
\[
L(x) = \frac{2^{\alpha n} \dist(x,R)}{\dist(x,R)+\dist(x,A_n)}\,.
\]
Then $L$ equals $\mu_n$ on $(V\cup W)\setminus U$, and $0\le L(x) \le 2^{\alpha n}$ for all $x$. Moreover, since the distance between $R$ and $A_n$ is $\Omega(\e^{1/2})$, a calculation shows that $L$ is Lipschitz with constant $O(\varepsilon^{-1/2}2^{\alpha n})$. Now
\begin{equation}\label{eq:Lipest1}
\begin{split}
&\int_V\mu_n\,d\eta_V-\int_W\mu_n\,d\eta_W\\
&=\int_{V}L\,d\eta_V-\int_{W}L\,d\eta_W
+\int_{U\cap V}(\mu_n-L)\,d\eta_V-\int_{U\cap W}(\mu_n-L)\,d\eta_W\,.
\end{split}
\end{equation}
But $\mathcal{H}^1(U\cap V),\mathcal{H}^1(U\cap W)=O(2^{\zeta n}\varepsilon^{(1-\gamma_1)/2})$ and $|\mu_n-L|\le 2^{n\alpha}$ so that
\begin{equation}
\left|\int_{U}(\mu_n-L)\,d\eta_V-\int_{U}(\mu_n-L)\,d\eta_W\right|=O\left(2^{(\zeta+\alpha)n}\varepsilon^{(1-\gamma_1)/2}\right)\,.
\end{equation}
On the other hand, the definition of $d$ yields
\begin{equation}\label{eq:Lipest3}
\left|\int_{V}L\,d\eta_V-\int_{W}L\,d\eta_W\right|=O\left(\varepsilon^{1/2} 2^{n\alpha}\right)\,.
\end{equation}
Putting \eqref{eq:Lipest1}--\eqref{eq:Lipest3} together, we have shown \eqref{H:Holder-a-priori} with $\gamma_0=\tfrac12(1-\gamma_1)$ and $\theta=\zeta+\alpha$.
\end{proof}

\begin{rem}
The only properties of $\mathcal{K}_k$ that were used in the proof are the size bound \eqref{H:size-parameter-space}, and the fact that the curves $V\cap\Upsilon$ have bounded length and are unions of $O(1)$ $h$-linearly approximable curves, for a function $h$ independent of $V\in\mathcal{K}_k$. Thus the theorem actually applies to much more general curve families.
\end{rem}

\subsection{Polynomial projections}
\label{sec:poly_projections}

In the linear setting of Theorem \ref{thm:linear-projections}, we were able to obtain a result on projections as a rather simple corollary of the corresponding result for intersections. A conceptually similar but more technical argument is available in the current setting, and allows us to prove the following, based on a variant of Theorem \ref{thm:alg_curves}.

\begin{thm} \label{thm:polynomial-projections}
  Let $(\mu_n)$ be an SI-martingale as in the statement of Theorem \ref{thm:alg_curves}. Then almost surely, for every non-constant polynomial $P\colon\R^2\to\R$,
  the image measure $P\mu_\infty$ is absolutely continuous and has a piecewise locally H\"{o}lder continuous density. The H\"{o}lder exponent and the number of discontinuities are quantitative in terms of $\deg P$ (and the parameters appearing in the assumptions on $(\mu_n)$).

  In particular, almost surely conditioned on $\mu_\infty\neq 0$, the projections $P(\supp\mu_\infty)$ have nonempty interior for all non-constant polynomials $P\colon\R^2\to\R$.
\end{thm}

Before proving this theorem, we present some preliminary results. Let $\widetilde{\cP}_k$ be the family of non-constant polynomials $\R^2\to\R$
of degree $\le k$.
Also write $\widetilde{\cP}_k^{\text{reg}}$ for the polynomials in $\widetilde{\cP}_k$ for which $0$ is a regular value (that is, $\nabla P(x)\neq 0$ for all $x\in P^{-1}(0)$). In the following we identify elements of $\widetilde{\cP}_k$ with their coefficients, and in this way see it as a subset of some Euclidean space. We also use the notation $|P-Q|$ for the Euclidean distance of the coefficients.

\begin{lemma} \label{lem:isotopic-curves}
Let $P\in\widetilde{\cP}_k^{\text{reg}}$, and let $\Upsilon$ be an open ball such that $P^{-1}(0)\cap \Upsilon\neq\varnothing$.

Then there exist neighbourhoods $U$ of $P$ (in $\widetilde{\cP}_k$) and $U'$ of $P^{-1}(0)\cap\Upsilon$, and a $C^1$ map $G:U\times U\times U'\to\R^d$  such that for all $Q_1,Q_2\in U$:
\begin{enumerate}
\item $G(Q_1,Q_1,\cdot)$ is the identity on $Q_1^{-1}(0)\cap U'$,
\item $G(Q_1,Q_2,\cdot)|_{Q_1^{-1}(0)\cap U'}$ is a diffeomorphism onto its image,\\ and $Q_2^{-1}(0)\cap\Upsilon\subset G(Q_1,Q_2,Q_1^{-1}(0)\cap U')$,
\item $Q_2(G(Q_1,Q_2,u))=0$ whenever $u\in Q_1^{-1}(0)\cap\Upsilon$.
\end{enumerate}
\end{lemma}
\begin{proof}
The idea is to define $G$ so that for each $u\in Q_1^{-1}(0)$, $G(Q_1,Q_2,u)$ is the point where the gradient flow $dx/dt=\nabla Q_1(x)$ started at $u$ hits $Q_2^{-1}(0)$. To be more precise, let $F(Q,u_0, \cdot)$ be the solution to the ODE $x'(t)=\nabla Q(x(t))$ with initial value $x(0)=u_0$. Since $0$ is a regular value of $P$ and $\overline{\Upsilon}$ is compact, this is a well defined $C^1$ function in some domain $\widetilde{U}\times \widetilde{U}'\times (-\delta,\delta)$, where $\widetilde{U}$ is a neighbourhood of $P$ and $\widetilde{U}'$ is a neighbourhood of $P^{-1}(0)\cap\overline{\Upsilon}$. In turn, using the implicit function theorem, the fact that $0$ is a regular value of $P$, and compactness, it follows that $Q_2(F(Q_1,u,t))=0$ defines an implicit $C^1$ function $t=t(Q_1,Q_2,u)$ for $(Q_1,Q_2,u)$ in some neighbourhood $U\times U\times U'$ as in the statement. The lemma follows with $G(Q_1,Q_2,u)=F(Q_1,u,t(Q_1,Q_2,u))$.
\end{proof}

Using this lemma, we can establish the following variant of Theorem \ref{thm:alg_curves}.

\begin{prop} \label{prop:modified-alg-curves}
Let $(\mu_n)$ be an SI-martingale satisfying the assumptions of Theorem \ref{thm:alg_curves}.

Let $P\in\widetilde{\cP}_k^{\text{reg}}$ such that $P^{-1}(0)\cap \supp\mu_0\neq\varnothing$. Then there exists a neighborhood $U$ of $P$ in $\widetilde{\cP}_k^{\text{reg}}$ such that, letting
\[
 \eta_Q = \frac{1}{|\nabla Q|} \,d\mathcal{H}^1|_{Q^{-1}(0)}\,,
\]
the following hold a.s.:
\begin{enumerate}
\item[\emph{(i)}] The sequence $Y_n^Q:=\int \mu_n \,d\eta_Q$ converges uniformly over all $Q\in U$; denote the limit by $Y^{Q}$.
\item[\emph{(ii)}] $|Y^{Q_1}-Y^{Q_2}|\le K\, |Q_1-Q_2|^{\gamma_k}$ for all $Q_1,Q_2\in U$,
\end{enumerate}
where $\gamma_k>0$ is a deterministic constant and $K$ is a finite random variable.
\end{prop}

\begin{proof}
Let $\Upsilon$ be an open ball as in Lemma \ref{lem:isotopic-curves} and let $U,U'$ be the neighbourhoods provided by Lemma \ref{lem:isotopic-curves}, which we can take to be compact. 
The first key observation is that
 \begin{equation} \label{eq:d-dominated-by-coeff-diff}
  d(V_{Q_1},V_{Q_2}) = O(|Q_1-Q_2|) \quad\text{for } Q_1,Q_2\in U\,,
 \end{equation}
 where the implicit constant depends on $U, U'$ (and $d$ is as in the Definition \ref{def:metric-algebraic}). This is a consequence of Lemma \ref{lem:isotopic-curves}. Indeed, fix $Q_1, Q_2\in U$, write $\e=|Q_1-Q_2|$, and let $L$ be a $1$-Lipschitz map supported on $\Upsilon$. Write $G(u)=G(Q_1,Q_2,u)$ for simplicity. By the second claim of Lemma \ref{lem:isotopic-curves},
\begin{align}\label{eq:Gchangeofvar}
\int_{V_{Q_2}} L\,d\mathcal{H}^1=\int_{V_{Q_1}\cap U'}L(G(u))|DG(u)|\,d\mathcal{H}^1(u)\,.
\end{align}
It follows from the first part of Lemma \ref{lem:isotopic-curves} that $G$ is $O(\e)$-close to the identity in the $C^1$ metric on $V_{Q_1}\cap U'$. Thus $|G(u)-u|=O(\varepsilon)$ and $|DG(u)|=1\pm O(\varepsilon)$ for $u\in V_{Q_1}\cap U'$. As $L$ is Lipschitz, with Lipschitz constant $1$ and $|L|=O(1)$, we have
\[\left|L(G(u))|G'(u)|-L(u)\right|=O(\varepsilon)\]
for $u\in V_{Q_1}\cap U'$ so that \eqref{eq:Gchangeofvar} yields
\begin{align*}
\left|\int_{V_{Q_2}} L\,d\mathcal{H}^1-\int_{V_{Q_1}} L\,d\mathcal{H}^1\right|=O(\e)\,,
\end{align*}
and \eqref{eq:d-dominated-by-coeff-diff} follows.

The result now follows as in the proof of Theorem \ref{thm:alg_curves} once we verify \eqref{H:Holder-a-priori}. To that end, we proceed as in the proof of Theorem \ref{thm:alg_curves} replacing \eqref{eq:Lipest1} by
\begin{align}\label{eq:Lipest2}
&\int_{V_{Q_1}}|\nabla Q_1|^{-1}\mu_n\,d\mathcal{H}^1-
\int_{V_{Q_2}}|\nabla Q_2|^{-1}\mu_n\,d\mathcal{H}^1\\
&=\int_{V_{Q_1}}|\nabla Q_1|^{-1}L\,d\mathcal{H}^1-\int_{V_{Q_2}}|\nabla Q_1|^{-1}L\,d\mathcal{H}^1+\int_{V_{Q_2}}L\left(|\nabla Q_1|^{-1}-|\nabla Q_2|^{-1}\right)\,d\mathcal{H}^1\notag\\
&\,\,\,+\int_{U\cap V_{Q_1}}|\nabla Q_1|^{-1}(\mu_n-L)\,d\mathcal{H}^1-\int_{U\cap V_{Q_2}}|\nabla Q_2|^{-1}(\mu_n-L)\,d\mathcal{H}^1\,.
\notag
\end{align}
Since $|\nabla Q_1|^{-1}L$ and $|\nabla Q_2|^{-1}L$ are $O(1)$-Lipschitz, and
\[\left|\nabla Q_1|^{-1}-|\nabla Q_2|^{-1}\right|=O(\e)\,,\]
for $x\in U'$, the $O(\e^{\gamma_0})$ bound for \eqref{eq:Lipest2} is obtained as in the proof of Theorem \ref{thm:alg_curves}.
\end{proof}

\begin{proof}[Proof of Theorem \ref{thm:polynomial-projections}]
 The core of the proof is to show that if $P\in\widetilde{\cP}_k$, then $P\mu_\infty$ is locally $\gamma_k$-H\"{o}lder continuous outside of $O_k(1)$ points of possible discontinuity. Once this is done, absolute continuity follows from the fact that $P\mu_\infty$ has no atoms, which is a consequence of Theorem \ref{thm:alg_curves} (which in particular implies that $Y^V$ is finite for $V=P^{-1}(t)\cap\Upsilon$). The statement on nonempty interior of the sets $P(\supp\mu_\infty)$ is immediate from the piecewise continuity of the density of $P\mu_\infty$.

 Cover the space of polynomials $P\in \widetilde{\cP}_k$  with no critical points in $P^{-1}(0)\cap \supp\mu_0$ such that $P^{-1}(0)\cap\supp\mu_0\neq\varnothing$ by countably many open sets $U_{k,j}$ such that Proposition \ref{prop:modified-alg-curves} applies in each of these neighbourhoods. We condition on the full probability event that the conclusion of Proposition \ref{prop:modified-alg-curves} holds for all $k$ and $j$.

 Fix $P\in\widetilde{\cP}_k$. Since the set $\text{crit}(P)$ of critical values of $P$ is finite (and has $O_k(1)$ elements), we will be done once we establish that $P\mu_\infty$ is absolutely continuous with a $\gamma_k$-H\"{o}lder continuous density in a neighbourhood of each $t_0\in\R\setminus \text{crit}(P)$. If we fix such $t_0$, then either $P^{-1}(t_0)\cap \supp\mu_0=\varnothing$, in which case there is nothing to do, or $P-t\in U_{j,k}$ for some $j$ and $t$ in a compact neighbourhood $I$ of $t_0$.

 After these reductions, the proof is similar to that of Theorem \ref{thm:linear-projections}(iii), except that we need to use the co-area formula instead of Fubini. Fix $t\in I^\circ$. For small $r>0$, the co-area formula (see \cite[Lemma 1 in Section 3.4]{EvansGariepy92}) gives that for all $n$,
\begin{align*}
P\mu_n(t-r,t+r)&=\int_{P^{-1}([t-r,t+r])}\mu_n(x) dx\\
&=\int_{P^{-1}([t-r,t+r])}\mu_{n}|\nabla P|^{-1}|\nabla P|\\
&=\int_{u=t-r}^{t+r} \int_{P^{-1}(u)}\mu_n |\nabla P|^{-1}\,d\mathcal{H}^1\, du\\
&=\int_{u=t-r}^{t+r}Y_{n}^{P-u}\,du\,,
\end{align*}
where $Y_n^{P-u}$ is as in the statement of Proposition \ref{prop:modified-alg-curves}. Since we are assuming the conclusion of this proposition holds, the maps $u\mapsto Y_n^{P-u}$ are uniformly bounded on $I$, and converge uniformly to $Y^{P-u}$. This implies that $P\mu_\infty|_{I^\circ}$ is absolutely continuous with density $Y^{P-u}$. Since this density is H\"{o}lder continuous with exponent $\gamma_k$, we are done.
\end{proof}

\begin{rem} \label{rem:poly-projections-percolation}
 Theorem \ref{thm:polynomial-projections} does not hold for $\mu_n^{\text{perc}}$ or any construction involving shapes with flat pieces. Indeed, the principal projections of fractal percolation are \emph{not} piecewise continuous. However, Theorem \ref{thm:polynomial-projections} does hold for $\mu_n^{\text{perc}}$ if we consider only polynomials of degree $\ge 2$ for which $\partial_x P$ and $\partial_y P$ are not identically zero. We skip the details, but explain the required changes briefly: We first observe that for any such polynomial $P$, $P-u$ has at most $O_{\deg P}(1)$ factors of the form $x-c$ or $y-c$ (since $x-c$ is a factor of $\partial_y P$ whenever it is  a factor of $P$, and likewise for $y-c$).

 After this, the only change required in the proofs comes when we verify condition \eqref{H:Holder-a-priori}, where instead of Lemma \ref{lem:porous} we use the following fact, which follows from Bezout's Theorem: for any algebraic curve $V=P^{-1}(0)\cap\Upsilon$ of degree $k$ which does not contain a factor of the form $x-c$ or $y-c$, its intersection with the $\e$-neighbourhood of a vertical or horizontal line can be covered by $C\e^{-\gamma_1}$ balls of radius $\e$ (where $\gamma_1\in (0,1)$ depends only on $k$ and $C<\infty$ is independent of the line and is uniform in a neighbourhood of $V$).

 In particular, it still holds that if $\alpha<1$, then a.s. conditioned on non-extinction, all polynomial projections of $\supp\mu_\infty^{\text{perc}(2,\alpha)}$ have nonempty interior: linear projections are dealt with already in \cite{RamsSimon13} and higher degree ones follow from these considerations.
\end{rem}

\section{Intersections with self-similar sets and measures}
\label{sec:self-similar}

The previous two sections deal with the natural volume (Hausdorff measure) on affine subspaces and algebraic curves. In this section, we apply Theorem \ref{thm:Holder-continuity} in a different direction: we consider intersections with parametrized families of self-similar sets and measures. We recall some basic definitions below, and refer the reader to \cite[Chapter 9]{Falconer03} for further background.

Let $m\ge 2$, and let $\Sigma_m$ be the full shift on $m$ symbols. Given $F=(f_1,\ldots,f_m)\in(\simi_{d}^c)^m$, let $\Pi_F$ denote the induced projection map $\Sigma_m\to\R^d$ (recall the notation from Section \ref{sec:notation}). We introduce some standard notation for words. Let $\Sigma_m^*$ denote the set of finite words with entries in $\{1,\ldots,m\}$. For $\iii\in\Sigma_m^*$, write $|\iii|$ for its length, and denote
\[
[\iii]=\{\jjj\in\Sigma_m\,:\,\jjj\text{ begins with }\iii\}\,.
\]
The juxtaposition of $\iii_1,\ldots,\iii_n\in\Sigma_m^*$ will be denoted $(\iii_1\ldots\iii_n)$.

If $F\in (\simi_{d}^c)^m$ and $\iii=(i_1\ldots i_n)\in\Sigma_m^*$, we write $F_\iii= F_{i_1}\circ\cdots\circ F_{i_n}$.
The similarity ratio of a similarity map $f$ will be denoted $\rho(f)$.

If $F\in (\simi_d^c)^m$, its \textbf{similarity dimension} $\dim_S(F)$ is the only positive solution $s$ to $\sum_{i=1}^m \rho(F_i)^s=1$. One always has a bound $\dim_H(E_F)\le \dim_S(F)$, with equality if the open set condition holds.

On $(\simi_{d}^c)^m$, we always consider the metric
\[d_1(F,G)=|\Pi_F-\Pi_G|_\infty\,.\]
Let
\[
P_m=\left\{(p_1,\ldots,p_m)\,:\,p_i>0\,,\, \sum_{i=1}^{m}p_i=1\right\}\,,
\]
denote the collection of probability vectors of length $m$ endowed with the maximum metric $d_2(p,p')=\max_{i}|p_i-p'_i|$. Define a metric $d$ on $(\simi_{d}^c)^m\times P_m$ as $d((F,p),(G,p'))=d_1(F,G)+d_2(p,p')$. Given $t=(F,p)\in\Gamma$, let $\nu_t$ be the product measure $p^\infty$  on $\Sigma_m$, i.e. $\nu_t$ satisfies
\[
\nu_t[i_1\ldots i_l]=\prod_{k=1}^l p_{i_k}\,.
\]
Let $\eta_t$ be the self-similar measure $\eta_t=\Pi_F(\nu_t)$. If $F$ satisfies the open set condition, we say that $\eta_t$ is the \textbf{natural measure} if $p_i=\rho(F_i)^s$, where $s=\dim_S(F)$. In this case it is well known that $\eta_t(B(x,r))=O_F(r^s)$ for all $x\in\R^d, r>0$.

\begin{thm}\label{thm:self_sim}
Suppose $(\mu_n)$ is an SI-martingale, which either is of $(\mathcal{F},\tau,\zeta)$-cell type with $\mu_n\le C 2^{\alpha n}$, or is of $(\mathcal{F},\alpha,\zeta)$-cutout type.

Let $\Gamma$ be a bounded subset of $(\simi_d^{c})^m\times P_m$ whose elements consist of contractive similitudes with contraction ratios bounded uniformly away from $0$ and $1$, and
such that for some $s>\alpha$,
\begin{equation}\label{eq:dim_selfsimilar_measure}
\eta_t(B(x,r))=O(r^s)\text{ for all }x\in\R^d, t\in\Gamma, 0<r<1\,.
\end{equation}

Assume furthermore that for some $0<\gamma_1<s$, all $0<\varepsilon<1$, all $(F,p)\in\Gamma$, and all $\Lambda\in\mathcal{F}$, the set  $\partial\Lambda(\varepsilon)\cap E_F$ can be covered by $O(\varepsilon^{-\gamma_1})$ balls of radius $\varepsilon$.

Then there are a quantitative $\gamma>0$ and a finite random variable $K$ (the latter depending on the implicit constants in the $O(\cdot)$ notation) such that:
\begin{enumerate}
\item[\emph{(i)}]\label{claim:ss1} The sequence $Y_n^t:=\int\mu_n(x)\,d\eta_t$ converges uniformly over all $t\in\Gamma$; denote the limit by $Y^t$.
\item[\emph{(ii)}] $|Y^t-Y^u|\le K\, d(t,u)^\gamma$ for all $t,u\in\Gamma$.
\end{enumerate}
\end{thm}

\begin{proof}
It is straightforward to check that \eqref{H:size-parameter-space} holds and, since \eqref{H:dim-deterministic-measures} and \eqref{H:codim-random-measure} are part of the assumptions, the claim follows from Theorem \ref{thm:Holder-continuity} if we can verify \eqref{H:Holder-a-priori}. To that end, we will apply Proposition \ref{prop:cutout-measures-satisfy-Holder-a-priori-2}, and thus it is enough to check that \eqref{eq:Holder-shape-2} and \eqref{eq:Holder-measure} are satisfied.

Fix $\delta>0$, and let $t=(F,p),u=(G,q)\in\Gamma$ with $d(t,u)<\delta$. It follows from the definition of $d_1$ that
$E_F\cap \Pi_F(\Pi_F^{-1}(\Lambda)\setminus\Pi_G^{-1}(\Lambda))$ and $E_F\cap \Pi_F(\Pi_G^{-1}(\Lambda)\setminus\Pi_F^{-1}(\Lambda))$ are both contained in $\partial\Lambda(\delta)\cap E_F$, and can thus be covered by $O(\delta^{-\gamma_1})$ balls of radius $\delta$. Combining this with the Frostman condition \eqref{eq:dim_selfsimilar_measure}, we get
\[
\nu_F\left(\Pi_F^{-1}(\Lambda)\Delta\Pi_G^{-1}(\Lambda)\right)\le 2\eta_t(E_F\cap\partial\Lambda(\delta))=O(\delta^{s-\gamma_1})\,.
\]
Thus \eqref{eq:Holder-shape-2} holds with $\gamma_0=s-\gamma_1$.

To prove \eqref{eq:Holder-measure}, assume first that $F=G$. Denote $\Pi=\Pi_F=\Pi_G$, and let $J=\cup_{i=1}^M\Pi^{-1}(\Lambda_i\cap\Omega)$.  Since all the similitudes involved are uniformly contracting, it follows by a simple covering argument on $\Sigma_m$ that there is a constant $C<\infty$, and for each $n$, a collection $\Xi=\Xi_n$ of $O(2^{C n})$ finite words $\iii$ such that the cylinder sets $[\iii]\subset J$, $\iii\in\Xi$,
are pairwise disjoint, $\diam(F_\iii(E_F))=O(2^{-n})$ for all $\iii\in\Xi$ and further
\begin{equation}\label{eq:away_from_bdr}
\Pi(J\setminus\cup_{\iii\in\Xi}[\iii])\subset\partial\Omega(2^{-n})\cup\bigcup_{i=1}^M \partial \Lambda_i(2^{-n})\,.
\end{equation}
Let $n$ be the smallest integer such that
$n>\tfrac{\log_2\delta}{-2C}$.

From $d_2(p,q)\le\delta$ and the definition of $d_2$, it follows by simple calculus that
$|\nu_t([\iii])-\nu_u([\iii])|=O(n)\delta=O(\delta|\log\delta|)$ for all $\iii\in\Xi$. Thus
\begin{equation}\label{eq:inner_part}
\left|\nu_u(\cup_{\iii\in\Xi}[\iii])-\nu_t(\cup_{\iii\in\Xi}[\iii])\right|\,\le\,\sum_{\iii\in\Xi}|\nu_u[\iii]-\nu_t[\iii]|\le O\left(2^{Cn}|\log\delta|\delta\right)=O(\delta^{\gamma})\,,
\end{equation}
for any $\gamma<\tfrac12$.

By our assumptions, each of the sets $\partial\Lambda_i(2^{-n})\cap E_F$ as well as $\partial\Omega(2^{-n})\cap E_F$ may be covered by $O(2^{n\gamma_1})$ balls of radius $2^{-n}$. Together with \eqref{eq:dim_selfsimilar_measure} and \eqref{eq:away_from_bdr}, this yields
\begin{align*}
&\left|\nu_u(J\setminus \cup_{\iii\in\Xi}[\iii])-\nu_t(J\setminus\cup_{\iii\in\Xi}[\iii])\right|\le\left|\eta_u\Pi(J\setminus \cup_{\iii\in\Xi}[\iii])|+|\eta_t\Pi(J\setminus\cup_{\iii\in\Xi}[\iii])\right|\\
&\le\eta_u(\partial\Omega(2^{-n}))+\eta_t(\partial\Omega(2^{-n}))+
\sum_{i=1}^M\eta_u(\partial\Lambda_i(2^{-n}))+\eta_t(\partial\Lambda_i(2^{-n}))\\
&=O(M)2^{n(\gamma_1-s)}=O(M)\delta^{(s-\gamma_1)/(2C)}\,.
\end{align*}
Together with \eqref{eq:inner_part}, this implies \eqref{eq:Holder-measure} in the case $F=G$. If $p=q$, \eqref{eq:Holder-measure} is trivially satisfied, and the general case then follows by the triangle inequality.
\end{proof}

\begin{rems}\label{rem:d-1}
\begin{enumerate}[(i)]
\item By choosing $k$-dimensional cubes $E_V\subset V$ such that $\supp\mu_0\cap V\subset E_V$ as the self-similar sets, we can view Theorem \ref{thm:linear-projections} as a special case of Theorem \ref{thm:self_sim}. This illustrates the generality of Theorem \ref{thm:self_sim}.
\item Theorem \ref{thm:self_sim} and its proof remain valid if we replace $(\simi_{d}^c)^m$ by any totally bounded family (in the metric $|\Pi_F-\Pi_G|_\infty$) of uniformly contractive IFSs satisfying \eqref{H:size-parameter-space} and the Frostman condition \eqref{eq:dim_selfsimilar_measure}.
\item If the boundaries of $\Lambda\in\mathcal{F}$ satisfy a uniform box-dimension estimate
\[
\text{Each }\partial\Lambda(\e)\text{ is covered by }C\e^{-\gamma_1}\text{ balls of radius }\e\,,
\]
then the same condition holds trivially for the coverings of $E_F\cap\partial\Lambda(\e)$ as well, and
Theorem \ref{thm:self_sim} holds provided that $s>\max(\gamma_1,\alpha)$. In particular,
 Theorem \ref{thm:self_sim} can be applied to $\mu_n^{\text{perc}(\alpha,d)}$ and $\mu_n^{\text{ball}(\alpha,d)}$ provided $s>\max(d-1,\alpha)$, and to $\mu_n^{\text{snow}(\alpha)}$ if $s>\max(\tfrac{\log 4}{\log 3},\alpha)$.
\end{enumerate}
\end{rems}

In the following proposition, we verify that the covering condition  in Theorem \ref{thm:self_sim} holds for all values of $s$ when $E_F$ is not contained in a hyperplane, and $\Lambda$ is either a circle or a polyhedron. This will allow us to deduce consequences on the intersections of $\mu_n^{\text{ball}}$ and $\mu_n^{\text{perc}}$ with such self-similar sets.
\begin{prop} \label{prop:inter-nbhd-sss}
Let $F\in(\simi_{d}^c)^m$ be an IFS such that $E_{F}$ is not contained in a hyperplane.

Then there exist $\gamma\in (0,\dim_S(F))$, $C>0$, and a neighbourhood $\mathcal{G}$ of $F$ in $(\simi_{d}^c)^m$, such that for each $G\in\mathcal{G}$, if $S$ is a sphere or a hyperplane,
then $S(\e)\cap E_G$ can be covered by $C\, \e^{-\gamma}$ balls of radius $\e$.
\end{prop}
\begin{proof}

Let $\mathcal{S}$ denote the collection of all spheres and hyperplanes of $\R^d$.
To begin with, we note that for each $S\in\mathcal{S}$, there exist $\delta>0$ and
$\iii\in\Sigma_m^*$ such that
$\dist(F_\iii(E_F), S)>2\delta$. Recall that a self-similar set not contained in a hyperplane cannot be contained in a sphere (see e.g. \cite[Theorem 1]{BandtKravchenko11}).
A straightforward compactness argument shows that $\delta$ and $|\iii|$ can be taken uniform in $S$. By iterating the original IFS, we can assume (for notational simplicity) that $|\iii|=1$ for all spheres and hyperplanes. Notice that the claim for the original IFS follows from the corresponding claim for the iterated one.

Since $G\mapsto E_G$ is continuous in the Hausdorff metric, one can find a neighbourhood $\mathcal{G}$ of $F$, such that for any $S\in\mathcal{S}$ and any $G\in\mathcal{G}$, there is $i\in\{1,\ldots,m\}$ such that $G_i(E_G)\cap S=\varnothing$.
Let $\overline{\rho}_i = \sup_{G\in\mathcal{G}} \rho(G_i)$, and relabel so that $\overline{\rho}_m = \min_{i=1}^m \overline{\rho}_i$. By making $\mathcal{G}$ smaller if needed, we may and do assume that the only positive solution $\gamma$ to $\sum_{i=1}^{m-1} \overline{\rho}_i^\gamma=1$ satisfies $\gamma<\dim_S(F)$.

Write $N(E,\e)$ for the minimum number of balls of radius $\e$ required to cover the bounded set $E\subset\R^d$. An easy geometric argument shows that if, given a bounded set $E$, it holds that $N(E\cap S,\e) \le C\,\e^{-\gamma}$ for all $S\in\mathcal{S}$, and some $\e,C>0$ independent of $H$, then $N(S(\e)\cap E,\e) \le C' \e^{-\gamma}$ for all $S\in\mathcal{S}$, where  $C'=C'(C,d)$.
Hence it will be enough to show that
\[
 N(E_G\cap S) = O(\e^{-\gamma})
\]
for all $G\in\mathcal{G}$, and all
$S\in\mathcal{S}$.
Fix $G\in\mathcal{G}$ and $S\in\mathcal{S}$ for the rest of the proof.

The key observation is that for any word $\iii\in\{1,\ldots,m\}^n$, there is some $j=j(\iii)\in\{1,\ldots,m\}$ such that $G_\iii G_j(E_G)\cap S=\varnothing$. Indeed, this is equivalent to $G_j(E_G)\cap G_{\iii}^{-1}(S)=\varnothing$, which is known to hold for some $j\in\{1\ldots,m\}$ as $G_{\iii}^{-1}(S)\in\mathcal{S}$.
Using this observation, we define a subset $\Delta$ of the full shift $\Sigma_m$ as follows: if $\iii\in\Sigma_m$, then $\iii\in\Delta$ if and only $i_{n+1}\neq j(i_1\ldots i_n)$ for all $n$. Then we have $E_G\cap S\subset \Pi_G(\Delta)$.

Write $\Delta^*$ for all finite words which are beginnings of words in $\Delta$. Given $\e>0$, let $\Delta_\e$ be the set of finite words $\iii=(i_1,\ldots,i_n)\in\Delta^*$ such that $\overline{\rho}_{i_1}\cdots \overline{\rho}_{i_n}<\e$ but $\overline{\rho}_{i_1}\cdots \overline{\rho}_{i_{n-1}}\ge \e$. Note that if $(i_1,\ldots,i_n)\in\Delta_\e$, then $\overline{\rho}_{i_1}\cdots \overline{\rho}_{i_n}\ge\overline{\rho}_m \e$. Also, the cylinders $\{[\iii]:\iii\in\Delta_\e\}$ form a disjoint minimal cover of  $\Delta$. These facts, together with an inductive argument in the length of the longest word in $\Delta_\e$, imply that
\[
\#\Delta_\e (\overline{\rho}_m \e)^\gamma  \le \sum_{\iii\in\Delta_\e} \left(\overline{\rho}_{i_1}\cdots \overline{\rho}_{i_{|\iii|}}\right)^\gamma  \le 1\,.
\]
On the other hand, $\Pi_G(\Delta)$ is covered by the cylinders $\{ \Pi_G([\iii]):\iii\in\Delta_\e\}$, each of which has diameter $O(\e)$. Since $E_G\cap S\subset \Pi_G(\Delta)$, the proof is complete.
\end{proof}

The following lemma will allow us to deal with self-similar sets with no separation assumptions.

\begin{lemma} \label{lem:reduction-to-ssc}
Let $F\in (\simi_d^c)^m$ be an IFS such that $E_F$ is not contained in a hyperplane and hits a nonempty open set $U$. Then for every $\e>0$ there exists a finite subset $\Delta_\e$ of $\Sigma_m^*$ such that, writing $F_0=\{ F_\iii:\iii\in\Delta_\e\}$, the set $E_{F_0}$ is not contained in a hyperplane, hits $U$, and further $\dim_S F_0=\dim_H F_0>\dim_H F-\e$. Moreover, $F_0$ satisfies the strong separation condition.
\end{lemma}
\begin{proof}
This is well known if we do not impose the extra requirements of not being contained in a hyperplane and hitting $U$; the proof is a variant of the standard argument. Since $E_F$ meets the open set $U$, there is $\jjj_0\in \Sigma_m^*$ such that $F_{\jjj_0}(E_F)\subset U$. Also, since $E_F$ is not contained in the hyperplane, arguing by compactness as in the proof of Proposition \ref{prop:inter-nbhd-sss}, we find a finite subset $\Xi=\{\iii_1,\ldots,\iii_M\}\subset\Sigma_m^*$ such that:
\begin{equation}\label{eq:outH}
\text{For each hyperplane }H\text{ there is }\iii\in\Xi\text{ with }F_\iii(E_F)\cap H=\varnothing\,.
\end{equation}
Moreover, we may assume that the sets $F_\iii(E_F), \iii\in\Xi$ are pairwise disjoint by iterating the following fact: If $F_\iii(E_F)\neq F_\jjj(E_F)$, then there are words $\iii',\jjj'\in\Sigma_{m}^*$ such that $F_{\iii \iii'}(E_F)\cap F_{\jjj \jjj'}(E_F)=\varnothing$.

For each nonempty $\iii\in\Sigma_m^*$, let $\iii^{-}$ be $\iii$ with the last coordinate removed. For each $r\in (0,1)$, let $\overline{\Delta}_r$ be a subset of
\[
\{ \iii\in\Sigma_m^*: \rho(F_\iii)<r, \rho(F_{\iii^-})\ge r\}
\]
such that $\{F_\iii(E_F)\}_\iii\in\overline{\Delta}_r$ is pairwise disjoint, and is maximal with this property. Then it is easy to check that
\begin{equation}
\label{eq:dim_delta}
\#\overline{\Delta}_r=\Omega_\e(r^{-(s-\e)})\,,
\end{equation}
where $s=\dim_B(E_F)=\dim_H(E_F)$ (recall that Hausdorff and box dimensions always coincide for self-similar sets, see \cite[Corollary 3.3]{Falconer97}). Enumerate $\overline{\Delta}_r = \{ \jjj_i\}_{i=1}^{M_r}$. In light of \eqref{eq:dim_delta}, by taking $r$ small enough, we may assume that $M_r\ge M$.
Let $\overline{\Delta}'_r$ be the family obtained by replacing $\jjj_i$ by $\iii_i \jjj_i$ for $i\in\{1,\ldots,M\}$, and replacing $\jjj_i$ by $\iii_1\jjj_i$ for $i\in\{M+1,\ldots,M_r\}$. Finally, define $F_0^r = \{ F_{\jjj_0\iii}: \iii\in  \overline{\Delta}'_r\}$.

It follows that $F_0^r$ is contained in $U$, is not contained in any hyperplane, satisfies the strong separation condition, contains $\Omega_\e(r^{-(s-\e)})$ maps, and the contraction ratio of each map is $\Omega(r)$ (where the implicit constant depends on $\rho(F_{\jjj_0})$ and $\min\{ \rho(\iii):\iii\in\Xi\}$). A straightforward calculation that we omit allows us to conclude that if $r$ is small enough then $\dim_S(F_0^r)>s-2\e$.  Thus, taking $F_0=F_0^r$ for suitably small $r>0$ finishes the proof.
\end{proof}

We obtain the following consequence on the intersections with self-similar sets. We state it only for $\mu_n^{\text{ball}}$ and $\mu_n^{\text{perc}}$, although it will be clear from the proof that it applies to many other SI-martingales. The size of these intersections will be investigated more carefully in Theorems \ref{thm:dim_intersection_selfsim} and \ref{thm:dim_lower_bound}.

\begin{thm} \label{thm:intersection-compact-sss}
 Let $(\mu_n)$ be either $\mu_n^{\text{ball}(\alpha,d)}$ or $\mu_n^{\text{perc}(\alpha,d)}$.

 Let $\Gamma$ be any compact subset of $(\simi_d^{c})^m$ such that, for all $F\in \Gamma$, the self-similar set $E_F$ is not contained in a hyperplane, hits the interior of $\supp\mu_0$, and has Hausdorff dimension $>\alpha$. Then there is a positive probability that $\supp\mu_\infty \cap E_F\neq \varnothing$ for all $F\in\Gamma$.
\end{thm}
\begin{proof}

 We recall a version of the FKG inequality valid for $\mu_n^{\text{ball}}$ and $\mu_n^{\text{perc}}$. A random variable $X$ depending on the realization of a Poisson point process is called \textbf{decreasing} if $X(\mathcal{Y}')\ge X(\mathcal{Y})$ whenever $\mathcal{Y},\mathcal{Y}'$ are two realizations of the process with $\mathcal{Y}'\subset\mathcal{Y}$. Then $\EE(X_1 X_2)\ge \EE(X_1)\EE(X_2)$ whenever $X_1,X_2$ are both decreasing, see e.g. \cite[Lemma 2.1]{Janson84}. Recall that fractal percolation can also be interpreted as a Poisson point process (with a discrete intensity measure) where $\mathcal{Y}$ corresponds to the family of ``removed cubes''.

 Thanks to the FKG inequality  and the compactness of $\Gamma$,  it is enough to show that for each $F\in \Gamma$, there is a neighbourhood $\mathcal{G}$ of $F$ such that the claim holds for $\mathcal{G}$ in place of $\Gamma$. Indeed, the indicator of the claimed event holding (for $\mathcal{G}$) is clearly decreasing.

 Moreover, thanks to Lemma \ref{lem:reduction-to-ssc} we may assume without loss of generality that $F$ satisfies the strong separation condition. Indeed, any neighbourhood of the IFS $\{ F_\iii:\iii\in \Delta_\e\}$, where $\Delta_\e$ is the set given by Lemma \ref{lem:reduction-to-ssc}, contains the IFS $\{ G_\iii:\iii\in\Delta_\e\}$ for $G$ in a neighbourhood of $F$.

 Thus, from now we fix $F$ satisfying the strong separation condition so that, in particular, $\dim_H(E_F)=\dim_S(F)$. Let $\mathcal{G}$ be the neighbourhood given by Proposition \ref{prop:inter-nbhd-sss}. By making $\mathcal{G}$ smaller if needed, we may assume that all $G\in\mathcal{G}$ satisfy the strong separation condition as well. Let $s=\inf\{\dim_H E_G:G\in\mathcal{G}\}$; by making $\mathcal{G}$ even smaller, we may assume $s>\alpha$. For $G\in\mathcal{G}$, let $\nu_G$ be the natural self-similar measure. Thanks to the strong separation condition, these measures satisfy a uniform Frostman bound $\nu_G(B(x,r)) = O(r^{\dim_H E_G}) = O(r^s)$. Recalling Proposition \ref{prop:inter-nbhd-sss}, the proof is finished by virtue of Lemma \ref{lem:Yt-survives} and Theorem \ref{thm:self_sim}.
\end{proof}

\begin{rem}
Compactness of $\Gamma$ is clearly crucial since the limit sets in question are nowhere dense so, for example, they will not hit many self-similar sets of sufficiently small diameter. The assumption that the self-similar set hits $\supp\mu_0$ is also trivially necessary.
\end{rem}

\section{Dimension of projections: applications of Theorem \ref{thm:small_dimension_projections}}
\label{sec:dim-of-projections}

In this section we present applications of Theorem \ref{thm:small_dimension_projections}, focusing on orthogonal projections. In particular, we obtain sufficient conditions for the orthogonal projections of $\mu_\infty$ onto all planes $V\in\mathbb{G}_{d,d-k}$ to have dimension $d-\alpha$, provided $d-\alpha<d-k$.

\begin{thm}\label{thm:cor_dim_small}
Suppose that the assumptions of Theorem \ref{thm:small_dimension_projections} hold with $s=k$ and $\theta$ arbitrarily close to $\alpha-k$, for the family $\Gamma=\{ \mathcal{H}^k|_V: V\in \AA_{d,k}, V\cap \supp\mu_0\neq\varnothing\}$.
Then, almost surely
\begin{equation}
\ldimloc(P_W\mu_\infty,x)\ge d-\alpha\,.
\end{equation}
for all $W\in\mathbb{G}_{d,d-k}$ and $P_W\mu_\infty$-almost all $x\in W$.
\end{thm}

\begin{proof}
Let $\theta>\alpha-k$. It follows from Theorem \ref{thm:small_dimension_projections} that, almost surely,
\begin{equation*}
\sup_{n\in\N,V\in\mathbb{A}_{d,k}}2^{-\theta n}Y_n^{V}<\infty\,.
\end{equation*}
Thus, a.s. there is $K<\infty$ such that for all $W\in\mathbb{G}_{d,d-k}$, $x\in W$, $n\in\N$,
\begin{align}\label{eq:mu_n_bound}
P_W\mu_n(B(x,2^{-n}))=\int_{y\in B(x,2^{-n})\subset W}Y_n^{V_y}\,dy\le K 2^{n(-d+k+\theta)}\,,
\end{align}
where $V_y$ is the affine $k$-plane orthogonal to $W$ passing through $y$.

We claim that for all $n\in\N$, $M<\infty$ and each $V\in\mathbb{A}_{d,k}$, we have
\begin{equation}\label{eq:mu_n_claim}
\begin{split}
&\PP\left(\mu_\infty(V(2^{-n}))>M\max(\mu_n(V(2^{-n})),2^{n(-d+k+\theta)})\right)\\
&\le\exp(-f(M)2^{n(k+\theta-\alpha}))\,,
\end{split}
\end{equation}
where $f(M)\longrightarrow\infty$ as $M\rightarrow\infty$. Before proving \eqref{eq:mu_n_claim}, let us show that this indeed implies our claim. For each $n$, there is a finite subset of $\mathbb{A}_{d,k}$ with at most $O(1)^n$ elements such that each $P_W^{-1}(B(x,2^{-n}))\cap\supp\mu_0$ is covered by $O(1)$ of the tubes $V(2^{-n})$, where $V$ are from this family. It thus follows from \eqref{eq:mu_n_claim} that
\begin{align*}
&\PP\left(\mu_\infty(V(2^{-n}))> M\max(\mu_n(V(2^{-n})),2^{n(-d+k+\theta)})\text{ for some }V,n\right)\\
&\le\sum_{n=1}^\infty O(1)^n\exp(-f(M)2^{n(k+\theta-\alpha)})\longrightarrow 0\quad\text{ as } M\longrightarrow\infty.
\end{align*}
Combining this with \eqref{eq:mu_n_bound} yields that a.s. there is $K'<\infty$ such that
\[P_W\mu_\infty(B(x,2^{-n}))\le K'\, 2^{n(-d+k+\theta)}\]
for all $n\in\N$, $W\in\mathbb{G}_{d,d-k}$ and $x\in W$. This implies that
\[
\ldimloc(P_W\mu_\infty,x)\ge d-k-\theta \quad\text{ for all } W\in\mathbb{G}_{d,d-k} \text{ and } P_W\mu_\infty\text{-almost all } x\in W\,.
\]
Letting $\theta\to \alpha-k$ along a subsequence concludes the proof.

It thus remains to prove \eqref{eq:mu_n_claim}. For this, we condition on $\mathcal{B}_n$, and let
\[
\mu_n(V(2^{-n}))= L 2^{n(-d+k+\theta)}\,.
\]
For $m\ge n$, let $Z_m=\min(1,L^{-1})2^{n(d-k-\theta)}\mu_{m}(V(2^{-n}))$ and
\[
\kappa_m=a^{-1}(m+1-n)^{-2}\log M\,,
\]
where $a=\sum_{j=1}^\infty j^{-2}$. Applying Hoeffding's inequality as in the proof of Lemma \ref{lem:large-deviation}, we get
\[
\PP\left(Z_{m+1}-Z_m\ge\kappa_m\sqrt{Z_m}\right)=O\left(\exp\left(-\Omega(\log M)\tfrac{1}{(m+1-n)^4}2^{n(-d+k+\theta)} 2^{m(d-\alpha)}\right)\right)\,.
\]
Finally, summing over all $m\ge n$ yields \eqref{eq:mu_n_claim}.
\end{proof}

\begin{cor}\label{cor:dim_projections}
Let $(\mu_n)$ be an SI-martingale that is of $(\mathcal{F},\alpha,\zeta)$-cutout type, or of $(\mathcal{F},\tau,\zeta)$-cell type with $\mu_n(x)\le C 2^{n\alpha}$,  for some $\alpha\ge k$. Suppose that
there are constants $0<\gamma_0,C<\infty$ such that for all $V\in\mathbb{A}_{d,k}$, all $\varepsilon>0$ and any isometry $f$ which is $\varepsilon$-close to the identity, we have:
\begin{equation} \label{eq:non-flat-dim}
\mathcal{H}^k\left(V\cap \Lambda\setminus f^{-1}(\Lambda)\right)\le C\, \varepsilon^{\gamma_0} \quad\text{for all } \Lambda\in\mathcal{F}.
\end{equation}
Then, almost surely,
\begin{align}\label{eq:P_mu_dim}
\ldimloc(P_W\mu,x)\ge d-\alpha \text{ for all } W\in \mathbb{G}_{d,d-k} \text{ and $P_W\mu$-almost all $x\in W$}.
\end{align}
In particular, almost surely on $\mu_\infty\neq0$,
\[
\dim_H(P_W A)\ge d-\alpha\quad\text{for all }W\in\mathbb{G}_{d,d-k}\,.
\]
\end{cor}

\begin{proof}
Conditions \eqref{H:size-parameter-space}--\eqref{H:codim-random-measure} from Theorem \ref{thm:Holder-continuity} are trivially true or part of the assumptions. Moreover, \eqref{H:Holder-a-priori} follows from Proposition \ref{prop:cutout-measures-satisfy-Holder-a-priori}, and assumption \eqref{H:finite_approx_family} in Theorem \ref{thm:small_dimension_projections} also holds, for any $\theta>0$, as an immediate consequence of \eqref{H:size-parameter-space} and \eqref{H:Holder-a-priori}. Recall that for any set $E\subset\R^d$, supporting a nonzero measure satisfying $\ldimloc(\mu,x)\ge s$ almost everywhere, it holds that $\dim_H(E)\ge s$, see e.g. \cite[Proposition 4.9]{Falconer03}. Thus, the claims follows from Theorem \ref{thm:cor_dim_small}.
\end{proof}

\begin{rems} \label{rem:dim-of-projs}
\begin{enumerate}[(i)]
\item In particular, the above Corollary applies to $\mu_n=\mu_n^{\text{ball}}$. Combined with Theorem \ref{thm:linear-projections}, this implies that a.s. on $\mu_\infty\neq 0$,
\[\dim_H(P_W\mu_\infty)=\dim_H(P_W A)=\min(\dim W,s)\,,\]
simultaneously for all linear subspaces $W\subset\R^d$, where $s=d-\alpha$ is the almost sure dimension of $\mu_\infty$ and $A$ (recall Theorem \ref{thm:dim_cut_out_set}).
\item
The hypothesis \eqref{eq:non-flat-dim} fails when the boundaries of $\Lambda\in\mathcal{F}$ contain $k$-flat pieces. In particular, for $\mu_n^{\text{perc}}$, Corollary \ref{cor:dim_projections} can be applied only to projections whose fibres are not parallel to any coordinate hyperplane.
 However, if the boundaries of $\Lambda\in\mathcal{F}$ contain flat pieces only in a finite set of exceptional directions, it is often possible to verify \eqref{H:finite_approx_family} directly, even if \eqref{H:Holder-a-priori} fails (and this is the reason why we stated Theorem \ref{thm:small_dimension_projections} in this generality). In particular, this is true for $\mu_n^{\text{perc}}$, see \cite[Proposition 3.3]{ShmerkinSuomala12} or \cite[Corollary 3.4]{CChen14}. Hence we recover (and generalize to arbitrary dimensions) the result of Rams and Simon \cite[Theorem 2]{RamsSimon14} on the dimension of projections of fractal percolation sets.
\item The above results are also valid for nonlinear projections for which there is a natural foliation of the fibres as in \eqref{eq:mu_n_bound}. In particular, arguing as in the proof of Theorem \ref{thm:polynomial-projections} for the local existence of such foliations at a.e. point, it follows that a.s. on $\mu_\infty^{\text{snow}(\alpha)}\neq 0$, all polynomial images of $\mu_{\infty}^{\text{snow}(\alpha)}$ are of dimension $\min(1,2-\alpha)$, and likewise for $\mu_\infty^{\text{perc}(\alpha,2)}$, further extending the results of Rams and Simon.
\end{enumerate}
\end{rems}

\section{Upper bounds on dimensions of intersections}

\label{sec:dim_of_intersections}

In the previous sections, we have shown that many random sets and measures satisfy strong quantitative Marstrand-Mattila type theorems for all orthogonal projections.

We now turn to the dual problem of understanding the dimension of intersections with affine planes, and with more general parametrized families of sets. Recall from the discussion in the introduction that for many random sets $A$ in $\R^d$, it is known that  if $V\subset\mathbb{R}^d$ is any Borel set, then $\dim_H(A\cap V)\le\max(0,\dim_H(A)+\dim_H(V)-d)$ almost surely, and $\dim_H(A\cap V)=\max(0,\dim_H(A)+\dim_H(V)-d)$ with positive probability.
This is the case for random similar images of a fixed set $A_0$ of positive Hausdorff measure in its dimension, see \cite[Sect. 10]{Mattila95}, and for the limit sets of fractal percolation \cite{Hawkes81}. In this section,  we show that for fractal percolation, and other random sets supporting SI-martingales, the upper bound $\dim_H(A\cap V_t)\le\max(0,\dim_H(A)+\dim_H(V_t)-d)$ holds a.s. for \emph{all} $t\in\Gamma$ for suitable parametrized families $V_t$. In Section \ref{sec:lower-bound-dim-intersections}, we will deal with the lower bounds.

The projection results in the previous sections are valid for fairly general SI-martingales. In this section, the geometric assumptions on the shapes $\Lambda\in\mathcal{F}$ are the same, but we specialize to the processes in $\subdivision$ or $\poissonian$. The reason is that, in order to estimate box-counting dimension, we need to know that $A_n$ is a good approximation to the neighborhood $A(2^{-n})$,  in a suitable sense.

\subsection{Uniform upper bound for  box dimension}

Suppose that $(\mu_n)\in\mathcal{PCM}$ or $(\mu_n)\in\mathcal{SM}$ with $\mu_n=\Theta( 2^{\alpha n}\mathbf{1}_{A_n})$ and suppose that each $\eta_t$, $t\in\Gamma$ is $s$-Ahlfors regular (that is, $\eta_t(B(x,r))=\Theta(r^s)$ for $x\in\supp\eta_t$). A heuristic calculation suggests that if $M_{n,t}$ is the number of dyadic cubes in $\mathcal{D}_n$ needed to cover $A_n\cap\supp \eta_t$, then $M_{n,t}\lesssim 2^{n(s-\alpha)}Y^{t}_n$. Since Theorem \ref{thm:Holder-continuity} asserts that intersections behave in a regular way, it seems plausible that $\dim_B(\supp\eta_t\cap A)\le s-\alpha$ for all $t\in\Gamma$. Below we verify this under fairly mild additional geometric conditions.

For Poissonian martingales, we will need to approximate the removed shapes from inside in a suitable way. The reason is that $A_n$ may contain many ``thin'' parts so that it is not possible to relate it to neighbourhoods $A(\Omega(2^{-n}))$ in any straightforward way, while if we run the process slightly reducing the size of the removed shapes, then we will have a direct inclusion: see \eqref{eq:A_n^beta} below. This motivates the following definition.  We use the notation
\begin{equation} \label{eq:quantitative-interior}
\Lambda^\kappa:=\{x\in\Lambda\,:\,d(x,\partial\Lambda)\ge\kappa\diam(\Lambda)\}
\end{equation}
for the $\kappa$-quantitative interior of $\Lambda$.

\begin{defn} \label{def:inner-approximation}
Let $\mathbf{Q}_0$ be a measure on the sets in $\mathcal{X}$ of unit diameter. We say that a family of Borel maps $\Lambda\mapsto\Lambda_\varrho$ defined on $\supp\mathbf{Q}_0$ for all $0<\varrho<1$ provides a \textbf{regular inner approximation} if for each $0<\varrho<1$ the following holds for $\mathbf{Q}_0$ almost all $\Lambda$:
\begin{enumerate}[(i)]
\item There is $\beta=\beta(\varrho)>0$ such that $\Lambda_\varrho\subset\Lambda^\beta$.
\item $\leb^d(\Lambda_\varrho)\ge (1-\varrho)\leb^d(\Lambda)$.
\end{enumerate}
Given a regular inner approximation for $\mathbf{Q}_0$, we use the following notation:
Let $\mathcal{Y}=\{s_i(\Lambda_i+t_i)\}$ be a realization of a Poisson point process constructed from $\mathbf{Q}_0$ as in Lemma \ref{lem:structure-scale-translation-invariant-distr}. We define 
\begin{equation}\label{eq:Anrho_def}
A_n^\varrho=\Omega\setminus
\bigcup_{2^{-n}\le s_i<1}s_i((\Lambda_i)_\varrho+t_i)
\end{equation}
and 
\begin{equation}\label{eq:munrho_def}
\mu_{n,\varrho}=2^{n\alpha_\varrho}\mathbf{1}[A^\varrho_{n}]\,,
\end{equation}
where $\alpha_\varrho$ is chosen so that
\[2^{-n \alpha_\varrho}=\PP(x\in A_n^\varrho)\,,\]
for all $x\in\Omega$ and $n\in\N$. 
\end{defn}

Recall from Remark \ref{rem:alpha} that such $\alpha_\varrho$ is well defined. It is easy to see that $(\mu_{n,\varrho})$ defines an SI-martingale of $(\mathcal{F_\varrho},\alpha_\varrho,d)$-cutout type, where $\mathcal{F}_\varrho$ consists of $\Omega$ and the shapes $s(\Lambda_\varrho+t)$. 

Regular inner approximations always exist (just take suitable quantitative interiors), but in order to be useful, the shapes $\Lambda_\varrho$ have to satisfy the same geometric assumptions as the original shapes $\Lambda\in\mathcal{F}$. 

We note that $(\mu_{n,\varrho})$ is essentially the $\mathcal{PCM}$ obtained from the push-forward of $\mathbf{Q}_0$ under $\Lambda\mapsto \Lambda_{\varrho}$ although this is not exactly the case, since the sets $s(\Lambda_\rho+t)$ of diameter $\le 1$ such that $s(\Lambda +t)$ has diameter $>1$ are never removed. Nevertheless, we can still apply our results from the earlier sections to $(\mu_{n,\varrho})$ as long as the shapes $\Lambda_\varrho$ satisfy the required geometric assumptions.

We can now state the uniform upper bounds for affine intersections.
\begin{thm}\label{thm:dim_linear-intersections}
Suppose that in the setting of Theorem \ref{thm:linear-projections},
\begin{enumerate}
\item[\emph{(i)}] Suppose $(\mu_n)\in\mathcal{SM}$ is such that $\alpha<k$, \eqref{eq:non-flat-shapes} holds, and for some constants $C_1,C_2>0$,
\[
C_1\mu_n(x)\le 2^{\alpha n}\mathbf{1}_{A_n}(x)\le C_2\mu_n(x)\,.
\]
Then, there is a finite random variable $K$ such that
\[
\mathcal{H}^k(A(2^{-n})\cap V)\le K\, 2^{-\alpha n}\text{ for all }V\in\mathbb{A}_{d,k}\,.
\]
\item[\emph{(ii)}]
Suppose $(\mu_n)$ is a $\mathcal{PCM}$ is such that $\alpha<k$, and there is a regular inner approximation $(\Lambda_\varrho)$ with the property that for any small $\varrho>0$, the non-flatness condition \eqref{eq:non-flat-shapes} holds for the initial domain $\Omega$ and for the sets $\Lambda_\varrho$ for $\mathbf{Q}_{0}$-almost all $\Lambda\in\mathcal{F}$ (with constants $C,\gamma_0$ possibly depending on $\varrho$).

Then for each $\varepsilon>0$, there is a finite random variable $K$ such that
\[
\mathcal{H}^k(A(2^{-n})\cap V)\le K\, 2^{(\e-\alpha) n}\text{ for all }V\in\mathbb{A}_{d,k}\,.
\]
\end{enumerate}
In particular, in either case a.s. $\overline{\dim}_B(V\cap A)\le k-\alpha$ for all $V\in\mathbb{A}_{d,k}$.
\end{thm}

\begin{proof}
We first show (i). It follows from Theorem \ref{thm:linear-projections} that a.s.
\[
K:=\sup_{n\in\N, V\in\mathbb{A}_{d,k}}Y^V_{n} <\infty\,.
\]
Fix $V\in\mathbb{A}_{d,k}$, $n\in\N$ and suppose that $A(2^{-n})\cap V$ intersects $M$ dyadic cubes of length $2^{-n}$. It follows from \eqref{SD:finite_elements} that for $r=O(2^{-n})$, the set $A_n\cap V(r)$ contains $\Omega(M)$ disjoint elements $F\in\mathcal{F}_n$ whence, using Fubini's Theorem and \eqref{SD:finite_elements} again,
\[
O_K(2^{n(k-d)})\ge \int_{d(y,V)<r}Y^{V_y}_n\,d\mathcal{H}^{d-k}|_{V^\perp}(y)=\mu_n(V(r))\ge \Omega(M) 2^{n(\alpha-d)}\,,
\]
where $V_y$ is the affine $k$-plane parallel to $V$ passing through $y\in V^{\perp}$. This gives the uniform bound $M=O_K(2^{n(k-\alpha)})$ implying the claim (i).

To prove (ii), let $\varrho>0$.  
For any realization $\{s_i(\Lambda_i+t_i)\}$ of $\mathcal{Y}$, we have
\begin{equation}\label{eq:A_n^beta}
\Omega\cap A_n(\beta 2^{-n})\subset A_n^\varrho\,,
\end{equation}
recall that
\begin{align*}
A_n&=\Omega\setminus
\bigcup_{2^{-n}\le s_i<1}s_i(\Lambda_i+t_i)\,,\\
A_n^\varrho&=\Omega\setminus
\bigcup_{2^{-n}\le s_i<1}s_i((\Lambda_i)_\varrho+t_i)\,.
\end{align*}
It follows from our assumptions and Theorem \ref{thm:linear-projections} that a.s. for the SI-martingale $(\mu_{n,\varrho})$, there is $K<\infty$ with $Y_n^{V}\le K$ for all $V\in\mathbb{A}_{d,k}$, $n\in\N$. Note that in order to apply Theorem \ref{thm:linear-projections}, we have to verify \eqref{eq:non-flat-shapes} for the shapes $s\Lambda_\varrho$ for all $s>0$ and $\mathbf{Q}_0$-almost all $\Lambda\in\mathcal{F}$ (the translations are irrelevant since \eqref{eq:non-flat-shapes} is clearly translation invariant). The assumptions imply this for $s=1$ and assuming $\gamma_0<k$ as we may, a simple calculation gives this for all $0<s<1$. 

Fix $V\in\mathbb{A}_{d,k}$. If $x\in A(\beta 2^{-n-1})\cap V$, it follows from \eqref{eq:A_n^beta} that $V\cap B(x,\beta 2^{-n-1})\subset A_{n}^\varrho$.
Then the maximal number $M$ of disjoint balls $B(x, \beta 2^{-n-1})$ with center in $A(\beta 2^{-n-1})\cap V$ satisfies $M 2^{n(\alpha_\varrho-k)}=O_\beta(Y_n^{V})=O_\beta(K)$, and hence
\begin{equation}\label{eq:alpha_e}
M=O\left(2^{n(k-\alpha_\varrho)}\right)\,.
\end{equation}
But $\alpha_\varrho\ge\alpha-\varrho$ using \eqref{eq:alpha}, \eqref{eq:alpharoo} and condition (ii) in the definition of regular inner approximation, so we are done.
\end{proof}

Note that for $\mu_n^{\text{ball}}$ we can simply define $B_\varrho=B^{c(d)\varrho}$ for all balls $B$, but for general $\mathcal{PCM}$ we cannot usually use scaled copies of $\Lambda\in\mathcal{F}$ to construct $\Lambda_\varrho$, since they may fail to fit inside $\Lambda$. Nor can we use $\Lambda^\varrho$ directly, since the non-flatness condition \eqref{eq:non-flat-shapes} might get destroyed. In Lemma \ref{lem:koch_epsilon} below, we construct suitable regular inner approximations for $\mu_n^{\text{snow}}$.

For $d=2$ we have the following variant of Theorem \ref{thm:dim_linear-intersections} for real algebraic curves (recall the notation from Section \ref{sec:algebraic}).

\begin{thm}\label{thm:dim_alg_curves}
In the setting of Theorem \ref{thm:alg_curves},
\begin{enumerate}
\item[\emph{(i)}]  Let $k\in\N$ and $(\mu_n)\in\mathcal{SM}$ such that $\alpha<1$, $\mu_n=\Theta(2^{\alpha n}\mathbf{1}_{A_n})$ and \eqref{K1} and \eqref{K2} hold for $\Lambda\in\mathcal{F}$ with a uniform constant $C$. Then, a.s. there is $K<\infty$, such that
\[\mathcal{H}^1(A(2^{-n})\cap V)\le K\, 2^{-\alpha n}\text{ for all }V\in\mathcal{K}_{k}\,.\]
\item[\emph{(ii)}] Suppose $(\mu_n)$ is a $\mathcal{PCM}$ such that $\alpha<1$, and there is a regular inner approximation $(\Lambda_\varrho)$ with the property that for any small $\varrho>0$, the quantitative unrectifiability conditions \eqref{K1} and \eqref{K2} hold for the initial domain $\Omega$ and for $\Lambda_\varrho$ for $\mathbf{Q}_{\varrho}$-almost all $\Lambda$ with a uniform constant $C$(that may depend on $\varrho$).

Then for each $\varepsilon>0$ and each $k\in\N$, almost surely, there is a random $K<\infty$ such that
\[\mathcal{H}^1(A(2^{-n})\cap V)\le K\, 2^{(\e-\alpha) n}\text{ for all }V\in\mathcal{K}_{k}\,.\]
\end{enumerate}
In particular, in both cases a.s. $\overline{\dim}_B(V\cap A)\le 1-\alpha$ for all $V\in \mathcal{K}$.
\end{thm}

\begin{proof}
The proof is a minor modification of the proof of Theorem \ref{thm:dim_linear-intersections} using Theorem \ref{thm:alg_curves} in place of Theorem \ref{thm:linear-projections}.
As in the proof of Theorem \ref{thm:alg_curves}, let $\mathcal{G}_k$ denote the subcurves of curves in $\mathcal{K}_k$ that are graphs of functions $f(x)=y$ or $x=f(y), J\to\R$, where $J\subset\R$ is some interval and the derivative of $f$ is monotone on $J$ and satisfies $|f'|\le 1$.

To prove (i), denote $Y^{W}_n=\int_{W}\mu_n\,d\mathcal{H}^1$ for $W\in\mathcal{G}_k$. It follows from Theorem \ref{thm:alg_curves} that a.s. there is $K<\infty$ such that
\begin{equation}\label{eq:Q_kbound}
\sup_{W\in\mathcal{G}_k, n\in\N}Y^{W}_n\le K\,.
\end{equation}

Let $W\in\mathcal{G}_k$ be the graph of $x\mapsto f(x)$, where $f$ is defined on an interval $J\subset\R$, as above. Given $\delta>0$, denote $\widetilde{W}_\delta=\{(x,y)\,:\,x\in J,\,|y-f(x)|<2\delta\}$. The coarea formula yields
\begin{align*}
\mu_n(\widetilde{W}_\delta)=\int_{h=-2\delta}^{2\delta}
\int_{(x,y)\in W_h}\phi(x,y)\mu_n(x,y)\,d\mathcal{H}^1(x,y)\,dh=O_K(\delta)\,,
\end{align*}
where $W_h=W+(0,h)$ and $\phi(x,y)=(1+(f'(x))^2)^{-1/2}\le 1$, and the right-most inequality follows from \eqref{eq:Q_kbound}.

But for any $V\in\mathcal{K}_k$, $\delta>0$, the neighbourhood $V(\delta)$ is covered by $O_k(1)$ such $\widetilde{W}_\delta$ with $W\in\mathcal{G}_k$, and this implies the uniform bound $\mu_n(V(2^{-n}))=O_{k,K}(2^{-n})$. Arguing as in the proof of Theorem \ref{thm:dim_linear-intersections}, this shows that $V\cap A(2^{-n})$ can intersect at most $O_{k,K}(2^{n(1-\alpha)})$ of the elements $F\in\mathcal{F}_n$, which is claim (i).

Likewise, in the proof of (ii), we find that a.s. there is $K<\infty$ such that
\[\int_{W}\mu_{n,\varrho}\,d\mathcal{H}^1\le K\,,\]
for all $n\in\N$, $W\in\mathcal{G}_k$. Note that since \eqref{K1}, \eqref{K2} are scale invariant, the assumptions of Theorem \ref{thm:alg_curves} are met for the SI-martingale $(\mu_{n,\varrho})$.
Applying the coarea formula as above, and the definition of $\mu_{n,\varrho}$ (recall \eqref{eq:Anrho_def}, \eqref{eq:munrho_def})
this implies that any $V\cap A(2^{-n})$, with $V\in\mathcal{K}_k$, can intersect at most $O_{k,K}\left(2^{n(1-\alpha+\varrho)}\right)$ disjoint balls of radius $2^{-n}$.
\end{proof}

\begin{rem}
If the $\mathcal{SM}$ martingale in the claim (i) of Theorem \ref{thm:dim_linear-intersections} (resp. Theorem \ref{thm:dim_alg_curves}) satisfies the weaker assumption that $C_1 2^{\alpha n}\mathbf{1}_{A_n}(x)\le\mu_n(x)\le C_2 2^{\beta n}\mathbf{1}_{A_n}(x)$ for some $\beta<k$ ($k=1$ in Theorem \ref{thm:dim_alg_curves}), then the proof still shows that $\dim_B(A\cap V)\le k-\alpha$ for all $V\in\mathbb{A}_{d,k}$ (resp. all $V\in\mathcal{K}$).
\end{rem}

In the next lemma we show that for the von Koch snowflake $\Lambda$ it is possible to construct sets $\Lambda_\varrho$ satisfying the claim (ii) of Theorem \ref{thm:dim_alg_curves}. Thus, $\mu_n^{\text{snow}(\alpha)}$ satisfies the assumptions of Theorem \ref{thm:dim_alg_curves} for $\alpha<1$ (and so does automatically $\mu_n^{\text{snowtile}}$, recall Example \ref{ex:snowflake}).

\begin{lemma}\label{lem:koch_epsilon}
In the setting of Lemma \ref{lem:porous}, let $\Lambda\subset\mathbb{R}^2$ be the Von Koch snowflake domain. There are $C>0$ and $\gamma_1\in (0,1)$ such that the following holds: for each $\beta>0$, there is a domain $\Lambda_\beta$ with
\[
\Lambda^{C\beta}\subset\Lambda_\beta\subset\Lambda^\beta\,,
\]
such that for each $0<\varepsilon<1$, $V\cap\{x\in\Lambda_\beta\mid d(x,\partial\Lambda_\beta)<\varepsilon\}$ may be covered by $C\varepsilon^{-\gamma_1})$ balls of radius $\varepsilon$.
\end{lemma}
\begin{proof}
Recall that the boundary of $\Lambda$ consists of three copies of the snowflake curve, with endpoints forming the vertices of an equilateral triangle of unit side-length. Fix $n\in\N$, and let $\widetilde{\Lambda}_n$ be the level $n$ approximation of $\Lambda$. We will construct several curves $\pi_j$; they depend on $n$, but we do not display this dependence explicitly. Let $\pi_0$ be the boundary of $\widetilde{\Lambda}_n$. Thus, $\pi_0$ consists of $3\times 4^{n}$ line segments of length $3^{-n}$. Now $\pi_1=\{d(x,\pi_0)=3^{-n-2}\}\cap\widetilde{\Lambda}_n$ consists of $3\times 4^{n}$ line segments together with some circular arcs. Let $\pi_2$ be the curve obtained from $\pi_1$ by removing the circular parts and replacing them by continuing the straight parts of $\pi_1$ up to the point where they meet (see Figure \ref{fig:inner_domain}).
Then $\pi_2$ consists of $3\times 4^n$ line segments of length $\Theta(3^{-n})$. Divide each of these into $10$ equally long sub-segments. Suppose $J=[a,b]$ is one of these segments, and let $\Lambda_J$ be a similar copy of the snowflake curve with endpoints $a$ and $b$. Finally, let $\pi_3$ be the union of all these $\Lambda_J$, and let $\Lambda'_n$ be the region bounded by $\pi_3$.

\begin{figure}
   \centering
\resizebox{0.9\textwidth}{!}{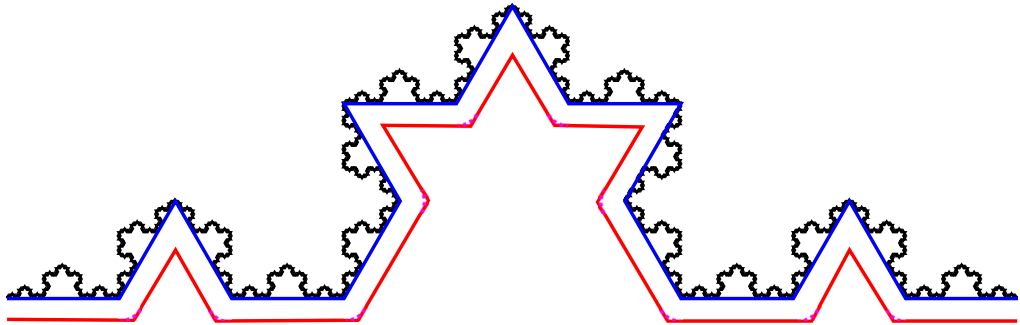}
\caption{Construction of $\pi_0$, $\pi_1$, $\pi_2$. In the last step, the circular arcs in $\pi_1$ are replaced by continuing the straight parts of $\pi_1$ up to the point where they meet.}
\label{fig:inner_domain}
\end{figure}

It follows directly from the construction of $\Lambda$, $\Lambda'_n$ that
\[
\Lambda^{3^{-n-1}}\subset\Lambda'_n\subset\Lambda^{3^{-n-3}}\,.
\]
It remains to show that given $\e>0$, $V\cap\{x\in\Lambda'_n\mid d(x,\partial\Lambda'_n)<\varepsilon\}$ may be covered by $O(\varepsilon^{-\gamma_1})$ balls of radius $\varepsilon$.

To that end, we consider two different cases. If $\varepsilon\ge 3^{-n}$, then
\[
\partial \Lambda'_n(\varepsilon)\subset\partial \Lambda_n(2\varepsilon)\subset\partial\Lambda(3\varepsilon)\,,
\]
and since (applying Lemma \ref{lem:porous} to $\Lambda$), there is $\gamma_1\in (0,1)$ such that $V\cap\partial\Lambda(3\varepsilon)$ may be covered by $O(\varepsilon^{-\gamma_1})$ balls of radius $\varepsilon$, the same therefore holds for $V\cap\partial \Lambda'_n(\varepsilon)$.

If $\varepsilon<3^{-n}$, we first apply the previous observation to cover $V\cap \partial \Lambda'_n(3^{-n})$ by $O(3^{\gamma_1 n})$ balls of radius $3^{-n}$. Let $B$ be one of these balls. Then $3B\cap \partial\Lambda'_n$ is covered by $O(1)$ scaled and translated copies $\Lambda_J$ of the snowflake curve of diameter $\Theta(3^{-n})$. By scaling and Lemma \ref{lem:porous},  each $\Lambda_J(\varepsilon)$, may be covered by $O(3^{-n\gamma_1}\varepsilon^{-\gamma_1})$ balls of radius $\varepsilon$. Adding up, $\partial \Lambda'_n(\varepsilon)$ is covered by $O(3^{n\gamma_1})O((3^n\varepsilon)^{-\gamma_1})=O(\varepsilon^{-\gamma_1})$ balls of radius $\varepsilon$.
\end{proof}

\begin{rems}
\begin{enumerate}[(i)]
\item Lemma \ref{lem:koch_epsilon} appears to be valid for any domain satisfying the assumptions of Lemma \ref{lem:porous}. Indeed, Rohde \cite{Rohde01} shows that each quasicircle $\pi$ is bi-Lipschitz equivalent to a curve $\pi'$ constructed from a process similar to the construction of the standard snowflake, and it appears that the above proof works for each such $\pi'$ as well.
\item Perhaps surprisingly, Theorem \ref{thm:dim_alg_curves}(ii) holds for $\mu_n^{\text{ball}}(\alpha,2)$ as well. Indeed, a simple modification of the proof of Lemma \ref{lem:koch_epsilon} applied to $\pi_\varrho=S(0,1-3\varrho)$ (divide $\pi_\varrho$ into arcs of length $\Theta(\varrho)$ and replace each of these arcs by a similar copy of the snowflake curve) implies that for a fixed $\gamma_1\in(0,1)$, the assumptions in the second part of Theorem \ref{thm:dim_alg_curves} are satisfied for $\Lambda=B(0,1)\subset\R^2$. Note that the constant $C$ in the conclusion of Lemma \ref{lem:koch_epsilon} blows up as $\varrho\to 0$ in this case, but this is allowed in Theorem \ref{thm:dim_alg_curves}.
\end{enumerate}
\end{rems}

\subsection{Upper bounds for intersections with self-similar sets}

We now turn to intersections with fractal objects, and focus on self-similar sets and $\mu_n^{\text{ball}}$ for simplicity.

\begin{thm}\label{thm:dim_intersection_selfsim}
For $\mu_n^{\text{ball}(\alpha,d)}$, almost surely, the random cut-out set $A$ satisfies
\[
\overline{\dim}_B(E\cap A)\le\dim_H E-\alpha
\]
for each self-similar set $E$ with $\dim_H E>\alpha$ satisfying the open set condition.
\end{thm}

\begin{proof}
The proof is again very similar to the proofs of Theorems \ref{thm:dim_linear-intersections} and \ref{thm:dim_alg_curves}, so we will skip some details.

Fix $m\in\N$, $s\in (\alpha,d]$, $C\in (1,\infty)$ and  $\gamma\in (0,s)$, and consider the subset $\Gamma=\Gamma(m,s,\gamma,C)$ of $(\simi_{d}^c)^m\times P_m$ such that:
\begin{enumerate}[(a)]
 \item $F$ satisfies the open set condition,
 \item $t=(F,p)\in(\simi_{d}^c)^m\times P_m$ for which $\eta_t$ is the natural self-similar measure,
 \item All contraction ratios of $F_i$ are between $C^{-1}$ and $1-C^{-1}$,
 \item $E_F\cap B(0,1)\neq\emptyset$,
 \item the Frostman condition $\eta_t(B(x,r))\le C\,r^s$ for all $x\in\R^d$ and $r\in (0,1)$ holds,
 \item The neighbourhood $\partial B(\e)\cap E_F$ can be covered by $C \e^{-\gamma}$ balls of radius $\e$ for all balls $B$ and $\e>0$.
\end{enumerate}

For $0<\varrho<1$, let $A_{n}^\varrho$ and $\mu_{n,\varrho}$ be as  in \eqref{eq:Anrho_def}, \eqref{eq:munrho_def} and note that we may take $\Lambda_i=B(0,\tfrac12)$ and $(\Lambda_i)_\varrho=B(0,\tfrac12-\beta)$, where $\mathcal{L}^d(B(0,\tfrac12-\beta))=(1-\varrho)\mathcal{L}^d(B(0,\tfrac12))$ so that $\alpha_\varrho=\alpha-\varrho$ (recall \eqref{eq:alpharoo}).

Theorem \ref{thm:self_sim} shows that almost surely, for some random $K=K(\varrho)<\infty$,
\begin{equation}\label{eq:not_too_many_balls}
\sup_{n\in\N, t\in\Gamma}2^{(\alpha-\varrho) n}\eta_t(A_{n}^\varrho)\le K\,.
\end{equation}
For $t=(F,p)\in\Gamma$, denote $q=q(F)=\dim_H(E_F)$. Recalling that $A_n(\beta 2^{-n})\subset A_n^{\varrho}$ (see \eqref{eq:A_n^beta}) and that $\eta_t(B(x,r))=\Omega(r^{q})$ for $x\in E_F$, $0<r<1$, it follows that $\eta_t(B(x,2^{-n})\cap A_n^{\varrho})=\Omega(2^{-n q})$ for each $x\in E_F\cap A_n(\beta 2^{-n-1})$. Combining this with \eqref{eq:not_too_many_balls}, we infer that for all  $n\in\N$ there can be at most $O_{F,\varrho}(2^{n(q-\alpha+\varrho)})$ disjoint balls of radius $2^{-n}$ with center in $E_F\cap A(\beta 2^{-n-1})$. Letting $\varrho\downarrow0$ along a subsequence, this implies that a.s. $\dim_B(E_F\cap A)\le \dim E_F-\alpha$ for all $(F,p)\in\Gamma$.

We claim that for every $F\in(\simi_{d}^c)^m$ satisfying the open set condition, there are $C<\infty, \gamma\in (0,\dim_H E_F)$ such that for any ball $B$ of diameter $\le 1$, the annular neighbourhood $\partial B(\e)\cap E_F$ can be covered by $C \e^{-\gamma}$ balls of radius $\e$. If $E_F$ is not contained in a hyperplane, this is shown in Proposition \ref{prop:inter-nbhd-sss}. Otherwise, let $H$ be an affine plane of minimal dimension containing $E_F$, and fix $\e>0$. Let $B$ be a ball of radius $r\le 1$, whose centre is at distance $r-\delta$ from $H$. We consider two cases, depending on a small parameter $\zeta>0$ to be chosen later.
\begin{enumerate}
\item $\delta<\e^\zeta$ (the ``almost tangential'' case). In this case, $\partial B(\e)\cap H$ is contained in a ball in $H$ of radius $O(\e^{\zeta/2})$. Hence, since $E_F$ is Ahlfors-regular, $\partial B(\e)\cap E_F$ can be covered by $O(\e^{-(\dim_H(E_F)(1-\zeta/2))})$ balls of radius $\e$.
\item $\delta\ge\e^\zeta$ (the ``transversal'' case). Here $\partial B(\e)\cap H \subset (\partial B\cap H)(O(\e^{1-\zeta/2}))$, so we can appeal to Proposition \ref{prop:inter-nbhd-sss} applied in $\R^{\dim H}$ to find some $\gamma_1\in (0,\dim_H(E_F))$ such that $\partial B(\e)\cap E_F$ can be covered by $O(\e^{-\gamma_1(1-\zeta/2)})$ balls of radius $\e$.
\end{enumerate}
The claim follows from these cases by taking $\zeta$ sufficiently small in terms of $\gamma_1$ and $\dim_H(E_F)$.

Hence, repeating the argument for a countable dense subset of $\{ (s,\gamma): \alpha<\gamma<s\le d \}$, a sequence $C_k\uparrow\infty$, and all $m\in\N$, yields that a.s, $\overline{\dim}_B(E\cap A)\le\dim E-\alpha$ for all self-similar sets $E$ satisfying the open set condition such that $\dim_H E>\alpha$.
\end{proof}

\begin{rem}
If the ``expected dimension'' ($k-\alpha$ in Theorem \ref{thm:dim_linear-intersections}, $1-\alpha$ in Theorem \ref{thm:dim_alg_curves} and $\dim E_F-\alpha$ in Theorem \ref{thm:dim_intersection_selfsim}) is $\le0$, we can use Theorem \ref{thm:small_dimension_projections} in place of Theorem \ref{thm:Holder-continuity} in the arguments above, to conclude that a.s. $\dim_B(A\cap V_t)=0$ for all such $t\in\Gamma$. For example, the cutout set $A$ for $\mu_n^{\text{ball}(\alpha,d)}$ intersects each open set condition self-similar set of dimension at most $\alpha$ in a set of box-dimension $0$.
\end{rem}

\section{Lower bounds for the dimension of intersections, and dimension conservation}
\label{sec:lower-bound-dim-intersections}

\subsection{Affine and algebraic intersections}
In the previous section, we have obtained upper bounds of the type:
\[
\eta_t(A(2^{-n}))\le K_\e\, 2^{n(s-\alpha+\e)}\quad\text{for all }t\in\Gamma\,,
\]
for some (random) $K_\e<+\infty$. Such bounds imply that almost surely $\dim_B (A\cap\supp\eta_t)\le s-\alpha$ for all $t\in\Gamma$ simultaneously. Since $A$ is a Cantor type set, we cannot in general hope for a lower bound $\dim(A\cap\supp\eta_t)\ge s-\alpha$, for all $t$, but it seems natural to ask whether such dimension lower bound holds for all $t$ such that $Y^t>0$ which, together with Lemma \ref{lem:Yt-survives} and Theorem \ref{thm:Holder-continuity},  would imply that, with positive probability, $\dim(\supp \eta_t\cap A)=s-\alpha$ for a nonempty open set of parameters $t\in\Gamma$. We will see that this is true for various classes of SI-martingales and families $\{\eta_t\}$.

We start by considering intersections with affine planes and algebraic curves. Intersections with self-similar sets can be handled in a similar way, but there are some additional technical complications so we address these later in this section. We first consider a class of Poissonian martingales, and discuss the situation for subdivision martingales afterwards.

\begin{thm}\label{thm:dim_lower_bound}
Let $(\mu_n)$ be a $\mathcal{PCM}$ constructed from the measure $\mathbf{Q}_0$ from an initial domain $\Omega$, with a regular inner approximation $(\Lambda_\varrho)$.
\begin{enumerate}
\item[\emph{(i)}]
Suppose the non-flatness condition \eqref{eq:non-flat-shapes} holds for $\Omega$, and for the sets $\Lambda$, $\Lambda_\varrho$ for $\mathbf{Q}$-almost all $\Lambda$
for all sufficiently small $\varrho>0$ (for some $\gamma_0, C$ that may depend on $\varrho$), and also that $\alpha<k\in \{1,\ldots,d-1\}$.

Write $Y^V=\lim_{n\to\infty} \int_V \mu_n\,d\mathcal{H}^k$. Then a.s.
\[
\dim_H(A\cap V)=\dim_B(A\cap V)=k-\alpha
\]
for all $V\in\mathbb{A}_{d,k}$ such that $Y^V>0$.
\item[\emph{(ii)}]
Suppose conditions \eqref{K1} and \eqref{K2} hold 
for $\Omega$, and for the sets $\Lambda$, $\Lambda_\varrho$ for $\mathbf{Q}$-almost all $\Lambda$
for all $0<\varrho<1$, in each case with uniform constants (that may depend on $\varrho$). Assume also that $\alpha<1$.

For any algebraic curve $V$, write $Y^V=\lim_{n\to\infty} \int_V \mu_n\,d\mathcal{H}^1$. Then a.s.
\[
\dim_H(A\cap V)=\dim_B(A\cap V)=1-\alpha
\]
for all algebraic curves $V$ such that $Y^V> 0$.
\end{enumerate}
\end{thm}

In particular, part (i) holds for $\mu_n^{\text{ball}(\alpha,d)}$ if $\alpha<k$, and for $\mu_n^{\text{snow}(\alpha)}$ if $\alpha<1$ and $k=1$, while claim (ii) holds for $\mu_n^{\text{snow}(\alpha)}$ if $\alpha<1$ (recall Lemma \ref{lem:koch_epsilon}).

Before proceeding to the proof of the theorem, we present a useful consequence.
\begin{cor} \label{cor:dim_lower_bound}
\begin{enumerate}[(i)]
\item In the same setting of Theorem \ref{thm:dim_lower_bound}(i), for each $\beta>0$ there is a positive probability that
\begin{equation} \label{eq:dim-intersection-affine}
\dim_H(A\cap V)=\dim_B(A\cap V)=k-\alpha
\end{equation}
for all $V\in\mathbb{A}_{d,k}$ such that $V\cap\Omega^\beta\neq \varnothing$ (where $\Omega^\beta$ is the quantitative $\beta$-interior of $\Omega$ defined in \eqref{eq:quantitative-interior}).
\item In the same setting of Theorem \ref{thm:dim_lower_bound}(ii), if $\Gamma_0$ is a fixed compact subset of $\mathcal{K}^\Upsilon_m$ (in the $d$ metric of Definition \ref{def:metric-algebraic} with respect to any open ball $\Upsilon$ with $\supp\mu_0\subset\Upsilon$)such that each $V\in\Gamma_0$ hits the interior of $\Omega$, then with positive probability,
\begin{equation} \label{eq:dim-intersection-algebraic}
\dim_H(A\cap V)=\dim_B(A\cap V)=1-\alpha
\end{equation}
for all $V\in\Gamma_0$.
\end{enumerate}
\end{cor}
\begin{proof}
To unify notation, denote $\Gamma_0=\{V\in\mathbb{A}_{d,k}:V\cap\overline{\Omega^\beta}\neq \varnothing\}$ in the first claim. Then $\Gamma_0$ is also compact. Recall that $\Omega=\supp\mu_0$. Since $V$ intersects $\Omega^\circ$ for all $V\in\Gamma_0$ by assumption, it follows from Lemma \ref{lem:Yt-survives} that $\PP(Y^V>0)>0$ for each fixed $V\in\Gamma_0$. As $V\mapsto Y^V$ is a.s. continuous on $\Gamma$ by Theorems \ref{thm:linear-projections} and \ref{thm:alg_curves}, we can find a finite open covering $\{ \Gamma_i\}_{i=1}^M$ of $\Gamma_0$ such that $\PP(E_i)>0$ for all $i$, where $E_i$ is the event that $Y^V>0$  for all $V\in \Gamma_i$. Clearly, the events $E_i$ are decreasing (since removing fewer shapes increases $Y_V$), hence  $\PP(\cap_i E_i)>0$ by the FKG inequality for Poisson point processes (see the proof of Theorem \ref{thm:intersection-compact-sss}). In light of Theorem \ref{thm:dim_lower_bound}, claims \eqref{eq:dim-intersection-affine} and \eqref{eq:dim-intersection-algebraic} hold for all $V\in\Gamma_0$ with positive probability.
\end{proof}

\begin{rem}
 This corollary is optimal in a number of ways. Write $E(\Gamma_0)$ for the event that $Y^V>0$ for all $V\in\Gamma_0$. Clearly, all $V\in\Gamma_0$ must hit $\Omega^\circ$ to have any chance of $\PP(E(\Gamma_0))>0$.  One cannot hope to prove that $E(\Gamma_0)$ has full probability even for a singleton $\Gamma_0$, nor that $E(\Gamma_0)$ has positive probability if $\Gamma_0$ is not compact. In particular, it is easy to see that $E(\Gamma_0)$ has probability zero if $V$ is the set of all affine $k$-planes hitting $\Omega^\circ$.
\end{rem}

We now start the proof of Theorem \ref{thm:dim_lower_bound}. To unify notation, in both claims (i)-(ii) we denote the corresponding measures by $\eta_V$, and the natural parameter space by $\Gamma$. That is, for claim (i), we set $\Gamma=\mathbb{A}_{d,k}$, and for claim (ii) we take $\Gamma=\mathcal{K}^\Upsilon_m$, where $\Upsilon$ is any open ball containing $\supp\mu_0$. Also, for claim (ii) we set $k=1$. Write
\[
Y^V_\varrho = \lim_{n\to\infty} \int \mu_{n,\varrho} \,d\eta_V\,,
\]
where $(\mu_{n,\varrho})$ is obtained from the regular inner approximation as in \eqref{eq:munrho_def}.
For all $0<\varrho<1$, the limit exists and is H\"{o}lder continuous thanks to Theorems \ref{thm:linear-projections} and \ref{thm:alg_curves}, and the assumptions of Theorem \ref{thm:dim_lower_bound}. Recall that $(\mu_{n,\varrho})$ is of $(\mathcal{F}_\varrho,\alpha_\varrho,d)$-cutout type with $\alpha_\varrho\ge\alpha-\varrho$.

Define random variables $X_1, X_1^\varrho, X_2^\varrho$ (for $\varrho>0$) as
\begin{align*}
X_1 &= \sup_{t\in\Gamma} Y^V\,,\\
X_1^\varrho &=\sup_{t\in\Gamma}Y^V_\varrho\,,\\
X_2^\varrho &= \limsup_{n\to\infty}\sup_{t\in\Gamma} \frac{\text{Pack}(V\cap A_n, 2^{-n})}{2^{n(k-\alpha+\varrho)}}\,,
\end{align*}
where $\text{Pack}(E,r)$ is the cardinality of the largest subset $(x_i)$ of $E$ such that $|x_i-x_j|>2r$ for $i\neq j$.

We have already observed that $X_1, X_1^\varrho$ $(\varrho> 0)$ are a.s. finite. Using Corollary \ref{cor:tail_estimate}, we obtain finer information on the tails of $X_1, X_1^\varrho$ and apply them to obtain a tail estimate for $X_2^\varrho$.
\begin{lemma}\label{lem:tails}
Fix $\varrho>0$. There is $\delta>0$ such that for $i=1,2$, and $M>1$ we have
\[
\mathbb{P}\left(X_i^\varrho\ge M\right)=O\left(\exp\left(-M^{\delta}\right)\right)\,,
\]
where the implicit constant is independent of $M$ and $\delta$. The same holds for $X_1$.
\end{lemma}

\begin{proof}
We know from Lemma \ref{lem:Poisson_is_cut_out_type} that $\mathbb{P}(N_0>N)\le 2\exp(-2^{Nd})$, where $N_0$ is the random variable in the definition of cutout type, and therefore also the random variable in \eqref{H:Holder-a-priori} (recall that Propositions \ref{prop:cutout-measures-satisfy-Holder-a-priori} or \ref{prop:cutout-measures-satisfy-Holder-a-priori-2} apply in this setting). The claims for $X_1, X_1^\varrho$ are then immediate from Corollary \ref{cor:tail_estimate}.

The claim concerning $X_2^\varrho$ follows from the claim for $X_1^\varrho$.
Indeed, we have
\[
\eta_V(B(x,2^{-n})\cap A_n^\varrho)=\Omega(2^{-k n}) \quad\text{if } x\in V\cap A_n\,.
\]
Hence, if $(x_j)_{j=1}^M$ is a $(2^{-n})$-packing of $V\cap A_n$, then, using that $\alpha_\varrho\ge \alpha-\varrho$,
\[
2^{-n(\alpha-\varrho)} Y_{n,\varrho}^V \ge \eta_V(A_n^\varrho) \ge M\, \Omega(2^{-kn})\,.
\]
This shows that $X_2^{\varrho}=O(X_1^{\varrho})$, which immediately gives the claim for $X_2^\varrho$.
\end{proof}

\begin{rem}
For $X^{\varrho}_2$, the weaker statement $X^\varrho_{2}<\infty$ a.s. is enough for the proof of Theorem \ref{thm:dim_lower_bound}. The tail estimate was included above because it is a direct consequence of the proofs.
\end{rem}

We will in fact need a version of Lemma \ref{lem:tails} for the restrictions of the process to scaled copies of $\supp\mu_0$.
\begin{lemma}\label{lem:tails_scaled}
Let $B$ be a scaled and translated copy of $\Omega$ of diameter $2^{-n_0}\diam(\Omega)$.
Define
\begin{equation} \label{eq:scaled-intersection}
X_1(B)=\sup_{V\in\Gamma}2^{n_0(k-\alpha)}\limsup_{m\rightarrow\infty}\mu^{V}_{m}(B)\,.
\end{equation}
Then
\[
\mathbb{P}\left(X_1(B)\ge M\right)\le C(\exp(-M^{\delta}))\,,
\]
for some $\delta,C>0$ independent of $B$ and $n_0$.
\end{lemma}
\begin{proof}
We may assume that $\diam(\Omega)=1$.  Let $f$ be the homothety of ratio $2^{-n_0}$ mapping $\Omega$ onto $B$.  By the translation and scale invariance of $\mathbf{Q}$, the process
$\mu_m|_B$, $m\ge n_0$  conditioned on $\Delta_{2^{-n_0}}^1(\mathcal{Y})=\varnothing$ has the same law as $2^{n_0\alpha}\mu_{m-n_0}\circ f^{-1}|_{B\cap\Omega}$, where we think of the $\mu_j$ as functions rather than measures. Hence $X_1(B)$ is stochastically dominated by the random variable
\[
\sup_{V\in\Gamma}
2^{n_0 k}\limsup_{m\to\infty}2^{\alpha(m-n_0)}\eta_V(f(A_{m-n_0}))\,.
\]
Note that
\begin{align*}
\eta_V(f(A_{m-n_0}))&=
\mathcal{H}^{k}(f(A_{m-n_0}\cap f^{-1}(V\cap B)) \\
&= \Theta(2^{-k n_0 }) \mathcal{H}^{k}(A_{m-n_0}\cap f^{-1}(V))\,.
\end{align*}
We deduce that  $\eta_V(f(A_{m-n_0}))=\Theta(2^{-k n_0}) \eta_{W}(A_{m-n_0})$
for some other parameter $W\in\Gamma$, so the claim follows from Lemma \ref{lem:tails}.
\end{proof}

\begin{proof}[Proof of Theorem \ref{thm:dim_lower_bound}]
The upper bounds follow from Theorems \ref{thm:dim_linear-intersections} and \ref{thm:dim_alg_curves}, so it is enough to consider the lower bounds.

Fix $\varrho>0$. For each $n\in\N$, let $\mathcal{B}^n$ be a covering of $\supp\mu_0$ by $2^{-n}$-scaled copies of $\Omega$ with $O(2^{dn})$
elements. For each $B\in \mathcal{B}^n$, let $X_1(B)$ be the random variable defined in \eqref{eq:scaled-intersection}. Then,
\[\mathbb{P}\left(X_1(B)\ge 2^{n\varrho}\right)=O(\exp(-2^{n\delta}))\,,\]
for some $\delta>0$.
Thus, by the Borel-Cantelli lemma, almost surely there is $K'<\infty$ such that for all $n\in\N$ and each $B\in \mathcal{B}^n$ we have $X_1(B)\le K' 2^{n\varrho}$.
Also, $X_2^\varrho\le K''$ for some random a.s. finite $K''$, again by Lemma \ref{lem:tails}. Let $K=\max(K',K'')$.

Fix $V\in\Gamma$, and some large $n\in\N$. Let $x_1,\ldots, x_j$ be a $2^{-n}$ separated set on $V$. Then, by the definition of $X_2^\varrho$, we have $B(x_i,2^{-n})\cap A_n\neq\emptyset$ for at most $O_K(2^{n(k-\alpha+\varrho)})$ indexes $i$. Further, $\mu_m^V(B(x_i,2^{-n}))\le O_K(2^{n(\alpha-k+\varrho)})$ for each such $i$ and large enough $m$.  Hence
\[
\sum_{i=1}^j\mu_m^V(B(x_i,2^{-n}))^2 =O_K\left(2^{n(k-\alpha+\varrho)}\left(2^{n(\alpha-k+\varrho)}\right)^2\right)= O_K\left(2^{n(3\varrho-k+\alpha)}\right)\,.
\]
Letting $m\rightarrow\infty$, we have $\sum_{i=1}^j\mu^V_{\infty}(B(x_i,2^{-n}))^2= O_K(2^{n(3\varrho-k+\alpha)})$ for all $V\in\Gamma$ and $n\in\N$, where $\mu^V_{\infty}$ is any weak accumulation point of the measures $\mu^{V}_m$. This implies that if $\mu^V_{\infty}\neq 0$ (which is the case if $Y^V>0$), then $\dim_2(\mu^V_{\infty})\ge k-\alpha-3\varrho$, where $\dim_2$ is the  correlation (or $L^2$) dimension. As
\[
\dim_H(A\cap V)\ge \dim_H(\mu^V_{\infty})\ge\dim_2(\mu^V_{\infty})
\]
always holds, see e.g. \cite{FLR2002}, the claim follows by letting $\varrho\downarrow 0$ along a sequence.
\end{proof}

\begin{rems}
\begin{enumerate}[(i)]
\item Although this is not needed in the above proof, it is not hard to use Theorem \ref{thm:Holder-continuity}  to show that, in the setting of Theorem \ref{thm:dim_lower_bound}, a.s. there is only one limit measure $\mu^{V}_\infty$ for each $V\in\Gamma$ (i.e. it does not depend on the chosen subsequence),  and $Y^V$ is nothing but the total mass of $\mu^{V}_{\infty}$.
\item A minor modification of the proof shows that almost surely $\tau(\mu^V_{\infty},q)=(q-1)(k-\alpha)$ for all $q\in(0,\infty)$ and all $V\in\Gamma$ such that $Y^V>0$, where $\tau(\nu,q)$ is the $L^q$ spectrum of the measure $\nu$ (see e.g. \cite{FLR2002}). Thus, the intersection measures $\mu^V_{\infty}$ are a.s. monofractal for all $V$ simultaneously.
\end{enumerate}
\end{rems}

We now state an analogue of Theorem \ref{thm:dim_lower_bound}(i) and Corollary \ref{cor:dim_lower_bound}(i) for fractal percolation; afterwards we comment on the situation for more general subdivision type martingales.

\begin{thm} \label{thm:dim_lower_bound_percolation}
Fix $k\in\{1,\ldots,d-1\}$, and let $\mu_n=\mu_n^{\text{perc}(\alpha,d)}$ where $\alpha<k$. As usual, let $Y^V=\lim_{n\to\infty}\int_V \mu_n \,d\mathcal{H}^V$. Writing $A$ for the percolation limit set, almost surely
\begin{equation} \label{eq:dim-intersection-affine-perc}
\dim_H(A\cap V)=\dim_B(A\cap V)=k-\alpha
\end{equation}
for all $V\in\mathbb{A}_{d,k}$ such that $Y^V>0$ and $V$ is not parallel to a coordinate hyperplane.

Moreover, for any compact set $\Gamma_0\subset\mathbb{A}_{d,k}$ of planes that are not parallel to a coordinate hyperplane and intersect the interior of the unit cube, there is a positive probability that \eqref{eq:dim-intersection-affine-perc} holds simultaneously for all $V\in\Gamma_0$.
\end{thm}
\begin{proof}
The proof of the first part has the same structure as the proof of Theorem \ref{thm:dim_lower_bound}, but the details are simpler. Of course, instead of relying on Theorem \ref{thm:linear-projections}, we rely on the modification outlined in Remark \ref{rem:affine-perc}. More precisely, for each $\delta>0$ the claims are proved for planes making an angle $\ge\delta$ with all coordinate hyperplanes, and at the end we let $\delta\downarrow 0$ along a sequence.

The tail bound given in Lemma \ref{lem:tails} for $X_1$ continues to hold with the same proof. Letting $X_2=  2^{(\alpha-k)n} M_n$, where $M_n$ is the number of chosen cubes of generation $n$, it follows from a simple and well-known martingale argument that $\mathbb{E}(X_2)<\infty$. The claim of Lemma \ref{lem:tails_scaled} when $B$ is a dyadic cube follows simply by the self-similarity of the process. With these ingredients, one argues as in the proof of Theorem \ref{thm:dim_lower_bound} to see that, given any $\varrho>0$, a.s. for all $V\in\mathbb{A}_{d,k}$ and sufficiently large $m$,
\[
\sum_i \mu_m^V(Q_i)^2=O\left(2^{n(k-\alpha)}\left(2^{n(\varrho+\alpha-k)}\right)^2\right)= O\left(2^{n(2\varrho-k+\alpha)}\right)\,,
\]
where the sum is over dyadic cubes of generation $n$ and the implicit constants are random (but a.s. finite). Since the correlation dimension can also be defined by summing over dyadic cubes, this implies that $\dim_2(\mu^V_\infty) \ge k-\alpha-2\varrho$ for any accumulation point $\mu^V_\infty$ of $\mu^V_m$, which implies the first claim.

The last claim follows exactly as in the proof of Corollary \ref{cor:dim_lower_bound}.
\end{proof}

We comment on the analogue of this theorem for more general subdivision martingales. In the general $\mathcal{SM}$ setting, in addition to the hypotheses of Theorem \ref{thm:dim_linear-intersections}, it is necessary to impose a homogeneity assumption on the shapes of the filtration. The simplest way to achieve this is by requiring that all $F_n$ be similar images of the seed $F_0$. The role of this assumption is to ensure that the restriction of the process to small domains satisfy the same a priori H\"older condition as the original process, which is crucial in obtaining an analogue of Lemma \ref{lem:tails_scaled}. With this homogeneity assumption, the proof of Theorem \ref{thm:dim_linear-intersections} goes through with minor variations (there is still a conceptual difference: in the general $\mathcal{SM}$ setting, the process can be very far from scale invariant; this requires some additional care to ensure that all constants are uniform in Lemma \ref{lem:tails_scaled}).

The reason we require some form of homogeneity here is that condition \eqref{eq:Holder-shape} (and the underlying geometric conditions, in the setting of Theorem \ref{thm:dim_lower_bound}) is not scale invariant in the appropriate way: much more is required from the shapes $\Lambda\in\mathcal{F}$ of large diameter than the ones of small diameter. Although they seem quite artificial, it is possible to construct $\subdivision$ processes satisfying the transversality conditions of Theorem \ref{thm:dim_linear-intersections} such that the boundaries of the shapes $\Lambda\in\mathcal{F}$ get flatter and flatter as $n\to\infty$, so that our method to obtain the dimension lower bound fails. Note that in the $\poissonian$ case (and for $\mu_n^{\text{perc}}$), this problem does not arise due to the intrinsic scale invariance properties of the process.

\subsection{Dimension conservation}
\label{subsec:dimension-conservation}

There is no Cavalieri principle for Hausdorff dimension: the graph of a continuous function may have any dimension up to the dimension of the ambient space, even though the coordinate fibres are singletons. In \cite{Furstenberg08}, Furstenberg introduced the concept of \emph{dimension conservation}: given a set $A\subset\R^d$ and a Lipschitz map $\pi:A\mapsto\R^k$, we say that $\pi$ is \textbf{dimension conserving} if there exists $\delta\ge 0$ such that
\[
\delta+\dim_H\{ x\in\R^k:\dim_H(\pi^{-1}(x))\ge\delta\} \ge \dim_H(A)\,.
\]
Here $\dim_H(\varnothing)=-\infty$; otherwise dimension conservation would always hold by taking $\delta\ge \dim_H(A)$. When $\pi$ is an orthogonal projection, the fact that $\pi|_A$ is dimension conserving indicates that $A$ satisfies some weak ``Cavalieri principle'' for Hausdorff dimension in the corresponding direction. Furstenberg proved that restrictions of linear maps to a class of homogeneous fractals, including many self-similar sets, are dimension conserving. Different variants of dimension conservation were recently explored by several authors \cite{Hochman10, BFS12, ManningSimon13, FalconerJin14b}. In particular, Falconer and Jin \cite{FalconerJin14b} established the following variant of dimension conservation for the fractal percolation set $A\subset\R^d$: a.s. if $\dim_H A>k$, then for every plane $W\in\mathbb{G}_{d,k}$ and every $\e>0$, the set $W_{\e}$ has positive $\mathcal{H}^{k}$ measure, where
\[
 W_{\e} = \{ w\in W: \dim_H(A\cap W^\perp_w) > \dim_H(A)-k-\e\}
\]
and $W^\perp_w$ is the $(d-k)$-plane orthogonal to $W$ through $w$.

We can see some of our intersection results in the light of the dimension conservation concept.
\begin{defn}
Let $A\subset\R^d$ and let $\pi:A\to\R^k$ be a Lipschitz map. We say that $\pi$ is \textbf{strongly dimension conserving} if there exists a nonempty open set $U\subset\R^k$ such that $\dim_H(A)=\dim_H(\pi^{-1}(x))+k$ for all $x\in U$.
\end{defn}

We make some remarks on this definition. Firstly, if $\pi$ is strongly dimension conserving, then $\pi(A)$ has nonempty interior, so in particular $\dim_H A\ge\dim_H\pi(A)\ge k$. When $\dim_H A\le k$, dimension conservation holds (with $\delta=0$) whenever $\dim_H\pi(A)=\dim_H(A)$ which, for linear maps $\pi$, is the case for many SI-martingales: recall Corollary \ref{cor:dim_projections}. Strong dimension conservation implies dimension conservation with $\delta=\dim_H(A)-k$, but is indeed much stronger since it not only determines the value of $\delta$, but also means that the set of $x$ that witnesses the dimension conservation can be taken to be open.

With this concept in hand, we obtain the following immediate consequence of our earlier results.
\begin{cor}
 \begin{enumerate}[(i)]
  \item If $A$ is the cutout set for a $\mathcal{PCM}$ $(\mu_n)$ satisfying the conditions of Theorem \ref{thm:dim_lower_bound}(i), and $\alpha<d-k$, then a.s. on $\mu_\infty\neq 0$, all the restrictions $\pi|_A$, where $\pi:\R^d\to\R^k$ is linear, are strongly dimension conserving.
  \item If $A$ is the cutout set for a $\mathcal{PCM}$ $(\mu_n)$ satisfying the conditions of Theorem \ref{thm:dim_lower_bound}(ii), and $\alpha<1$, then a.s. on $\mu_\infty\neq 0$, for all polynomial maps $P:\R^2\to \R$, the restrictions $P|_A$ are strongly dimension conserving.
  \item If $A\subset\R^d$ is a fractal percolation limit set for which $\alpha<d-k$, then a.s. on $A\neq 0$, all the restrictions $\pi|_A$, where $\pi$ is an orthogonal projection onto $V\in\mathbb{G}_{d,k}$ such that $V^\perp$ is not parallel to any coordinate hyperplane, are strongly dimension conserving.
 \end{enumerate}

\end{cor}
\begin{proof}
 For the first claim, we recall from Theorem \ref{thm:linear-projections} that a.s. on $\mu_\infty\neq 0$, the orthogonal projections of $\mu_\infty$ onto all $W\in \mathbb{G}_{d,k}$ are nontrivial and their densities are given by $Y^V$ where $V$ ranges over the planes orthogonal to $W$. Hence the claim is immediate from Theorems \ref{thm:linear-projections} and \ref{thm:dim_lower_bound}.

Likewise, the second claim is a consequence of Theorems \ref{thm:polynomial-projections} and \ref{thm:dim_lower_bound}. Recall that in the proof of Theorem \ref{thm:polynomial-projections}, it is shown that a.s. $P\mu_\infty$ is absolutely continuous with a locally piecewise H\"{o}lder continuous density given (in a neighbourhood of a regular value of $P$) by a random variable $\widetilde{Y}^{P,u}=\Theta(Y^{P^{-1}(u)})$. If $\mu_\infty\neq 0$, then $P\mu_\infty\neq 0$, so $\widetilde{Y}^{P,u}>0$ for $u$ in some nonempty open set $U$, and hence $Y^{P^{-1}(u)}>0$ for $u\in U$.  Since Theorem \ref{thm:dim_lower_bound} implies that a.s. $\dim_H(A\cap V)=\dim_H(A)-1$ for all $V$ with $Y^V>0$, the second claim follows.

Finally, the last assertion follows just like the first, using Theorem \ref{thm:dim_lower_bound_percolation} instead of Theorem \ref{thm:dim_lower_bound}.
\end{proof}

\begin{rems}
 \begin{enumerate}[(i)]
  \item The assumptions on $\alpha$ simply mean that $\dim_H A>k$ a.s. on $\mu_\infty\neq 0$ (recall Theorem \ref{thm:dim_cut_out_set}).
  \item For projections onto planes with fibres not parallel to a coordinate hyperplane, we improve the result of Falconer and Jin for fractal percolation discussed above, by eliminating the $\e>0$ and obtaining an open (rather than just positive measure) subset of the range of $\pi$. The question of whether the $\e$ can be replaced by $0$ here was posed in \cite[Section 5]{FalconerJin14b}. At least in dimension $d=2$, it seems very likely that one can obtain the same result for coordinate lines as well, using the dyadic (rather than Euclidean) metric as in \cite{PeresRams14}, but we do not pursue this.
  \item To the best of our knowledge, part (ii) of the Corollary is the first nontrivial positive result for dimension conservation for some nonlinear maps in the case when $\dim_H A>k$. Using Remark \ref{rem:poly-projections-percolation}, the same result holds also for fractal percolation, provided the polynomial is not a function of $x$ or $y$ only.
  \end{enumerate}
\end{rems}

\subsection{Lower bounds on the dimension of intersections with self-similar sets}

To conclude our discussion on the dimension of intersections, we present lower bounds on intersections with self-similar sets. For simplicity, we restrict our attention to $\mu_n^{\text{ball}(\alpha,d)}$ and $\mu_n^{\text{perc}(\alpha,d)}$. We write $T_t(x)=x+t$ and, given an IFS $F\in(\simi_d^c)^m$, we denote $T_t F=(T_t F_i T_t^{-1})_{i=1}^m$. Note that $T_t F\in (\simi_d^c)^m$, and the attractor $E_{T_t F}$ is the translate $T_t E_F$, and likewise for the natural measures.

\begin{thm} \label{thm:dim_lower_bound_self_sim}
Suppose $\mu_n=\mu_n^{\text{ball}(\alpha,d)}$ or $\mu_n^{\text{perc}(\alpha,d)}$. Fix $m$, and write $\Gamma$ for the subset of all $F\in(\simi_d^c)^m$ such that the strong separation holds, $E_F$ is not contained in a hyperplane, and $\dim_H E_F>\alpha$.
\begin{enumerate}[(i)]
\item Almost surely
\begin{equation} \label{eq:dim-intersection-self-sim}
 \dim_H(A\cap E_F)=\dim_B(A\cap E_F)=\dim_H(E_F)-\alpha
\end{equation}
for all $F\in\Gamma$ such that $Y^F=\lim_{n\to\infty} \int\mu_n \,d\eta_F>0$, where $\eta_F$ is the natural measure.
 \item Let $\Gamma_0$ be a compact subset of $\Gamma$ such that $E_F$ hits the open domain $\Omega^\circ$ for each $F\in\Gamma_0$. Then there is a positive probability that \eqref{eq:dim-intersection-self-sim} holds for all $F\in \Gamma_0$ simultaneously.
 \item Almost surely on $\mu_\infty\neq 0$, there exists a nonempty open set $U\subset (\simi_d^c)^m$ such that \eqref{eq:dim-intersection-self-sim} holds
for all $F\in U$. Moreover, $U$ can be chosen to contain a translate $T_t F$ of any fixed
IFS $F\in\Gamma$.
 \end{enumerate}
\end{thm}
\begin{proof}
 Let us start with the claim (i). Many of the steps are similar to previous proofs, and hence we will only indicate the main differences. The upper bounds follow from Theorem \ref{thm:dim_intersection_selfsim} (and its proof). Although this theorem is stated only for $\mu_n^{\text{ball}}$, an inspection of the proof shows that it also holds for $\mu_n^{\text{perc}}$ if we only consider self-similar sets which are not contained in a hyperplane; recall Proposition \ref{prop:inter-nbhd-sss}.

 For the lower bound, notice first that thanks to  Proposition \ref{prop:inter-nbhd-sss} we can decompose $\Gamma$ into countably many sets which satisfy the assumptions of Theorem \ref{thm:self_sim}. Let $\Gamma'$ be one of these sets. The proof of (i) for $F\in\Gamma'$ is almost identical to the proof of Theorems \ref{thm:dim_lower_bound} and \ref{thm:dim_lower_bound_percolation}; the dimension $k$ has to be replaced by the variable quantity $q(F)=\dim_H(E_F)$, but this causes no problem as we assume $q(F)>\alpha$. The only nontrivial modification needed is in verifying a version of Lemma \ref{lem:tails_scaled} for $\mu_n^{\text{ball}}$. To do this, first define
 \begin{equation} \label{eq:def-Gamma-ss}
 \Gamma'' = \left\{ (g F_i g^{-1})_{i=1}^m: F\in\Gamma', g \text{ is a similarity of ratio in } [C^{-1},C] \right\},
\end{equation}
where $C$ is a large constant to be chosen later. Then the self-similar sets $E_G, G\in\Gamma''$ are similar images of $E_F, F\in\Gamma_0$ via a similarity of ratio $\Theta(1)$, and the hypotheses of Theorem \ref{thm:self_sim} still apply to $\Gamma''$.  Let $B$ be a ball of radius $2^{-n_0}$ as in Lemma  \ref{lem:tails_scaled}. We note that for each $F\in\Gamma'$, the set $B\cap E_F$ can be covered by $O(1)$ sets of the form $G_j E_F$, where $G_j$ is a similarity of ratio $\Theta(2^{-n_0})$ (here $G_j$ is a composition of maps in $F$, and  we use the strong separation condition and the fact that the contraction ratios are bounded away from $0$ and $1$). By our definition of the enlarged space $\Gamma''$, the IFS $\{ f^{-1}G_j F_i (f^{-1}G_j)^{-1} \}_{i=1}^m$ is in $\Gamma''$, provided $C$ was taken large enough  (independently of $F\in\Gamma'$ and $B$). As in the proof of Lemma  \ref{lem:tails_scaled}, we conclude that $\eta_F(f(A_{m-n_0}))$ can be bounded by the sum of $O(1)$ terms of the form $\Theta(2^{-n_0 q(F)}) \eta_G(A_{m-n_0})$ for some $G\in\Gamma''$, so the desired claim follows from Lemma \ref{lem:tails} (which in this setting applies to $\Gamma''$ with the same proof).

Since there are countably many $\Gamma'$, this finishes the proof of (i). Claim (ii) follows from (i), the FKG inequality and Lemma \ref{lem:Yt-survives} (it is the same argument from the proof of Corollary \ref{cor:dim_lower_bound}).

For the last claim, fix $F\in\Gamma$. The family $\{ T_t F : T_t E_F\cap \Omega\neq\varnothing\}$ is contained in $\Gamma$, and satisfies the assumptions of Theorem \ref{thm:self_sim}. Hence by Theorem  \ref{thm:self_sim}, a.s. $Y_n^{T_t F}\to Y^{T_t F}$ uniformly (where $Y_n^{T_t F}$ is defined in the obvious way). Hence, using this and Fubini,
\begin{align*}
\int Y^{T_t F} dt &= \lim_{n\to\infty} \int Y_n^{T_t F} \,d t\\
&= \lim_{n\to\infty} \int \int \mu_n(x) \,d(T_t\eta_F)(x) \,dt\\
&= \lim_{n\to\infty} \int \int \mu_n(x+t) \,dt \, d\eta_F(x)\\
&= \lim_{n\to\infty} \int \mu_n(\R^d) \, d\eta_F(x)\\
&= \mu_\infty(\R^d)\,.
\end{align*}
Thus a.s. on $\mu_\infty\neq 0$, there exists $t$ such that $Y^{T_t F}>0$, and claim (iii) follows from Theorem  \ref{thm:self_sim} and the first claim.
\end{proof}

\section{Products and convolutions of spatially independent martingales}
\label{sec:products}

\subsection{Convolutions of random and deterministic measures}
\label{subsec:conv-random-det}

Recall that the convolution of two measures $\mu*\nu$ on $\R^d$ is the push-down of the product measure $\mu\times\nu$ under the addition map $(x,y)\mapsto x+y$. In this section we study convolutions of SI-martingales with other measures: these could be a deterministic measure, an independent realization of the SI-martingale, or the same SI-martingale. We start here with the deterministic case, which is the easiest. Provided the deterministic measure has a sufficiently large Frostman exponent (depending on the smoothness of the boundaries of the shapes in $\mathcal{F}$), these convolutions turn out to have a continuous density. This has applications on the interior of sumsets.

\begin{thm} \label{thm:sum-random-and-deterministic}
Let $(\mu_n)$ be an SI-martingale which either is of $(\mathcal{F},\tau,\zeta)$-cell type and also satisfies $\mu_n(x)\le C 2^{\alpha n}$, or is of $(\mathcal{F},\alpha,\zeta)$-cutout type. Let $\nu$ be a compactly supported measure on $\R^d$ such that $\nu(B(x,r)) \le C\, r^s$, and suppose that for each $\Lambda\in\mathcal{F}$, $\e\in (0,1)$, and each $f\in\iso_d$, the set $\partial\Lambda(\e)\cap f(\supp\nu)$ can be covered by $C\, \e^{-\gamma_1}$ balls of radius $\e$, where $\gamma_1\in (0,1)$. Assume furthermore that $s>\max(\alpha,\gamma_1)$.

Then almost surely the convolution $\mu_\infty * S \nu$ is absolutely continuous for all $S\in \GL_d(\R)$. Moreover, if we denote its density by $f_S$, then the map $(S,x)\mapsto f_S(x)$ is locally H\"{o}lder continuous, with a quantitative deterministic exponent.

In particular, given a Borel set $E\subset\R^d$ with $\dim_H(E)>\max(\gamma_1,\alpha)$, a.s. on $\mu_\infty\neq 0$, the sumset $A+S(E)$ has nonempty interior for all $S\in\GL_d(\R)$.
\end{thm}
\begin{proof}
Express $\GL_d(\R)$ as a countable union of compact sets.
Let $\Xi\subset\GL_d(\R)$ be such a compact set, and let $\Gamma=\Xi\times \Upsilon$, where $\Upsilon$ is a large enough ball that $\Omega+S(\supp\nu)\subset \Upsilon$ for all $S\in\Xi$ (here $\Omega=\supp\mu_0$). Consider the family of
measures $\{ f_{S,z}\nu \}_{(S,z)\in\Gamma}$, where $f_{S,z}(y)=z-Sy$. Thus, identifying $(S,z)$ with $f_{S,z}$ we consider $\Gamma$ as a subset of $\aff_d$. Given $f_1,f_2\in\Gamma$,
with $d(f_1,f_2)=\e$ and $\Lambda\in\mathcal{F}$, we have

\[f_1^{-1}(\Lambda)\Delta f_2^{-1}(\Lambda)\subset
f_1^{-1}\left(\partial\Lambda(O(\e))\right)
\cup f_2^{-1}\left(\partial\Lambda(O(\e))\right)
\]
and thus the set $\supp\nu\cap\left(f_1^{-1}(\Lambda)\Delta f_2^{-1}(\Lambda)\right)$ can be covered by $O(\e^{-\gamma_1})$ balls of radius $\e$. Therefore
\[
\nu(f_1^{-1}(\Lambda)\Delta f_2^{-1}(\Lambda)) \le O(1) \e^{-\gamma_1} \e^s = O(\e^{\gamma_0})\,, \quad\text{where }\gamma_0=s-\gamma_1>0.
\]
Hence we are in the setting of Proposition \ref{prop:cutout-measures-satisfy-Holder-a-priori}, which enables us to apply Theorem \ref{thm:Holder-continuity} and deduce that there is a quantitative $\gamma>0$ such that $Y_n^{(S,z)}:=\int \mu_n \, df_{S,z}\nu$ converge uniformly over all $(S,z)\in\Gamma$ and are uniformly H\"{o}lder with exponent $\gamma$. Then, given $(S,x)\in\Gamma$,
\[
(\mu_n*S\nu)(x) = \int \mu_n(x-y) d(S\nu)(y) =\int \mu_n\, d(f_{S,x}\nu) = Y_n^{(S,x)}\,.
\]
It follows that
\[
(\mu_\infty* S\nu)(B^\circ(x,r)) \le \liminf_{n\to\infty} \int_{B^\circ(x,r)} Y_n^{(S,z)} dz \le K r^d\,,
\]
for some random variable $K<+\infty$. This shows that $\mu_\infty*S\nu$ is absolutely continuous and, since $(S,x)\mapsto Y_n^{(S,x)}$ are uniformly $\gamma$-H\"{o}lder, the density of $\mu_\infty*S\nu$ is $f_S(x)=Y^{(S,x)}$.

The statement about sumsets follows by letting $\nu$ be a Frostman measure on $E$ (see e.g. \cite[Theorem 8.8]{Mattila95}), and using the fact that the convolution $\mu_\infty* S\nu$ is supported on $A+S(E)$.
\end{proof}

\begin{rems}
\begin{enumerate}[(i)]
\item Note that when $d=1$ and the shapes in $\mathcal{F}$ are intervals, we can take $\gamma_1=0$ and the assumption on $s$ reduces to $s>\alpha$. Thus, for example, the standard fractal percolation set $A$ on the line has the property that, given any Borel set $E$ with $\dim_H(A)+\dim_H(E)>1$, the sumset $A+r E$ has nonempty interior for all $r\in\R\setminus{0}$, almost surely conditioned on $A\neq\varnothing$.

\item Let $A=\supp \mu_\infty^{\text{ball}(d,\alpha)}$. Another consequence of the theorem is that if $E$ is either an arbitrary set of Hausdorff dimension $>d-1$, or a self-similar set of dimension $>\alpha$, then a.s. on $A\neq\varnothing$, the sumset $A+S(E)$ has nonempty interior for all $S\in\GL_d(\R)$. If $\dim_H(E)>d-1$ this is immediate from the last claim of the theorem. In the self-similar case, by approximating $E$ in dimension from inside by a self-similar set with strong separation, we may assume that the strong separation holds (recall Lemma \ref{lem:reduction-to-ssc}). But we saw in the proof of Theorem \ref{thm:dim_intersection_selfsim} that, in this case,  $\partial B(\e)\cap E$ can be covered by $O(\e^{-\gamma_1})$ balls of radius $\e$ for some $\gamma_1<s$.

It seems plausible that in fact the conclusion on the nonempty interior $A+S(E)$ for $A=\mu_\infty^{\text{ball}(d,\alpha)}$ holds for arbitrary sets of dimension $>\alpha$, but we have not been able to prove this.
\item The theorem applies, in particular, when $\nu$ is an independent realization of another random measure, provided $\nu$ satisfies the required assumptions. In general, SI-martingales do not satisfy uniform Frostman conditions (although some do), but the special case of the product of two independent SI-martingales will be addressed next in a slightly different way.
\end{enumerate}
\end{rems}

\subsection{A generalization of Theorem \ref{thm:Holder-continuity}}

If $\mu'_n, \mu''_n$ are independent SI-martingales in dimensions $d',d''$, then their product $\mu_n=\mu'_n\times \mu''_n$ satisfies \eqref{SI:bounded}, \eqref{SI:martingale} and \eqref{SI:quotients-bounded}, but it does not satisfy the spatial independence condition \eqref{SI:spatial-independence}: indeed, if $Q_1,Q_2$ are any subsets of $\R^{d'+d''}$ meeting the plane $\{ (x,y)\in\R^{d'+d''}: x=x_0\}$ for some fixed $x_0\in\R^d$, then $\mu_{n+1}(Q_1)$ and $\mu_{n+1}(Q_2)$ are in general \emph{not} independent conditioned on $\mathcal{B}'_n\times \mathcal{B}''_n$ (the filtrations corresponding to $\mu'_n,\mu''_n$). One cannot expect Theorem \ref{thm:Holder-continuity} to hold for these products since, again, planes of the form $x=x_0$  (or $y=y_0$) are clearly exceptional. However, if we work with a family of measures $\{ \eta_t\}_{t\in\Gamma}$ which are transversal to these ``horizontal and vertical'' planes in a suitable sense, then Theorem \ref{thm:Holder-continuity} still holds, with the same proof. The key observation is that spatial independence is only used along the support of the measures $\eta_t$. Even though the main application of the next result is for products of two SI-martingales, we state it in greater generality.

\begin{defn} \label{def:spatially-indep-relative}
Let $(\mu_n)$ be a sequence of absolutely continuous measures on $\R^d$, and let $\{\eta_t\}_{t\in\Gamma}$ be a collection of measures parametrized by a totally bounded metric space $\Gamma$. We say that $(\mu_n)$ is a \textbf{spatially independent martingale relative to $\{\eta_t\}_{t\in\Gamma}$} if $(\mu_n)$ satisfies \eqref{SI:bounded}, \eqref{SI:martingale} and \eqref{SI:quotients-bounded}, and furthermore the following holds:
\begin{enumerate}
\renewcommand{\labelenumi}{(SI\arabic{enumi}')}
\renewcommand{\theenumi}{SI\arabic{enumi}'}
\setcounter{enumi}{3}
\item \label{SI':spatial-indep-relative} There is $C<\infty$ such that for any  $t\in\Gamma$, and any $(C 2^{-n})$-separated family $\mathcal{Q}$ of dyadic squares of length $2^{-(n+1)}$ hitting $\supp\eta_t$, the restrictions $\{\mu_{n+1}|_Q | \mathcal{B}_n\}$ are independent.
\end{enumerate}
\end{defn}

\begin{thm} \label{thm:Holder-continuity-relative-to-measures}
Theorem \ref{thm:Holder-continuity} holds under the weaker assumption that $(\mu_n)$ is an SI-martingale relative to $\{\eta_t\}_{t\in\Gamma}$.
\end{thm}
\begin{proof}
In the proof of Theorem \ref{thm:Holder-continuity}, spatial independence is only used to deduce \eqref{eq:appl-large-deviation}  from  an application of  Lemma \ref{lem:large-deviation}. One only needs to observe that only dyadic squares hitting the support of $\eta_t$ need to be considered in \eqref{eq:appl-large-deviation}, so the independence assumption in Lemma \ref{lem:large-deviation} continues to hold if $(\mu_n)$ is spatially independent relative to $\{\eta_t\}_{t\in\Gamma}$.
\end{proof}

\subsection{Applications to cartesian products of measures and sets}
\label{subsec:products-independent}

In this section we obtain some consequences of Theorem \ref{thm:Holder-continuity-relative-to-measures} when $\mu_n=\mu'_n\times \mu''_n$ and $\mu'_n,\mu''_n$ are independent realizations of an SI-martingale on $\R^d$. If $\mu'_n,\mu''_n$ are of cutout or cell-type, then so is $\mu_n$, so it is not surprising that one can get similar results to those in the previous sections, with the caveat that one needs to avoid measures $\eta_t$ whose support has large intersection with planes of the type $V'_{y_0}=\{ (x,y)\in\R^{2d}:y=y_0\}$ and $V''_{x_0}=\{ (x,y)\in\R^{2d}:x=x_0\}$. We focus on a concrete application: the smoothness of convolutions.

The following result shows that, for large classes of SI-martingales of dimension $>d/2$, the convolution of the limit measure and a linear image of an independent realization of the same process is absolutely continuous, with a density that is jointly H\"{o}lder continuous in $x\in\R^d$ and the linear map.

\begin{thm} \label{thm:convolutions}
Let $(\mu'_n), (\mu''_n)$ be two independent realizations of an SI-martingale, which either is of $(\mathcal{F},\tau,\zeta)$-cell type with $\mu_n(x)\le C 2^{\alpha n}$, or is of $(\mathcal{F},\alpha,\zeta)$-cutout type.  We assume $\alpha<\frac{d}{2}$, and also that there are $\gamma_0, C>0$ such that
\begin{equation} \label{eq:Holder-affine-images}
\leb^d(\partial \Lambda(\e)) \le C\, \e^{\gamma_0} \quad\text{for all }\Lambda\in\mathcal{F}.
\end{equation}

Then almost surely the convolution
\[
\mu'_\infty * S\mu''_\infty
\]
is absolutely continuous for each $S\in \GL_d(\R)$. Moreover, if the density is denoted by $f_S$, then the map $(S,x)\mapsto f_S(x)$ is jointly locally H\"{o}lder continuous with a quantitative deterministic exponent.

In particular, this holds for $\mu_n^{\text{ball}}$, $\mu_n^{\text{snow}}$, $\mu_n^{\text{perc}}$ and $\mu_n^{\text{tile}}$ as in Remark \ref{rem:tiling}, and whenever $\mathcal{F}$ consists of (bounded) affine copies of a fixed set $\Lambda$ with $\overline{\dim}_B(\Lambda)<d$.
\end{thm}

\begin{proof}
We identify the invertible affine map $f(x)=Sx+z$ on $\R^d$ with the pair $(S,z)\in\GL_d(\R)\times \R^d$. Fix a compact subset $\Gamma=\Xi\times \Upsilon$ of $\GL_d(\R)\times \R^d$ such that if $S\in\Xi$, then $\supp\mu_0+S(\supp\mu_0)\subset \Upsilon$.

Given $f=(S,z)\in\Gamma$, let $V_f:=\{ (x,y)\in\R^{2d}: x+S y=z\}$, and define
\[
\eta_f = \phi(S)\, \leb^d|_{V_f}\,,\quad\text{where } \phi(S)=\det(\I_d+(S^{-1})^* S^{-1})^{-1/2}\,.
\]
We will see below that the factor $\phi(S)$ arises from the area formula; note that $\phi$ is a smooth positive function, and in particular it is bounded on $\Xi$. Our first goal is to verify all the hypotheses of Theorem \ref{thm:Holder-continuity-relative-to-measures}.

Note that the ``bad'' planes $V''_{x_0}, V'_{y_0}$ intersect each plane $V_f$ in a single point; moreover, since $\Gamma$ is compact, for $f\in\Gamma$ these intersections are uniformly transversal, in the sense that the intersection of $\e$-neighbourhoods of the planes is contained in a ball of radius $O_\Gamma(\e)$. This implies that $\mu_n=\mu'_n\times\mu''_n$ is an SI-martingale relative to $\{\eta_f\}_{f\in\Gamma}$.

Conditions \eqref{H:size-parameter-space}, \eqref{H:dim-deterministic-measures} and \eqref{H:codim-random-measure} are immediate (with $s=d$ and $2\alpha$ in place of $\alpha$). Note that $d>2\alpha$ by assumption. Therefore we are left to verify the a priori H\"{o}lder estimate \eqref{H:Holder-a-priori}, and we do this with the help of Proposition \ref{prop:cutout-measures-satisfy-Holder-a-priori}. We let $\mathcal{M}=\R^d$ endowed with $d$-dimensional Lebesgue measure $\nu=\leb^d$, and for $f=(S,z)\in\Gamma$ define $\Pi_f(x)=(x,S^{-1}(z-x))$. Then $\Pi_f(\R^d)=V_f$, and $\eta_f=\Pi_f\nu$ thanks to the area formula (see e.g. \cite[Lemma 1 in Section 3.3]{EvansGariepy92}). Indeed, the Jacobian of $x\mapsto (x,-S^{-1}x)$ is easily checked to equal $\phi(S)^{-1}$.

Assume now $\mu'_n$ is of cutout-type; the proof in the cell-type case is similar. Then $\mu_n$ is of $(\mathcal{G},2\alpha,\tau)$-cutout type, where $\mathcal{G}$ consists of $\Omega\times\Omega$ and sets of the form $\Omega\times\Lambda$, $\Lambda\times\Omega$ with $\Lambda\in\mathcal{F}$.

Note that if $\Lambda',\Lambda''\subset\R^d$ are contained in a ball $B(0,O(1))$, then, writing $\Lambda=\Lambda'\times \Lambda''$ we have
$\Pi_{f_1}^{-1}(\Lambda)\Delta \Pi_{f_2}^{-1}(\Lambda) = \Lambda'\cap h_1(\Lambda'')\Delta h_2(\Lambda'')$,
where for $f_i(x)=Sx+z$, $h_i(x)=-Sx+z$. Thus
\begin{align*}
\nu\left(\Pi_{f_1}^{-1}(\Lambda) \,\Delta\, \Pi_{f_2}^{-1}(\Lambda)\right) =  \leb^d\bigl(\Lambda' \cap h_1(\Lambda'')\Delta h_2(\Lambda''))\bigr) \le O_\Gamma(1) \leb^d\bigl(\Lambda''\Delta h_1^{-1}h_2 (\Lambda'')\bigr)\,.
\end{align*}
Let $\e=\dist(f_1,f_2)$.  Since $\dist(h_1^{-1}h_2,\I_d)=O_\Gamma(\e)$, it follows that $\Lambda''\Delta h_1^{-1}h_2 \Lambda''\subset \partial \Lambda''(O_\Gamma(\e))$. Hence \eqref{eq:Holder-shape} follows from the assumption \eqref{eq:Holder-affine-images}.

We have therefore checked the assumptions of Theorem \ref{thm:Holder-continuity-relative-to-measures}. To conclude the desired H\"{o}lder continuity from this, we employ an argument similar to the proof of Theorem \ref{thm:linear-projections}(iii). More precisely, recall that the convolution $\mu'_\infty * S\mu''_\infty$ is the image of the product $\mu'_\infty\times\mu''_\infty$ under the projection $g_S(x,y)= x+Sy$. Hence using the co-area formula (see e.g. \cite[Lemma 1 in Section 3.4]{EvansGariepy92}), and denoting $\mu_n=\mu'_n\times\mu''_n$, we estimate
\begin{align*}
(g_S\mu_\infty)(B^\circ(x,r)) &\le \liminf_{n\to\infty} \mu_n(g_S^{-1}(B^\circ(x,r)))\\
&=   \liminf_{n\to\infty} 2^{2n\alpha}\leb^{2d}(g_S^{-1}(B^\circ(x,r))\cap (A'_n\times A''_n)) \\
&=   \psi(S)^{-1} \liminf_{n\to\infty} 2^{2n\alpha} \int_{B^\circ(x,r)} \leb^d( g_S^{-1}(z)\cap  (A'_n\times A''_n)) dz\\
&=  \psi(S)^{-1}\phi(S)^{-1}  \liminf_{n\to\infty} \int_{B^\circ(x,r)} \int \mu_n \, d\eta_{S,z} \, dz\\
&=  \psi(S)^{-1}\phi(S)^{-1}  \liminf_{n\to\infty} \int_{B^\circ(x,r)} Y_n^{S,z} \, dz\\
&\le \psi(S)^{-1}\phi(S)^{-1} K\, r^d\,
\end{align*}
where $\psi(S)=\sqrt{\det(\I_d+S S^*)}$ is the $d$-Jacobian of the map $g_S$, and
\[
K=\sup_{f\in\Gamma}\sup_{n\in\N} Y_n^f
\]
is finite thanks to Theorem \ref{thm:Holder-continuity-relative-to-measures}. This shows that $\mu'_\infty* S\mu''_\infty$ is absolutely continuous. Since $Y_n^{S,z}$ is uniformly H\"{o}lder continuous in $(S,z)$, it also follows that the density $f_S(z)$ is given by $\psi(S)^{-1}\phi(S)^{-1} Y^{S,z}$. As $\psi,\phi$ are smooth and bounded away from $0$, this concludes the proof.
\end{proof}

As an immediate corollary, we get (recall from Remark \ref{rem:survives} that $\mu'_\infty,\mu''_\infty\neq0$ with positive probability):

\begin{cor} \label{cor:sumset}
Under the assumptions of Theorem \ref{thm:convolutions}, let $A'=\supp\mu'_\infty, A''=\supp\mu''_\infty$. Then there is a positive probability that $A'+S(A'')$ has nonempty interior for all $S\in\GL_d(\R)$.
\end{cor}

\begin{rems}
\begin{enumerate}[(i)]
\item Since the corollary applies, in particular, to fractal percolation of dimension $>d/2$ in $\R^d$, this recovers and substantially generalizes a result of Dekking and Simon \cite[Corollary 1]{DekkingSimon08}, who proved that the difference set of two independent realizations of fractal percolation  of dimension $>1/2$ on the line has nonempty interior a.s.
\item It is not important in the proof that the measures $\mu'_n$ and $\mu''_n$ are realizations of the same SI-martingale. They could be independent realizations of two different SI-martingales satisfying \eqref{eq:Holder-affine-images} so that $\mu_n'=O(2^{\alpha'n}), \mu_n''=O(2^{\alpha''n})$ with $\alpha'+\alpha''<d$.
\end{enumerate}
\end{rems}

\subsection{Products and the breakdown of spatial independence}

While Definition \ref{def:spatially-indep-relative} provides a good substitute of spatial independence, as illustrated by Theorem \ref{thm:convolutions}, its applicability is limited. For example, if $(\mu'_n,\mu''_n)$ are independent realizations of an SI-martingale on $\R^d$, then the products $\mu'_n\times\mu''_n$ are not spatially independent relative to families of measures supported on algebraic varieties of dimension $>d$ (outside of degenerate situations involving massive parallelism). Such intersections arise naturally in many problems. For example, to study the distance set $D(A',A'')= \{ |x'-x''|:x'\in A', x''\in A''\}$, one needs to intersect the products $\mu'_n\times\mu''_n$ with hypersurfaces $\{ (x,y)\in\R^{2d}: |x-y|^2=t\}$. The situation becomes even worse if one considers products of $\ell>2$ independent realizations of an SI-martingale. For example, to study convolutions
\[
\mu^{(1)}_\infty * S_2 \mu^{(2)}_\infty * \cdots *  S_\ell \mu^{(\ell)}_\infty
\]
one needs to consider intersections with $d(\ell-1)$-planes in $\R^{\ell d}$; on the other hand, there are dependencies along $d(\ell-1)$-planes as well, so there is no relative spatial independence. Similar issues arise when studying many kinds of patterns inside random fractals, such as finite configurations or angles.

It turns out that it is possible to overcome this obstacle by allowing a certain degree of dependency. The key to this is the fact that slightly weaker versions of Hoeffding's inequality continue to hold if one allows some controlled dependency between the random variables. Both the precise statement of the appropriate condition that replaces spatial independence, and its  verification in concrete examples, take a substantial amount of work, and hence we defer the details to our forthcoming article \cite{ShmerkinSuomala14}.

\subsection{Self-products of SI-martingales}
\label{subsec:self-convolutions}

In the last section we studied products of independent realizations of SI-martingales. What about self-products $\nu_n:=\mu_n\times\mu_n$ where $(\mu_n)$ is a given SI-martingale? In this case spatial independence fails, just as for products of independent realizations, but now the martingale property \eqref{SI:martingale} also fails. Indeed, $\mu_{n+1}(x)$ determines $\nu_{n+1}(x,x)$, so \eqref{SI:martingale} always fails for points on the diagonal, and therefore also near the diagonal. Fortunately, this is not a serious issue if the supports of the measures $\eta_t$ are uniformly transversal to the diagonal. Even in this case, there are still some new dependencies: if $Q_i$ are dyadic cubes of side length $2^{-k}$ in $\R^d$ then $(\nu_{n+1}|_{Q_1\times Q_2})|\mathcal{B}_n$ and $(\nu_{n+1}|_{Q_2\times Q_3})|\mathcal{B}_n$ are \emph{not} independent. Again, it turns out that this issue is not serious when the supports of $\eta_t$ are transversal to the diagonal.

Rather than formulating a more general version of Theorem \ref{thm:Holder-continuity} valid for a weaker notion of SI-martingale, we will establish the required modification for the setting that interests us here, deferring a systematic study of martingales with weak spatial dependency to \cite{ShmerkinSuomala14}.

\begin{thm} \label{thm:self-convolutions}
Let $(\mu_n)$ be an SI-martingale which is either of $(\mathcal{F},\tau,\zeta)$-cell type with $\mu_n(x)\le C 2^{\alpha n}$, or of $(\mathcal{F},\alpha,\zeta)$-cutout type. We assume $\alpha<\frac{d}{2}$, and also that there are $\gamma_0, C>0$ such that \eqref{eq:Holder-affine-images} holds.

Write $\mathcal{O}= \{ S\in\GL_d(\R)\,: \, S+\I_d\text{ is not invertible}\}$.
Then almost surely the convolution
\[
\mu_\infty * S\mu_\infty
\]
is absolutely continuous for each $S\in \GL_d(\R)\setminus \mathcal{O}$. Moreover, if the density is denoted by $f_S$, then the map $(S,x)\mapsto f_S(x)$ is jointly locally H\"{o}lder continuous with a quantitative deterministic exponent.

In particular, this holds for $\mu_n^{\text{ball}}$, $\mu_n^{\text{snow}}$ and $\mu_n^{\text{perc}}$, and whenever $\mathcal{F}$ consists of affine copies of a fixed set $\Lambda$ with $\overline{\dim}_B(\Lambda)<d$.
\end{thm}

Let $\nu_n=\mu_n\times\mu_n$, and fix a compact set
\[
\Gamma=\Xi\times \Upsilon\subset (\GL_d(\R)\setminus\mathcal{O})\times\R^d\,.
\]
Given $(S,z)\in\Gamma$, let $V_{S,z}=\{ (x,y):x+Sy=z\}$ and $\eta_{S,z}=\leb^d|_{V_{S,z}}$. As discussed above, $\nu_n$ is not an SI-martingale relative to $\{\eta_{S,z}\}_{(S,z)\in\Gamma}$ but nevertheless the claim of Theorem \ref{thm:Holder-continuity} holds. To see this, we first establish the following analogue of Lemma \ref{lem:large-deviation}. We denote the diagonal $\{(x,x)\,:\,x\in\R^d\}\subset\R^d\times\R^d$ by $V_\Delta$.

\begin{lemma} \label{lem:large-deviation-self-conv}
There is $C'>0$ such that for all $t=(S,z)\in\Gamma$, $n\in\N$ and $\kappa^2\ge 2^{(2\alpha-d)n}$,
\[
\PP\left(|Y_{n+1}^t - Y_n^t| > \kappa(\sqrt{Y_n^t}+C'\kappa)\right)= O\left(\exp(-\Omega(\kappa^2 2^{(d-2\alpha)n} ))\right),
\]
where, as usual, $Y_n^t = \int \nu_n \,d\eta_t$.
\end{lemma}
\begin{proof}
Note that $\{\eta_t\}_{t\in\Gamma}$ has Frostman exponent $d$, and $\nu_n(x)\le 2^{2\alpha n}$ so the numerology is that of Lemma \ref{lem:large-deviation}. The martingale assumption \eqref{SI:martingale} is only used in the proof of Lemma \ref{lem:large-deviation} to ensure that $\EE(X_Q)=0$, where
\[
X_Q =\int_Q \mu_{n+1} \,d\eta_t - \int_Q \mu_n\,d\eta_t\,.
\]
Observe that $V_t$ intersects the neighbourhood of the diagonal $V_\Delta(\sqrt{d} 2^{-n})$ in a set of diameter $O_\Gamma(2^{-n})$ (this is the point where we use that $-1$ is not an eigenvalue of $S$). We split the family of dyadic cubes hitting $\supp\eta_t$ into two families, the ones which hit the diagonal, and the rest.  Let us denote these families by $\mathcal{Q}'_n,\mathcal{Q}''_n$ respectively. Then $|\mathcal{Q}'_n|=O(1)$, and $\EE(X_Q)=0$ for $Q\in\mathcal{Q}''_n$.

We will further split $\mathcal{Q}''_n$ into $O(1)$ subfamilies $\mathcal{Q}''_{n,j}$ as follows. Define a graph on the vertices $\mathcal{Q}''_n$ by drawing an edge between $Q_1\times Q_2$ and $Q_3\times Q_4$ (where $Q_i$ are dyadic cubes in $\R^d$) if $\{ Q_1,Q_2\}\cap\{ Q_3,Q_4\}\neq\varnothing$. Since the planes $V_t$ are uniformly transversal, the vertices in this graph have degree uniformly bounded by $M=O_\Gamma(1)$. We can then split $\mathcal{Q}''_n$ into $M$ families $\mathcal{Q}''_{n,j}$ such that there is no edge joining distinct elements of $\mathcal{Q}''_{n,j}$ for any $j$: let $\mathcal{Q}''_{n,1}$ be a maximal subset of $\mathcal{Q}''_n$ with no edge between different elements, $\mathcal{Q}''_{n,2}$ a maximal subset of $\mathcal{Q}''_n\setminus\mathcal{Q}''_{n,1}$ with no edge between different elements, and so on. By construction, the random variables $\{ X_Q: Q\in\mathcal{Q}''_{n,j}\}$ are independent for each $j$.

It follows that
\[
Y_{n+1}^t - Y_n^t = \sum_{Q\in\mathcal{Q}'_n} X_Q +  \sum_{j=1}^M \sum_{Q\in\mathcal{Q}''_{n,j}} X_Q =: Z' + \sum_{j=1}^M Z''_j\,.
\]
For each $Z''_j$ we can apply exactly the same argument as in the proof of Lemma \ref{lem:large-deviation}, to conclude that
\[
\PP(|Z''_j|>\kappa \sqrt{Y_n^t})= O\left(\exp\left(-\Omega(\kappa^2 2^{(d-2\alpha)n})\right)\right)\,.
\]
Hence the same holds for $Z''=\sum_{j=1}^M Z''_j$ by absorbing $M$ to the $O(1)$ constant.

For $Z'$, we have the deterministic bound $|Z'|= O(1) 2^{(2\alpha-d)n}= O(\kappa^2)$, by the assumption on $\kappa$. Combining the bounds for $Z'$ and $Z''$ yields the lemma.
\end{proof}

We can deduce that the conclusion of Theorem \ref{thm:Holder-continuity} remains valid in this setting.
\begin{prop} \label{prop:Holder-self-convolutions}
Almost surely, $Y_n^t$ converges uniformly over $t\in\Gamma$, and the limit $Y^t$ is H\"{o}lder continuous with a deterministic quantitative exponent.
\end{prop}
\begin{proof}
Firstly, we note that $2\alpha<d$ by assumption, and hypotheses \eqref{H:size-parameter-space}--\eqref{H:Holder-a-priori} of Theorem \ref{thm:Holder-continuity} are met. Only \eqref{H:Holder-a-priori} is nontrivial, but this follows exactly as in the proof of Theorem \ref{thm:convolutions} (the fact that we have self-products instead of products of two independent realizations does not change the geometry underlying \eqref{H:Holder-a-priori}). Thus, we need to check that the proof of the theorem goes through in the current setting.

We see  that in the proof of Theorem \ref{thm:Holder-continuity}, Lemma \ref{lem:large-deviation} is applied with $\kappa=2^{-\lambda n}\ll 1$, thus, using that $\sqrt{x}+1\le 2\max(x,1)$, Lemma \ref{lem:large-deviation-self-conv} implies that, instead of \eqref{eq:appl-large-deviation}, we get
\[
\PP(|Y_{n+1}^t - Y_n^t|> 2^{-\lambda n} 2\max(Y_n^t,1)) \le \exp\left(-\Omega(2^{(d-2\alpha-2\lambda)n})\right).
\]
Continuing as in the proof of the theorem, this shows that \eqref{eq:bound_Z_n} holds with a new factor of $2$ in the right-hand side, however this factor is immaterial and the rest of the proof proceeds in the same way.
\end{proof}

\begin{proof}[Proof of Theorem \ref{thm:self-convolutions}]
The proof is identical to the proof of Theorem \ref{thm:convolutions}, except that we rely on Proposition \ref{prop:Holder-self-convolutions} instead of Theorem \ref{thm:Holder-continuity-relative-to-measures}.
\end{proof}

\begin{rem}
The exclusion $S\notin\mathcal{O}$ is in general necessary: note that e.g.
\[
\mu*(-\I_d\mu)(B(0,\e)) = \int \mu(B(x,\e))d\mu(x)\,,
\]
and therefore
\[
\dim_2(\mu)=\lim_{\e\downarrow 0} \frac{\log \mu*(-\I_d\mu)(B(0,\e))}{\log\e}\,,
\]
where $\dim_2$ is correlation dimension (see e.g. \cite[Section 4]{PeresSolomyak00} for this way of defining $\dim_2$). Since $\dim_2\mu\le \dim_H\supp\mu$ for any measure, it follows that the density $\e^{-d}\mu*(-\I_d\mu)(B(0,\e))$ grows exponentially as $\e\downarrow 0$ whenever our random measures are supported on a set of dimension $<d$ (which is the case for all our main examples when $\alpha>0$).

More generally, $\mu_\infty*S\mu_\infty$ can have singularities on the set $E_S=(S+\I_d)(\R^d)$ for any $S\in\mathcal{O}$. However, it still holds that the measure $\mu_\infty*S\mu_\infty$ is absolutely continuous, and the set $A+S(A)$ has nonempty interior, for \emph{all} $S\in\GL_d(\R)$.
\end{rem}

\begin{cor}
Under the assumptions of Theorem \ref{thm:self-convolutions}, a.s. on $\mu_\infty\neq 0$,  $\mu_\infty*S\mu_\infty$ is absolutely continuous, and the set $A+S(A)$ has nonempty interior, for \emph{all} $S\in\GL_d(\R)$.
\end{cor}

\begin{proof}
Let $\Gamma=\Xi\times \Upsilon\subset\GL_d(\R)\times\R^d$ be a compact set as in the proof of Theorem \ref{thm:convolutions}.
For each $n\in\N$, let $D_n$ be a finite union of sets $\Lambda'\times\Lambda''$, where each $\Lambda',\Lambda''$ is a similar image of some $\Lambda\in\mathcal{F}$, such that
\[
(\Omega\times\Omega)\cap V_\Delta\left((C+\sqrt{d}) 2^{-n}\right)\subset D_{n}\subset V_\Delta\left(O(2^{-n})\right)\,.
\]
Note that to construct such $D_n$ it is enough to find one $\Lambda\in\mathcal{F}$ with nonempty interior, and if there were no such $\Lambda$, assumption \eqref{eq:Holder-affine-images} would imply that $\mu_\infty=0$ a.s.

Fix $n_0\in\N$, and consider the sequence $\nu_{n,n_0}=(\mu_n\times\mu_n)|_{\R^{2d}\setminus D_{n_0}}$, $n\ge n_0$, conditioned on $\mathcal{B}_{n_0}$. This is an SI-martingale (because we are cutting out a sufficiently large neighbourhood of the diagonal); denote the limit by $\nu_{\infty,n_0}$. Although these sequences are not spatially independent (even relative to $\eta_t$), just as above the dependencies are bounded, so the proofs of Lemma \ref{lem:large-deviation-self-conv}, Proposition \ref{prop:Holder-self-convolutions} and Theorem \ref{thm:self-convolutions} carry over, except that we do not need to exclude maps $S$ such that $S+\I_d$ is not invertible. Condition \eqref{eq:Holder-shape} is checked exactly as in the proof of Theorem \ref{thm:convolutions} (note that \eqref{eq:Holder-affine-images} is valid for the sets $\Lambda',\Lambda''$ used to define $D_{n_0}$, where the constant $C$ is allowed to depend on $n_0$). Hence a.s. the images of $\nu_{\infty,n_0}$ under the projections $g_S(x,y)=x+Sy$ are H\"{o}lder continuous for all $S\in\GL_d(\R)$.

Further, since $\mu_\infty$ has no atoms a.s., $(\mu_\infty\times\mu_\infty)(V_\Delta)=0$ and it follows that a.s. on $\mu_\infty\neq 0$ there is $n_0$ such that $\nu_{\infty,n_0}\neq 0$. Conditioning on such $n_0$ then implies that a.s $A+SA$ has nonempty interior for all
$S\in\GL_d(\R)$.

The absolute continuity of the convolutions $\mu_\infty*S\mu_\infty$ is obtained by writing $\mu_\infty*S\mu_\infty=\sum_{n=1}^\infty g_S((\mu_\infty\times\mu_\infty)|_{U_n})$, where $U_1=\R^{2d}\setminus D_1$ and $U_n=D_{n-1}\setminus D_{n}$ for $n> 1$.
\end{proof}

\begin{rem}
As mentioned in the introduction, K{\"o}rner \cite{Korner08} constructed random measures $\mu$ on the real line supported on a set of any dimension $s\in (1,2)$, such that the self-convolution $\mu*\mu$ is absolutely continuous, and moreover its density is H\"{o}lder with the optimal exponent $s-1/2$. Other than for the value of the H\"{o}lder exponent, Theorem \ref{thm:self-convolutions} shows that this holds for a rich class of random measures, including fractal percolation and related models. Moreover, the existence of measures $\mu$ on $\R^d$ supported on sets of dimension $d/2+\e$ such that $\mu * S\mu$ is absolutely continuous (even with a H\"{o}lder density) for all $S\in\GL_d(\R)\setminus\mathcal{O}$ is a new result.
\end{rem}

\section{Applications to Fourier decay and restriction}
\label{sec:Fourier}

\subsection{Fourier decay of SI-martingales}
\label{subsec:Salem}

Recall that for a measure $\mu\in\mathcal{P}_d$, its Fourier dimension is $\dim_F\mu=\sup\{ 0\le\sigma\le d: \widehat{\mu}(\xi)=O_\sigma (|\xi|^{-\sigma/2})\}$, and we say that $\mu$ is a \textbf{Salem measure} if $\dim_F\mu=\dim_H\mu$ (it is well known that $\dim_F\mu\le\dim_H\mu$). There are relatively few known classes of Salem measures, and among those many are ad-hoc random constructions, see e.g. \cite{LabaPramanik09} and references there. We next show that for some classes of SI-martingales the limits are indeed Salem measures. In particular, this is the case for $\mu_n^{\text{perc}(\alpha,d)}$ provided $\alpha\ge d-2$ (and this range is sharp), which appears to be a new result.

\begin{thm} \label{thm:Salem}
Let $(\mu_n)$ be a subdivision martingale (recall Definition \ref{def:subdivision}), defined using the dyadic filtration $\mathcal{Q}_{n}$ of the unit cube $[0,1]^d$, and suppose for each $n$ there is $p_n\in (0,1]$ such that $W_Q\in \{0,p_n^{-1}\}$ for all $Q\in\mathcal{Q}_n$. Write $\beta_n=(p_1\cdots p_n)^{-1}$, and note that in this case $\mu_n = \beta_n \mathbf{1}_{A_n}$ for some random set $A_n\subset [0,1]^d$.

Suppose $\frac1n \log_2\beta_n\to\alpha\in [d-2,d]$. Then a.s. $\mu=\mu_\infty$ is a Salem measure and, moreover, $\dim_H\mu=\dim_H\supp\mu=\dim_B\supp\mu=d-\alpha$ a.s. on $\mu\neq 0$.
\end{thm}

The proof is an adaptation  of the construction in \cite[Section 6]{LabaPramanik09}, which in essence deals with a special case, but is also closely related to the proof of Theorem \ref{thm:Holder-continuity}, and in particular relies on an argument similar to Lemma \ref{lem:large-deviation}, further emphasizing the unity of the SI-martingale approach.
\begin{proof}
The dimension upper bounds are standard, but we sketch the argument for the reader's convenience. The assumptions imply that $A_n$ is a union of a random number $M_n$ of cubes in $\mathcal{Q}_n$. Each cube $Q$ of $A_n$ gives rise to a random number $X_{Q,n}$ of offspring, where $\EE(X_{Q,n})=2^d p_{n+1}$. It follows that $2^{-dn}\beta_n M_n$ is a nonnegative martingale, and hence converges a.s. to a finite limit $K$. This shows that (for large $n$) $A_n$ can be covered by $2K 2^{dn}\beta_n^{-1}$ cubes in $\mathcal{Q}_n$, and hence $\overline{\dim}_B(\supp\mu) \le d-\alpha$.

For the rest of the proof, it suffices to show that $\dim_F\mu \ge d-\alpha$ or, in other words, that given $\sigma<d-\alpha$, we have the estimate $\widehat{\mu}(k) = O_\sigma(\|k\|^{-\sigma/2})$ for $k\in\Z^d$ (for the fact that it suffices to establish decay at integer frequencies, see \cite[Lemma 9A4]{Wolff03}). It is convenient to use the $L^\infty$ norm $\|k\|=\max(|k_1|,\ldots,|k_d|)$.

We write $e(x)=\exp(-2\pi i x)$ and define a (complex) measure $\eta_k = e(k\cdot x) dx$. Now
\[
\widehat{\mu}_{n+1}(k)-\widehat{\mu}_n(k)=\sum_{Q\in\Q_n}X_Q\,,
\]
where $X_Q=\int_Q\mu_{n+1}\,d\eta_k-\int_Q\mu_n\,d\eta_k$. Note that, conditional on $\mathcal{B}_n$, $A_n$ is a union of $M_n$ cubes $Q\in\mathcal{Q}_n$, and $|X_Q|=O(\beta_n 2^{-nd})$. Using Hoeffding's inequality \eqref{eq:Hoeffding-Janson} for the real and imaginary parts of $\eta_k$, we obtain
\begin{align*}
\PP\left(|\widehat{\mu}_{n+1}(k)-\widehat{\mu}_n(k)| > 2^{-\sigma n/2}\|\mu_n\|^{1/2}\,|\,\mathcal{B}_n\right) &=
O\left(\exp\left(-\Omega(2^{(2d-\sigma)n}\|\mu_n\|\beta_n^{-2}M_n^{-1})\right)\right)\\
&=O\left(\exp(-\Omega(2^{(d-\sigma)n}\beta_n^{-1}))\right)\,.
\end{align*}
Since $\log_2\beta_n\to\alpha$ and $\sigma<d-\alpha$, we get
\begin{equation} \label{eq:Salem-larve-dev}
\PP\left(|\widehat{\mu}_{n+1}(k)-\widehat{\mu}_n(k)| > 2^{-\sigma n/2}\|\mu_n\|^{1/2} \text{ for some } \|k\|< 2^{n+1}\right) \le \e_n\,,
\end{equation}
where $\e_n=O(2^n)\exp(-\Omega(2^{(d-\sigma)n}\beta_n^{-1}))$ is summable.

Let $Q\in\mathcal{Q}_{n+1}$. Since all coordinates of all vertices of $Q$ are of the form $m 2^{-n-1}$, it follows that for $\|k\|<2^{n+1},0\neq\ell\in\mathbb{Z}^d$,
\[
\widehat{\mathbf{1}}_Q(k+2^{n+1}\ell) = \prod_{j:k_j+2^{n+1}\ell_j\neq 0}\frac{k_j}{k_j+ 2^{n+1}\ell_j} \widehat{\mathbf{1}}_Q(k)\,.
\]
Since $\widehat{\mu}_{n+1}$ is a linear combination of the functions $\widehat{\mathbf{1}}_Q$, the same relation holds between $\widehat{\mu}_{n+1}(k+2^{n+1}\ell)$ and $\widehat{\mu}_{n+1}(k)$ (and this holds also for $\widehat{\mu}_{n}$ in place of $\widehat{\mu}_{n+1}$). Note that we can write any frequency $k'$ with $\|k'\|\ge 2^{n+1}$ as $k'=k+2^{n+1}\ell$, where $\|k\|<2^{n+1}$, $k_j\ell_j\ge 0$ for all $j$ and $\ell\neq 0$. For such $k'$, we have
\begin{align*}
|\widehat{\mu}_{n+1}(k')-\widehat{\mu}_n(k')| &\le  \prod_{j:k_j+2^{n+1}\ell_j\neq 0}\frac{|k_j|}{|k_j+ 2^{n+1}\ell_j|}|\widehat{\mu}_{n+1}(k)-\widehat{\mu}_n(k)|\\
&< \frac{2^{n+1}}{\|k'\|}|\widehat{\mu}_{n+1}(k)-\widehat{\mu}_n(k)|\,,
\end{align*}
where we have bounded each factor by $1$, except the one for which $|k_j+2^{n+1}\ell_j|=\|k'\|$. Combining this with \eqref{eq:Salem-larve-dev}, we arrive at the following key fact: $\PP(E_n)<\e_n$, where $E_n$ is the event
\[
\left|\widehat{\mu}_{n+1}(k)-\widehat{\mu}_n(k)\right|>
\|\mu_n\|^{1/2}\min\left(1,\frac{2^{n+1}}{\|k\|} \right) 2^{-\sigma n/2}\text{ for some }k\in\Z^d\,.
\]
Since $\e_n$ was summable, and $\|\mu_n\|$ is a.s. bounded, it follows from the Borel-Cantelli Lemma that a.s. there are $C$ and $n_0$ such that
\[
\left|\widehat{\mu}_{n+1}(k)-\widehat{\mu}_n(k)\right|  \le C \min\left(1,\frac{2^{n+1}}{\|k\|} \right) 2^{-\sigma n/2}\quad\text{for all }k\in\Z^d,n\ge n_0\,.
\]
Thus, choosing $n_1\in\N$ such that $2^{n_1}\le\|k\|<2^{n_1+1}$, and telescoping, we have
\begin{equation}\label{eq:telescope}
\begin{split}
\left|\widehat{\mu}_m(k)-\widehat{\mu}_{n_0}(k)\right|&\le\sum_{n_0\le n\le n_1}2C2^{n(1-\sigma/2)}\|k\|^{-1} +\sum_{\max(n_1,n_0)<n\le m}C 2^{-n\sigma/2}\\
&=O(\|k\|^{-\sigma/2}))
\end{split}
\end{equation}
 for all $k\in\mathbb{Z},m\ge n_0$.

Noting that $\widehat{\mu}_{n_0}(k)=O_{n_0}(\|k\|^{-1})$ and letting $m\to\infty$ finishes the proof.
\end{proof}

\begin{rems}
\begin{enumerate}[(i)]
\item The above theorem is sharp: if $\alpha<d-2$, then $\dim_F(\mu_\infty)=2$ a.s. on $\mu_\infty\neq 0$ (so $\mu_\infty$ is \emph{not} a Salem measure). Indeed, one can see that $\widehat{\mu}_\infty(k) = O(\|k\|^{-1})$ from the estimate \eqref{eq:telescope}, where the first sum is of order $O_{n_0}(\|k\|^{-1})$ if $\sigma>2$ (otherwise the proof is the same). This shows that $\dim_F(\mu_\infty)\ge 2$.

On the other hand, the orthogonal projection $\nu$ of $\mu_\infty$ onto a fixed coordinate axis cannot be absolutely continuous with a continuous density, as the dyadic nature of the construction together with the independence assumption \eqref{SD:independent} imply that a.s. on $\mu_\infty\neq 0$ there are discontinuities at all dyadic points in $\supp\nu$. Since a measure $\nu$ on $\R$ with $|\widehat{\nu}(\xi)|=O(|\xi|^{-1-\delta})$ for $\delta>0$ has a H\"{o}lder continuous density (see \cite[Proposition 3.2.12]{Grafakos08}), we have shown that $\dim_F\mu_\infty \le 2$.
\item The result is valid also in slightly less homogeneous situations: Suppose $(\mu_n)\in\mathcal{SM}$ is defined via the dyadic filtration and that for each $Q\in\mathcal{Q}$, there is $p(Q)\in(0,1]$ such that $W_Q\in\{0,p(Q)^{-1}\}$.
For $Q\in\mathcal{Q}_n$, let $Q_1,\ldots Q_{n}$ be its dyadic ancestors and denote $\beta_n(Q)=\prod_{i=1}^n p(Q_i)^{-1}$. If
\[
\frac1n \log_2\beta_n(Q)\to\alpha\in [d-2,d]\,
\]
uniformly as $n\to\infty$, the proof of Theorem \ref{thm:Salem} with straightforward modifications still shows that a.s. on $\mu_\infty\neq 0$, $\dim_F\mu_\infty=\dim_H\mu_\infty=\dim_H\supp\mu_\infty=\dim_B\supp\mu_\infty=d-\alpha$.
\end{enumerate}
\end{rems}

\subsection{Application to the restriction problem for fractal measures}
\label{subsec:question-laba}

Recall that the restriction problem for a measure $\mu\in\mathcal{P}_d$ on Euclidean space asks for what pairs $p,q$ there is an estimate
\begin{equation} \label{eq:restriction}
\|\widehat{f d\mu}\|_{L^p(\R^d)}  \le O_{p,q}(1) \|f\|_{L^q(\mu)}\,.
\end{equation}
See \cite{Laba14} for background and the dual formulation (from which the name ``restriction problem'' originates). One often takes $q=2$ and looks for values of $p$ such that the above holds. The study of the restriction problem for fractal measures was pioneered by Mockenhaupt \cite{Mockenhaupt00} (see also \cite{Mitsis02}), who proved that if
\begin{align}
\mu(B(x,r))&=O(r^\alpha)\,, \label{eq:decay-mass}\\
|\widehat{\mu}(\xi)| &= O((1+|\xi|)^{-\beta/2}) \,,\label{eq:decay-FT}
\end{align}
then \eqref{eq:restriction} with $q=2$ holds whenever
\[
p >  p_{\alpha,\beta,d} = \frac{2(2d-2\alpha+\beta)}{\beta}\,.
\]
This was extended to the endpoint by Bak and Seeger in \cite{BakSeeger11}. It is well known that if \eqref{eq:decay-mass} and \eqref{eq:decay-FT} hold, then $\beta,\alpha\le \alpha_0$, where $\alpha_0$ is the Hausdorff dimension of the support of $\mu$, and the most interesting case is that in which $\beta,\alpha$ can be taken arbitrarily close to $\alpha_0$. Let us call measures with this property \textbf{strong Salem measures}. Recently, Hambrook and {\L}aba \cite{HambrookLaba13} established the sharpness of the range of $p$ for strong Salem measures, in the case $d=1$. (This was extended to general pairs $0\le\beta\le\alpha\le 1$ by Chen \cite{Chen13}.) However, this left open the problem of whether the restriction estimate \eqref{eq:restriction} holds for a range of $p<p_{\alpha,\beta,1}$ for \emph{some} strong Salem measures on the real line. This question was asked by {\L}aba \cite[Question 3.5]{Laba14}.

A different restriction theorem based on convolution powers was proved by Chen in \cite[Theorem 1]{Chen14}. In particular, it implies the following.
\begin{thm} \label{thm:chen}
If $\mu\in\mathcal{P}_d$ has the property that $\mu*\mu$ is absolutely continuous with a bounded density, then \eqref{eq:restriction} holds for $p\ge 4$ and $q\ge p/(p-2)$, and in particular for $p\ge 4$ and $q=2$.
\end{thm}
Previously, there were very few examples of fractal measures $\mu$ for which the self-convolution was known to have a bounded density. Recall that, for $d=1$,  K\"{o}rner \cite{Korner08} constructed such measures of any dimension in $[1/2,1)$; however,  it seems hard to understand other geometric properties such as \eqref{eq:decay-mass} and \eqref{eq:decay-FT} for those measures. In connection with his theorem, Chen \cite[Remark (iii) after Corollary 1]{Chen14} poses the problem of the existence of strong Salem measures whose self-convolution has a bounded density.

Theorem \ref{thm:self-convolutions} provides a rich class of measures for which the self-convolution has a continuous density. Combining this with Theorems \ref{thm:Salem} and \ref{thm:chen}, we are able to give a positive answer to the questions of {\L}aba and Chen, in the range $\alpha_0\in (1/2,2/3]$. Note that for $d=1$ the range of  the Mockenhaupt/Bak and Seeger Theorem in the strong Salem case is $p\in (4/\alpha_0-2,\infty)$, and $4/\alpha_0-2>4$ if and only if $\alpha_0<2/3$.

\begin{thm} \label{thm:restriction-application}
Let $\alpha_0\in (1/2,1)$. Then there is a strong Salem measure $\mu\in\mathcal{P}_1$ supported on a set of Hausdorff dimension $\alpha_0$, such that $\mu*\mu$ has a H\"{o}lder continuous density, and \eqref{eq:restriction} holds for all $p\ge 4$ and $q\ge p/(p-2)$.
\end{thm}
\begin{proof}
In light of Theorems \ref{thm:self-convolutions}, \ref{thm:Salem} and \ref{thm:chen}, all that is left to do is to exhibit limits of SI-martingales $\mu_\infty$ of arbitrary dimension $\alpha_0$ which satisfy the hypotheses of Theorem \ref{thm:Salem} (which automatically imply the conditions of Theorem \ref{thm:self-convolutions} provided $\alpha_0>1/2$), and for which $\mu_\infty(B(x,r)) = O(r^{\alpha_0})$. The construction we present here is closely related to the one form \cite{ShmerkinSuomala12}.

Let us then fix $\alpha_0$, and pick a sequence $N_j$ with $N_j\in\{1,2\}$ for all $j$ and such that $P_n:=N_1\cdots N_n=\Theta(2^{\alpha_0 n})$. Using this sequence we construct the following $\subdivision$ martingale: start with $\mu_0=\mathbf{1}_{[0,1)}$. Suppose that we have defined $\mu_n$ of the form $2^n P_n^{-1}\sum_{j=1}^{P_n} \mathbf{1}[I_j]$, where the $I_j$ are different half-open dyadic intervals of length $2^{-n}$. We iterate as follows: if $N_{n+1}=2$, then we split each $I_j$ into its two dyadic offspring $I'_j, I''_j$ and keep both of them (deterministically), and set
\[
\mu_{n+1} = \frac{2^{n+1}}{P_{n+1}} \sum_{j=1}^{P_n} (\mathbf{1}[I'_j]+\mathbf{1}[I''_j]) = \mu_n\,.
\]
Otherwise, if $N_{n+1}=1$, then choose one of the two dyadic offsprings of $I_j$ uniformly at random and call it $I'_j$, with all choices independent. Set
\[
\mu_{n+1} = \frac{2^{n+1}}{P_{n+1}} \sum_{j=1}^{P_{n+1}} \mathbf{1}[I'_j]\,.
\]
In both cases $\mu_{n+1}$ has the same form as $\mu_n$. It is clear that $(\mu_n)\in\subdivision$ and that it meets the assumptions of Theorem \ref{thm:Salem} with $\beta_n=2^n P_n^{-1}$. Note that $\log_2\beta_n/n\to 1-\alpha_0$. Since all chosen intervals $I$ of step $n$ have the same mass $P_n^{-1}=\Theta(2^{-\alpha_0 n})$ for all $n$, \eqref{eq:decay-mass} holds with $\alpha=\alpha_0$, and this completes the proof.
\end{proof}

\begin{rems}
\begin{enumerate}[(i)]
\item The measure constructed in the proof of Theorem \ref{thm:restriction-application} satisfies \eqref{eq:decay-mass} with $\alpha=\alpha_0$ while \eqref{eq:decay-FT} holds for all $\beta<\alpha_0$. Actually, it follows from the construction that
$\mu_\infty$ is Ahlfors $\alpha_0$-regular so that $\mu(B(x,r))=\Theta(r^\alpha_0)$ for $x\in\supp(\mu_\infty)$ and $0<r<1$. It is not known if there are Ahlfors $\alpha$-regular measures which satisfy \eqref{eq:decay-FT} with $\beta=\alpha$, see \cite{Mitsis02}, \cite[Problem 20]{Mattila04}.

\item With some additional work, it is possible to extend the above theorem to $\alpha_0=1/2$. The same construction works, with $P_n=\Theta(n^C 2^{-n/2})$ for a sufficiently large $C$. Although Theorem \ref{thm:self-convolutions} is not directly applicable, the $n^C$ factor allows the proof of the theorem to go through, except that the density is now uniformly continuous but not necessarily H\"{o}lder, and this is enough to show that $\mu*\mu$ has a continuous (and thus bounded) density. The value $\alpha_0=1/2$ is perhaps of special significance because the range $p\ge 4$ for $q=2$ is optimal in this case, see \cite{Laba14}.
\item For $\alpha_0\in (0,1/2]$, we can also construct strong Salem measures for which \eqref{eq:restriction} holds for $q=2$ and a range strictly larger than the one coming from Mockenhaupt's Theorem. However, this involves higher order convolutions and therefore higher order cartesian products, which are not SI-weak martingales. This will be addressed in \cite{ShmerkinSuomala14}.
\item Theorem \ref{thm:restriction-application} also works, with a nearly identical proof, in dimensions $d=2,3$ for the range $\alpha_0\in(d/2,2]$. However, the one-dimensional case is perhaps the most interesting, since surface area on the sphere $S^{d-1}\subset\R^d$ is a strong Salem measure for which the optimal range of $p$ for which \eqref{eq:restriction} holds (with $q=2$) follows from the Stein-Tomas restriction theorem (the well-known restriction conjecture concerns finding the optimal range of $q$ in this case). This is due to the curvature of the sphere, which is a meaningless concept for subsets of $\R$. Our result therefore highlights the analogy between ``curvature'' and ``randomness'' (see \cite{Laba14} for a detailed discussion of this analogy). The construction \cite{HambrookLaba13} contains many highly structured subsets which act as an obstruction to a restriction estimate if $p$ is small enough while limits of SI-martingales cannot contain any patterns of size larger than their dimension allows (this will also be explored in \cite{ShmerkinSuomala14}).
\end{enumerate}
\end{rems}

\begin{rem}
Immediately after an earlier version of this paper was posted to the ArXiv, X. Chen informed us that in joint work in progress with A. Seeger they construct random measures which also satisfy Theorem \ref{thm:restriction-application}. Their construction is different from ours  and is based on Korner's from \cite{Korner08}, and some features are also different. For example, our example is Ahlfors-regular while theirs is not, but they get the sharp H\"{o}lder exponent and their construction works in any dimension. On the other hand, they do not address the simultaneous convolutions $\mu*S\mu$ (recall Theorem \ref{thm:self-convolutions}). We thank X. Chen for sharing an early manuscript of their work with us.
\end{rem}

\bibliographystyle{plain}
\bibliography{randomintersections}

\end{document}